\documentclass{amsart}
\usepackage[english]{babel}
\usepackage[dvipsnames]{xcolor}
\usepackage{graphicx,pifont}
\usepackage[utopia]{mathdesign}
\usepackage[a4paper,top=3cm,bottom=3cm,
           left=3cm,right=3cm,marginparwidth=60pt]{geometry}
\usepackage[utf8]{inputenc}
\usepackage{braket,caption,comment,mathtools,stmaryrd}
\usepackage[usestackEOL]{stackengine}
\usepackage{multirow,booktabs,microtype,relsize}
\usepackage{enumerate}
\usepackage{subfigure}

\usepackage[colorlinks,bookmarks]{hyperref} %

      \hypersetup{colorlinks,%
            citecolor=britishracinggreen,%
            filecolor=black,%
            linkcolor=cobalt,%
            urlcolor=cornellred}
      \numberwithin{equation}{section}
\usepackage[capitalise]{cleveref}
\DeclareSymbolFont{usualmathcal}{OMS}{cmsy}{m}{n}
\DeclareSymbolFontAlphabet{\mathcal}{usualmathcal}
\DeclareMathAlphabet\BCal{OMS}{cmsy}{b}{n}

\makeatletter
\newcommand{\mylabel}[2]{#2\def\@currentlabel{#2}\label{#1}}
\makeatother

\definecolor{cornellred}{rgb}{0.7, 0.11, 0.11}
\definecolor{britishracinggreen}{rgb}{0.0, 0.26, 0.15}
\definecolor{cobalt}{rgb}{0.0, 0.28, 0.67}

\usepackage[colorinlistoftodos]{todonotes} 
\usepackage{amsmath,float,soul}
\usetikzlibrary{decorations.markings}

\newcommand{\rrr}{{\mathsf{r}}}
\newcommand{\Gr}{\mathrm{Gr}}

\newcommand{\In}{\mathrm{In}}

\newcommand{\Span}{\mathrm{Span}}

\newcommand{\OO}{\mathscr O}
\newcommand{\OZ}{\mathscr Z}
\newcommand{\OH}{\mathscr H}
\newcommand{\OI}{\mathscr I}

\newcommand{\bH}{\boldsymbol{H}}
\newcommand{\be}{\boldsymbol{e}}

\DeclareMathOperator{\Hilb}{Hilb}

\DeclareMathOperator{\GL}{GL}
\DeclareMathOperator{\colen}{colen}

\DeclareMathOperator{\Hess}{Hess}

\DeclareMathOperator{\Hom}{Hom}

\DeclareMathOperator{\Spec}{Spec}

\DeclareMathOperator{\Sets}{Sets}

\DeclareMathOperator{\red}{red}
\DeclareMathOperator{\length}{len}

\DeclareMathOperator{\Sch}{Sch}
\DeclareMathOperator{\opp}{op}
\DeclareMathOperator{\Supp}{Supp}

\DeclareMathOperator{\reg}{reg}

\newcommand{\BA}{{\mathbb{A}}}

\newcommand{\BC}{{\mathbb{C}}}

\newcommand{\BG}{{\mathbb{G}}}

\newcommand{\BN}{{\mathbb{N}}}

\newcommand{\BQ}{{\mathbb{Q}}}

\newcommand{\BZ}{{\mathbb{Z}}}

\newcommand{\Fm}{{\mathfrak{m}}}

\newcommand{\hilbert}[2]{{\Hilb^{#1}{#2}}}

\usepackage[all]{xy}
\usepackage{tikz}
\usepackage{tikz-cd}
\usepackage{adjustbox}
\usepackage{rotating}
\usepackage{comment}

\usetikzlibrary{matrix,shapes,intersections,arrows,decorations.pathmorphing}
\tikzset{commutative diagrams/arrow style=math font}
\tikzset{commutative diagrams/.cd,
mysymbol/.style={start anchor=center,end anchor=center,draw=none}}

\tikzset{
shift up/.style={
to path={([yshift=#1]\tikztostart.east) -- ([yshift=#1]\tikztotarget.west) \tikztonodes}
}
}

\usepackage{youngtab} 

\usepackage{ytableau} 
  \usetikzlibrary {decorations.pathreplacing}

\theoremstyle{definition}

\newtheorem*{lemma*}{Lemma}
\newtheorem*{theorem*}{Theorem}
\newtheorem*{example*}{Example}
\newtheorem*{fact*}{Fact}
\newtheorem*{notation*}{Notation}
\newtheorem*{definition*}{Definition}
\newtheorem*{prop*}{Proposition}
\newtheorem*{remark*}{Remark}
\newtheorem*{corollary*}{Corollary}

\newtheorem*{conventions*}{Conventions}

\newtheorem{definition}{Definition}[section]

\newtheorem{example}[definition]{Example}

\newtheorem{notation}[definition]{Notation}

\newtheorem{remark}[definition]{Remark}

\newtheoremstyle{thm} 
        {3mm}
        {3mm}
        {\slshape}
        {0mm}
        {\bfseries}
        {.}
        {1mm}
        {}
\theoremstyle{thm}
\newtheorem{theorem}[definition]{Theorem}
\newtheorem{corollary}[definition]{Corollary}
\newtheorem{lemma}[definition]{Lemma}
\newtheorem{prop}[definition]{Proposition}

\newtheorem{thm}{Theorem}

\newtheorem*{Acknowledgments*}{Acknowledgments}
\crefname{notation}{Notation}{Notations}

\usepackage[normalem]{ulem}

\usetikzlibrary {patterns,patterns.meta}

\newcommand{\bh}{\textnormal{\textbf{h}}}
\newcommand{\sss}{{\mathsf{s}}}
\newcommand{\ttt}{{\mathsf{t}}}
\newcommand{\oh}{{\mathsf{h}}}
\newcommand{\uh}{{\mathsf{q}}}

\title[New components of Hilbert schemes and 2-step ideals]{New components of Hilbert schemes of points and 2-step ideals}

\author[F. Giovenzana]{Franco Giovenzana}
\address[F. Giovenzana]{Laboratoire de Math\'ematiques d’Orsay\\ Universit\'e Paris-Saclay\\Rue Michel Magat, B\^at. 307, 91405\\ Orsay, France}
\email{franco.giovenzana@universite-paris-saclay.fr}

\author[L. Giovenzana]{Luca Giovenzana}
\address[L. Giovenzana]{Department of Pure Mathematics\\ University of Sheffield\\ Hicks Building, Hounsfield Road\\ Sheffield, S3 7RH,\\ UK}
\email{l.giovenzana@sheffield.ac.uk}

\author[M. Graffeo]{Michele Graffeo} 
\address[M. Graffeo]{SISSA\\ Via Bonomea 265\\ 34136 Trieste\\ Italy}
\email{mgraffeo@sissa.it}

\author[P. Lella]{Paolo Lella} 
\address[P. Lella]{Dipartimento di Matematica\\ Politecnico di Milano\\ Piazza Leonardo da Vinci 32\\ 20133 Milan\\ Italy}
\email{paolo.lella@polimi.it}
 

\thanks{{\em Funding.} The first author was funded by Deutsche Forschungsgemeinschaft (DFG, German Research Foundation) Projektnummer 509501007, and partially supported by the European Research Council (ERC) under the European Union’s Horizon 2020 research and innovation programme (ERC-2020- SyG-854361- HyperK). The fourth author was supported by the grant PRIN 2022K48YYP {\em Unirationality, Hilbert schemes, and singularities}.  All the authors are members of the GNSAGA - INdAM}

\begin{document}
\begin{abstract}
This paper presents new examples of elementary and non-elementary irreducible components of the Hilbert scheme of points and its nested variants. The results are achieved via a careful analysis of the deformations of a class of finite colength ideals that are introduced in this paper and referred to as 2-step ideals. The most notable reducibility results pertain to the 4-nested Hilbert scheme of points on a smooth surface, the reducibility of $\Hilb^{3,7}\BA^4$, and a method to detect a large number of generically reduced elementary components. To demonstrate the feasibility of this approach, we provide an explicit description of 215 new generically reduced elementary components in dimensions 4, 5 and 6. 
\end{abstract}

\maketitle
{\hypersetup{linkcolor=black}\tableofcontents} 

\section{Introduction}

 Moduli spaces of sheaves are among the objects that most interest algebraic geometers. One of the most classical, namely Hilbert schemes and more generally Quot schemes were introduced by Grothendieck in \cite{Grothendieck_Quot}, and  have recently received a lot of interest due to their connections with, and applications in, other areas of research such as Enumerative Geometry and Theoretical Physics \cite{Szabo2023,Nekrasov,Kanno}. In the present paper, we are interested in the nested Hilbert scheme of points on a smooth and connected quasi-projective variety $X$ of dimension $\dim X=n$, i.e.~the scheme locally of finite type 
\[
\Hilb^\bullet X
\]
representing  the nested Hilbert functor of points on $X$, i.e.~the functor associating to each base scheme $B$ the set of finite sequences of families
\begin{equation}\label{eq:nestintro}
    \OZ^{(1)}\subset\cdots \subset\OZ^{(r)}\subset X\times B
\end{equation}
of closed $B$-flat and $B$-finite subschemes of $X\times B$.
\subsection{The geometry of the Hilbert scheme of points}
By results of Fogarty and Cheah \cite{FOGARTY,CHEACELLULAR}, the connected components of $\Hilb^\bullet X$ are cut out by imposing conditions on the $B$-length of each $\OZ^{(i)}$, for $i=1,\ldots,r$. Explicitly, given a non-decreasing sequence of positive integers $\underline{d}=(0< d_1\leqslant\cdots \leqslant d_r)\in\BZ^r$, the $\underline{d}$-nested Hilbert scheme $\Hilb^{\underline{d}} X$ is the connected component of $\Hilb^\bullet X$ whose $B$-points correspond to nestings of the form \eqref{eq:nestintro} with  $\length_B\OZ^{(i)}=d_i$, for $i=1,\ldots,r$.  

The schemes $\Hilb^{\underline{d}} X$ are  in general  wild and their geometry is nowadays quite inaccessible, see \cite{JJ-Hilb-open-problems,farkas2024irrational,UPDATES,double-nested-1,MOTIVES,GGGL}. They are singular in the following cases:
   \begin{itemize}
   \item[($r=1$)] $n \geqslant 3$ and $d_1 \geqslant 4$;
   \item[($r=2$)] $n = 2$ and $d_2 - d_1 > 1$ or $n \geqslant 3$ and $(d_1,d_2) \notin \{(1,2),(2,3)\}$;
   \item[($r \geqslant 3$)] $n \geqslant 2$.  
    \end{itemize}

 Moreover, they have generically non-reduced irreducible components already for $r=1$ as soon as $n\geqslant 4$ and $d\geqslant 21$, or $n\geqslant 6$ and $d\geqslant 13$, see \cite{13PUNTI,jelisiejewnonred}. On the other hand, we show in \Cref{thm: intro A} that, for   $r \geqslant 5$, this phenomenon already occurs in dimension $n = 2$, see \cite{UPDATES}. 
  
Although the closure of the open locus parametrising nestings of reduced subschemes of $X$ always defines a distinguished component $\Hilb_{\mathrm{sm}}^{\underline{d}}X\subset \Hilb^{\underline{d}}X$ of dimension $n\cdot d_r$, named the \textit{smoothable component}, the problem of detecting all its irreducible components remains one of the biggest challenges in the field. The aim of this paper is to attack this problem and provide new examples of reducible Hilbert schemes that cannot be obtained from existing constructions in the literature. To this end, we introduce a new class of ideals suitable for our purpose, which we name \textit{2-step ideals}, see \Cref{def:2step}.
The main idea relies on Iarrobino's observation that if an algebra is \textit{"large"} enough, then the locus parametrising similar algebras has dimension higher than the dimension of the smoothable component, and this ensures reducibility, see \cite{IARRO}. In Iarrobino's work, the notion of \textit{"large"} was incarnated by compressedness; in the present paper, this is replaced by the property of being a  2-step ideal. 

On the other hand, a possible way to certify the existence of (generically reduced) elementary components is to find a point having \textit{Trivial Negative Tangent} (TNT), see \Cref{def:TNT} and \cite{ELEMENTARY}. These components are considered the building blocks of the Hilbert scheme of points as any other irreducible component can be recovered from their knowledge. From this perspective, 2-step nestings having linear syzygies behave particularly well. Indeed, linear syzygies prevent the presence of tangents of degree strictly smaller than minus one, which is a necessary condition for TNT, see \Cref{thm:tnt per nested}.

Thanks to our method, it is possible to prove the existence of a huge number of elementary irreducible components. As a proof of concept, we present many of them, thus answering some open questions in the subject. 

When $X$ is a curve, the $\underline{d}$-nested Hilbert scheme is irreducible and smooth independently of $r\geqslant 1$. For $n=2$, the situation gets more complicated. Indeed, although the scheme $\Hilb^{\underline{d}} X$ is  irreducible for $r\leqslant 2$, see \cite{Rasul-irr-nested}, the minimum value of $r$ for which the reducibility of $\Hilb^{\underline d}X$ is known for some $\underline{d}\in\BZ^r$ is 5. This was shown in \cite{ALESSIONESTED}, where
 the authors prove that $\Hilb^{\underline{d}}X$ is reducible for $\underline{d}=(380, 420, 462, 506, 552) $. The following result improves upon this by reducing $r$ to 4 or significantly reducing  the involved lengths, and furthermore, it provides   new examples. 
 
\begin{thm}[\Cref{thm:nested-srf} and \Cref{cor:5nestnonred}] \label{thm: intro A} 
    If $\underline{d}$ is one of the following increasing sequences of positive integers
    \begin{enumerate}[\rm (a)]
        \item $\underline{d} = (454,491,527,565) \in \BZ^4$,
        \item $\underline{d} = (51,64,76,87,102) \in \BZ^5$
        \item $\underline{d} = (21,30,38,45,51,61)\in \BZ^6$,
        \item $\underline{d} = (11,18,24,29,33,40,50)\in \BZ^7$,
        \item $\underline{d} = (3,8,12,18,24,29,34,43)\in \BZ^8$,
    \end{enumerate}
    then the nested Hilbert scheme $\hilbert{\underline{d}}{\BA^2}$ is reducible. Moreover,  the nested Hilbert scheme $\hilbert{1,\underline{d}}{\BA^2}$ has at least one generically non-reduced component.
\end{thm}

As a consequence of \Cref{thm: intro A}, there are reducible 4-nested Hilbert schemes of points on smooth surfaces, and the  question about the irreducibility of $\hilbert{\underline{d}}{\BA^2}$  remains open only for $r=3$, see \cite{Rasul-irr-nested}. 

Irreducibility in the case $n=3$ is the least understood already for $r=1$. Indeed, the classical Hilbert scheme of points on a smooth threefold is known to be irreducible for $d_1\leqslant 11$, cf. \cite{Klemen,10points,JOACHIM} and reducible for $d_1\geqslant 78$, cf. \cite{IARRO}.  

In \Cref{sec:red3fold}, we recover Iarrobino's result about the reducibility of $\Hilb^{78}\BA^3$ in terms of the new class of ideals we present, namely 2-step ideals, c.f. \Cref{ubsec:2stp intro,sec:2-step}. Moreover, thanks to this notion, we find many examples of families of non-smoothable zero-dimensional algebras of embedding dimension 3. Also, in dimension 3, we show that the $\underline{d}$-nested Hilbert scheme is reducible for $d_r$ much smaller than 78 already for $r=2,3$.

\begin{thm}[\Cref{thm:nested-threefold}] \label{thm: intro 3nest}  
    If $\underline{d}$ is one of the following increasing sequences of positive integers
    \begin{enumerate}[\rm (a)]
        \item $\underline{d} \in \{ (14,24), (15,24), (13,26) \} \subset \BZ^2$,
        \item $\underline{d} \in \left\{\begin{array}{c}(7,13,17), (7,12,18), (6,13,18),  (8,13,18), (6,12,20),(8,12,20), (5,13,20), \\ (5,14,20), (4,13,21),(3,14,21),(4,14,21), (6,11,22),(7,11,22),(3,13,22),\\ (4,12,23),(5,12,23),(2,14,23), (2,15,23), (3,12,24), (2,13,24), (2,12,25) \end{array}\right\} \subset \BZ^3$
    \end{enumerate}
    then the nested Hilbert scheme $\hilbert{\underline{d}}{\BA^3}$ is reducible.
\end{thm}

In higher dimension, the classical Hilbert scheme is irreducible if and only if $d\leqslant 7$, see \cite{Iarrob,MAZZOLA,8POINTS}. After having provided many examples of elementary components of $\Hilb^{\bullet}\BA^4$, we focus on the nested case and in \Cref{thm:intro B}, we show that for $r>1$ they arise very soon.
\begin{thm}[{\Cref{thm:redim4nest}}]\label{thm:intro B} 
        The nested Hilbert scheme $\hilbert{(3,7)}{\BA^4}$ has a generically reduced elementary  component $V$. 
        Moreover, we have an isomorphism
    \[
    \begin{tikzcd}
        (V)_{\red}\cong   \Gr(2,4) \times \Gr(2,10)\times \BA^4
        .
    \end{tikzcd}
    \]
     As a consequence, the nested Hilbert scheme $\hilbert{(1,3,7)}{\BA^4}$ has a generically non-reduced elementary component $V_1$ such that $(V_1)_{\red}=(V)_{\red}$.
\end{thm}

To conclude this subsection, we mention that in \Cref{sec:small} we give 181 examples of elementary components of $\Hilb^{\bullet}\BA^n$, for $n=5,6$. The connected component of $\Hilb^{\bullet}\BA^6$ for which we are able to find the largest number of generically reduced elementary components is $\Hilb^{34}\BA^6$.


\begin{thm}[\Cref{thm:many components}] \label{thm: intro tantecomp} The Hilbert scheme $\hilbert{34}{\BA^6}$ has at least 12 generically reduced elementary components.
\end{thm}  

In the search for elementary components, the \textit{potential TNT area} is a particularly important object. This is defined in \Cref{subsec:TNT2step}, see \Cref{def:TNTarea}. It is a subset of the natural plane, the complement of which consists of points corresponding to 2-step ideals that cannot lie on a generically reduced elementary component, i.e. that do not have the TNT property, see \Cref{def:TNT}.

\subsection{The class of 2-step ideals}\label{ubsec:2stp intro}

In \Cref{sec:2-step} we introduce the main object of our study, the class of 2-step ideals. These ideals are defined by the condition of being  sandwiched  in between two  powers of distance two of the maximal ideal $\Fm\subset R = \BC[x_1,\ldots,x_n]$ generated by the variables. In symbols, an $\Fm$-primary ideal $I$ is 2-step if  
\[
\mathfrak{m}^{k+2} \subset I \subset \mathfrak{m}^{k}\quad\mbox{ and }\quad I\not \subset \mathfrak{m}^{k+1},
\] 
for some positive integer $k>0$, which we call the order of $I$. We focus on this class of ideals because, as we show in the paper, the loci parametrising 2-step ideals have  very large dimension. So, they often do not fit in the smoothable component, thus certifying the existence of exceeding components of the Hilbert scheme.

Our first result on nestings of 2-step ideals concerns the $\BG_m$-equivariant decomposition
\[
\mathsf{T}_{[\underline{I}]}\Hilb^\bullet \BA^n=\bigoplus_{j\in\BZ}\mathsf{T}^{=j}_{[\underline{I}]}\Hilb^\bullet \BA^n,
\]
 of the tangent space  at a $\BG_m$-fixed point $[\underline I]\in\Hilb^\bullet \BA^n$, where  the torus $\BG_m$ acts on $\BA^n$ via the scalar action and we can denote points of the Hilbert scheme using ideals in virtue of the correspndence between ideals of $I$ and closed subschemes of $\BA^n$.

\begin{thm}[{\Cref{cor:1nonob-nested}}]\label{thmintro:tgnested} 
    Let  $[\underline I ]{=[(I^{(i)})_{i=1}^r]}\in \Hilb^\bullet \BA^n$ be a nesting of 2-step homogeneous ideals. Denote by $k_i>0$ the order of the ideal $I^{(i)}$, for $i=1,\ldots,r$. Suppose that $k_{i+1}-k_i>0$, for all $i=1,\ldots,r-1$.  Then,  there is an isomorphism  
\[
\mathsf{T}_{[\underline I]}^{>0} \Hilb^\bullet \mathbb{A}^n =\mathsf{T}_{[\underline I]}^{=1} \Hilb^\bullet \mathbb{A}^n \cong \bigoplus_{i=1}^r \Hom_{R}\big(I^{(i)},R/I^{(i)}\big)_1.
\]
Moreover, all the tangent vectors of degree one are unobstructed. 
\end{thm}

Since we only address local questions, working over $\BA^n$ is not restrictive at all for our purpose. Thanks to \Cref{thmintro:tgnested}, we are able to compute the dimension of some loci parametrising 2-step ideals by considering the Bia{\l{}}ynicki--Birula decomposition as presented in \cite{ELEMENTARY}, see also \cite{UPDATES}. 
We distinguish 2-step ideals according to the rank  of their module of linear syzygies. To any homogeneous 2-step ideal $I$ we can attach its Hilbert function $\bh_I$, a discrete invariant that refines the colength $\colen I=\dim_{\BC}R/I $, see \Cref{subsec:HSfun}. In the 2-step framework, this invariant is equivalent to the pair 
\[
\left(\oh_k,\oh_{k+1}\right)=\left(\dim_{\BC} I_k,\dim_{\BC} I_{k+1}\right).
\]
With this terminology, the presence of linear syzygies is predicted as explained in \Cref{explanation-syzygies} 
 by the sign of the integer \[
\sss_{\bh}=\oh_{k+1}-n\oh_{k}.
\]
Given a sequence $\underline{\bh}=(\bh^{(1)},\ldots,\bh^{(r)})$ of Hilbert functions, we focus on the dimension of the locally closed subset $H^n_{\underline{\bh}}\subset \Hilb^{\bullet}\BA^n$ parametrising nestings $I^{(1)}\supset \cdots \supset I^{(r)}$ whose sequence of respective Hilbert functions agrees with $\underline{\bh}$. 

In order to certify the dimension of some stratum $H^n_{\underline\bh}$ we compute the dimension of the homogeneous locus $\OH^n_{\underline\bh}$ and  the dimension of the generic fibre\footnote{In general $\OH^n_{\underline\bh}$ is not irreducible.} of the initial ideal morphism $H^n_{\underline\bh}\to\OH^n_{\underline\bh}$. We do this in \Cref{sec:2-step} for the special case of nestings of 2-step Hilbert functions, in which case we show that the tangent space to $H^n_{\underline\bh}$ at a $\BG_m$-fixed point is concentrated in degree 0 and 1 and it consists of unobstructed tangent vectors, \Cref{rem:homobb}. Since we want a lower bound for the dimension of these strata we consider the open subset corresponding to \textit{nestings of homogeneous ideals having natural first anti-diagonal of the Betti table}, i.e.~ideals having no linear syzygies among the degree $k$ generators if $s_{\bh}\geqslant 0$ or ideals generated only in degree $k$ and $k+2$ if $s_{\bh}<0 $, see \Cref{def:regularity}.  We assume the existence of ideals having natural first anti-diagonal as it turns out to be the technical tool providing the useful lower bound we are looking for.
 
\begin{thm}[\Cref{cor:formula dim nested,cor:nestveryfew}]\label{thmintro: D}  Let $\underline{\bh} = (\bh^{(i)})_{i=0}^{r-1}$ be a $r$-tuple of 2-step Hilbert functions of respective order $k,\ldots,k+r-1$. Assume that there exists at least a nesting of homogeneous ideals having natural first anti-diagonal of the Betti table. Suppose that one of the following two conditions holds
\begin{itemize}
    \item $\sss_{\bh^{(i)}}  \geqslant 0$, for all $i=0,\ldots,r-1$, or
\item  $\begin{cases}
\oh_{k+1}^{(0)} \geqslant (n-\frac{1}{n}) \oh_{k}^{(0)},\\
    \oh_{k+1+i}^{(i)} \geqslant \left( \max\left\{ n - \frac{1}{n}, n - \frac{\oh_{k+i}^{(i-1)}}{\rrr_{k+i+1}}\right\}\right) \oh_{k+i}^{(i)},&\mbox{for all } i=1,\ldots,r-1.
\end{cases}$
\end{itemize}

 Then, we have
\[
\dim H_{\underline{\bh}}^n \geqslant 
 \oh_{k}^{(0)}\left(\rrr_{k}-\oh_{k}^{(0)}\right) + \sum_{i=1}^{r-1} \oh_{k+i}^{(i)}\left(\oh_{k+i}^{(i-1)}-\oh_{k+i}^{(i)}\right) + \sum_{i=0}^{r-1}\left(\oh_{k+i+1}^{(i)} - (n-1)\oh_{k+i}^{(i)}\right)\left(\rrr_{k+i+1}-\oh_{k+i+1}^{(i)}\right), 
\]
 where $ \rrr_k=\dim R_k$.
\end{thm} 

It is worth mentioning that the requirements in \Cref{cor:nestveryfew} are slightly different and imply those given in the present introduction.  For the sake of readability, however, we present  a more concise statement here and defer to \Cref{subsec:2-stepfew,subsec:nest2step} for the technicalities.

Applying \Cref{thmintro: D}, we introduce functions $\Delta_{n,r,k}:\BN^{2r}\to \BQ$, indexed by the triple: \textit{"dimension, length of the nesting, and order"} that measure how much the dimension of a Hilbert stratum corresponding to a 2-step Hilbert function is expected to exceed the dimension of the smoothable component. Precisely, the  $r$-vectors $\underline{\bh}$ of 2-step Hilbert functions, of respective orders $k,\ldots,k+r-1$, for which $H^n_{\underline{\bh}}$ has dimension greater than or equal to that of the smoothable component can be identified by the sign of $\Delta_{n,r,k}$.

These functions are a key tool for most of the results in \Cref{sec:dim2,sec:red3fold}. Indeed, we are able to determine a large number of examples of \lq\lq big\rq\rq~locally closed subsets of the Hilbert scheme parametrising non-smoothable nestings of closed zero-dimensional subschemes of the affine space just studying the behaviour of a quadratic function.  To mention one result, we are able to recover the reducibility result of $\Hilb^{78}\BA^3$ proven  by Iarrobino in terms of 2-step ideals, cf.~\cite{IARRO}. Nevertheless, Iarrobino considers compressed ideals and, as we show in this paper, not all 2-step ideals are compressed. Therefore, most of the examples we present were not yet known in the literature.

We would like to emphasise that the irreducible components presented in this paper constitute only a small proportion of those that can be generated using our method. For this reason, the paper comes with the \textit{Macaulay2} package \href{www.paololella.it/software/TwoStepIdeals.m2}{\tt TwoStepIdeals.m2} and two ancillary files referenced in \Cref{sec:dim2,sec:red3fold,subsec:smallcomp,sec:small} that can be used to construct many more examples.

\subsection{Organisation of the content}
In \Cref{sec:preli}, we give the basic tools to develop our theory. Precisely, we recall the notion of Hilbert function and Betti table in \Cref{subsec:HSfun} and then we move to a more geometrical setting by introducing the main object of our study, namely the nested Hilbert scheme  of points in \Cref{subsec:prlinested} and the stratification of the punctual locus coming from the Bia{\l{}}ynicki--Birula decomposition in \Cref{subsec:BB,subsec:HSstrat}.

In \Cref{sec:2-step}, we define and study 2-step ideals. First, in \Cref{subsec:2-stepdef} we define them and we give bounds for the dimension of the $\mathbb G_m$-equivariant parts of the tangent space to the Hilbert scheme at points corresponding to homogeneous 2-step ideals. We also prove unobstructedness of positive tangent vectors. In \Cref{subsec:nest2stepbase}, we perform the same computations for nesting of 2-step ideals and we prove \Cref{thmintro:tgnested}. Then, in \Cref{subsec:2-stepnosyz,subsec:2-stepfew} we consider two sub-classes of 2-step ideals defined in terms of their first syzygy module and we prove \Cref{thmintro: D} and we define the functions $\Delta_{n,r,k}$, for $n \geqslant 2,r\geqslant 1,k\geqslant 1$, so computing the dimension of the loci parametrising them. Then, in \Cref{subsec:TNT2step} we define the potential TNT area and investigate the TNT property for 2-step ideals. Finally in \Cref{subsec:nest2step} we study the nested case.

In the remaining sections, we apply our results to show the existence of yet unknown irreducible components of the Hilbert scheme of points. In \Cref{sec:dim2} we present the surface case by proving \Cref{thm: intro A}. Then, in \Cref{sec:red3fold} we focus on smooth threefolds. First, we recover the result by Iarrobino about $\Hilb^{78}\BA^3$ in terms of 2-step ideals and we prove the existence of many non-smoothable 2-step ideals of embedding dimension 3 and order 6, 7 and 8. Then, in \Cref{subsec:3dimnest} we prove \Cref{thm: intro 3nest} concerning the nested setting. We treat dimension four in \cref{subsec:smallcomp}. In this setting, we are able to certificate the existence of  new generically reduced elementary components. \Cref{thm:intro B} is then proven in \Cref{subsec:dim4nest}.
\Cref{sec:small} is devoted to the presentation of the generically reduced elementary components of $\Hilb^{\bullet}\BA^n$, for $n=5,6$. In this section, we also highlight the existence of many loci parametrising non-smoothable 2-step ideals and we prove \Cref{thm: intro tantecomp}.

Finally Appendix \ref{legenda} consists of a legend of the notation adopted in the figures of the paper.
 
\subsection*{Acknowledgments} We thank Barbara Fantechi, Anthony Iarrobino, Roberto Notari, Ritvik Ramkumar, Matthew Satriano and Klemen \v{S}ivic for interesting conversations. We thank Sergej Monavari and Andrea Ricolfi for very useful discussions and for their precious suggestions regarding the presentation. We thank Joachim Jelisiejew and Alessio Sammartano for the support they gave to the project from the very first moments. Special thanks to the organisers of the conference \textit{“AGATES-Deformation theory workshop”} at IMPAN (Warsaw) where this project was born.  The third author thanks Enrico Arbarello for introducing him to the subject. 


\section{Preliminary material}\label{sec:preli}

\begin{notation}\label{notation:initial}
    Let $R = \BC[x_1,\ldots,x_n]$ be the polynomial ring in $n$ variables with complex coefficients and let $\mathfrak{m} = (x_1,\ldots,x_n)$ be the maximal ideal. Note that we omit the dependence on $n$ as we will take care to not create confusion later in the paper. We endow the polynomial ring $R $ with the standard grading, i.e.~$\deg(x_i)=1$, for $i=1,\ldots,n$. The $k$-th graded piece of $R$ will be denoted $R_k$. 
Similarly, given a  homogeneous ideal $I\subset R$, we denote by $I_k$ and $(R/I)_k$ the $k$-th graded piece of the ideal and the quotient respectively. Finally, we denote by $\rrr_{k}$ the dimension of the vector space $R_k$, i.e.
\[
\rrr_{k}=\dim_{\BC} R_k=\binom{k+n-1}{n-1}.
\]

Whenever not specified, a $\BC$-algebra $A$ will always be of finite type and an $A$-module $M$ will always be finitely generated.
\end{notation}
\subsection{The Hilbert function and the Betti table}\label{subsec:HSfun} In this subsection, we recall some basic invariants attached to a graded $R$-module of finite type.

\begin{definition}
Let $A=\bigoplus_{k\in\BZ} A_k$ be a graded $\BC$-algebra and let $M=\bigoplus_{t\in\BZ} M_t$ be a graded $A$-module. The \emph{Hilbert function} $\bh_M$ associated to $M$ is the function
\[
\begin{array}{rrcl}
\bh_M: & \mathbb{Z} & \to & \mathbb{N} \\ & t & \mapsto & \dim_{\BC} M_t.
\end{array}
\]
 Let $(A,\Fm_A)$ be a local Artinian $\BC$-algebra. The \emph{Hilbert function} $\bh_A$ of $A$ is defined to be the Hilbert function of its associated graded algebra  
\begin{equation}\label{eqn:grading-A-module}
\mathsf{gr}_{\Fm_A}(A) =\bigoplus_{t\geq 0}\, \mathfrak m_A^{t}/\mathfrak m_A^{t+1},
\end{equation}
where $\mathsf{gr}_{\Fm_A}(A)$ is seen as a graded module over itself.  
\end{definition}

Since some notational ambiguity is sometimes present in the literature, we recall now the definition of initial ideal.
\begin{definition}\label{def:initialideal}
    Consider  an element $f\in R$ and write it as a sum of homogeneous pieces $f=f_m+f_{m+1}+\cdots+f_{\deg(f)}$, where $f_i\in R_i$ and $f_m\not=0$. Then, the initial form of $f$ is the homogeneous polynomial $\In f= f_{m}$. Moreover, if $I\subset R $ is any ideal, its initial ideal is $
    \In I =\left(\Set{ \In f | f\in I}\right)$.
\end{definition}

\begin{remark}\label{rmk:gradini}
    When the   $\BC$-algebra $(R/I,\Fm/I)$ is local, there is an isomorphism of graded algebras
    $\mathsf{gr}_{\Fm_A}(A) \cong R/(\In I)$,
     see \cite[\S 5.4]{EISENBUD}.
\end{remark}

\begin{notation}\label{notation:HStutti}
Whenever no confusion is possible, given a homogeneous ideal $I\subset R$, we will write $\oh_{k}$ for the value of the Hilbert function $\bh_I(k)$ of the ideal $I$ and $\uh_{k}$ for the value of the Hilbert function $\bh_{R/I}(k)$ of the quotient $R/I$. To have a compact notation, sometimes we encode the Hilbert function of a graded module $M$ in the so-called Hilbert series $\bH_M(T)=\sum_{t\in \BZ}\bh_M(t)T^t$.
\end{notation}
 
We recall now the definition and the main properties of the graded Betti numbers, see \cite{SYZYGIES} for more details.  Recall that any finitely generated graded $R$-module $M$ admits a minimal graded free resolution, i.e.~an exact sequence  $\begin{tikzcd}
0& M \arrow[l]&  F_\bullet \arrow[l],    
\end{tikzcd}$
where
\[
\begin{tikzcd}
F_\bullet:&\cdots& F_{i-1} \arrow[l]&  F_{i}\arrow[l,"{\delta_i}"']&\cdots\arrow[l],    
\end{tikzcd}
\]
and each $F_i$ can be written as 
\begin{equation}\label{eq:bettinumber}
    F_i=\bigoplus_j R(-j)^{\oplus\beta_{i,j}(M)},
\end{equation}
and such that $\delta _i(F_i)\subset \Fm F_{i-1}$. Moreover, a resolution with these properties is unique up to canonical, \cite[Section 20.1]{EISENBUD}.
\begin{definition}\label{def:regularity}
The natural numbers $\beta_{i,j}(M)$ in \eqref{eq:bettinumber} are the graded Betti numbers of the module $M$. Usually, they are arranged in the so-called Betti table (see \Cref{fig:bettitable}). The regularity $\reg(I)$ of a homogeneous ideal $I\subset R$ is the integer $\reg(I) = \max\{ j-i\ \vert\ \beta_{ij}(I) \neq 0 \}$. For the sake of brevity, we will omit the dependence on $I$ in the notation of the graded Betti numbers taking care not to cause any possible confusion. Moreover, for a non-homogeneous ideal $I\subset R$, the integer $\beta_{i,j}(I)$ is defined to be the $(i,j)$-th graded Betti number of its initial ideal $\In I$.
\begin{figure}[ht]
    \centering
        \begin{tabular}{c|ccccc}
             &0&  $\cdots$ &$i$ &$\cdots$ & $n$\\ 
             \hline
             $\vdots$&  $\ddots$ &$\cdots$   &$\vdots$ &$\cdots$&$\cdots$\\
             $j$&  $\cdots$   &$\cdots$   &$\beta_{i,i+j}$ &$\cdots$&$\cdots$\\
             $\vdots$&  $\vdots$  &$\cdots$   &$\vdots$ &$\cdots$ &$\ddots$
        \end{tabular}
    \caption{The Betti table}
    \label{fig:bettitable}
\end{figure}
\end{definition}
By convention, all the non-displayed  entries correspond to zero Betti numbers.
The choice of indices in the Betti table and the number of displayed columns are motivated by the following propositions that we shall use implicitly later in the paper.
\begin{prop}[{\cite[Proposition 1.9]{SYZYGIES}}]\label{thmhilbertsyz} Let ${\beta_{i,j}}$, for $i,j\in\BZ$, be the graded Betti numbers of a graded $R$-module. If for a given $i$ there is an integer $d$ such that $\beta_{i,j} = 0$ for all $j < d$, then $\beta_{i+1,j+1} = 0$ for all $j < d$. 
\end{prop}
\begin{prop}[{\cite[Hilbert syzygy theorem]{AproduNagel}}]  
  Any graded finitely generated $\mathbb{C}[x_1,  \ldots,x_n]$-module $M$ has a graded free resolution of length at most $n$.
\end{prop}

\subsection{The nested Hilbert scheme of points and its tangent space}\label{subsec:prlinested}
In this subsection, we recall some well-known facts about classical and nested Hilbert schemes of points and we settle the notation. Although the nested Hilbert scheme is  considered a generalisation of the classical Hilbert scheme defined by Grothendieck, we present here the theory in the nested setting as many of the applications of our results concern this generalisation. The classical case will then be recovered as a special instance of the nested one.
	
	Let $X$ be a quasi-projective variety, and let $Z\hookrightarrow X$ be a closed subscheme defined by the ideal sheaf $\OI_Z\subset \OO_X$. When $Z$ is a zero-dimensional subscheme, the ring $H^0(Z, \OO_Z)$ is a semilocal Artinian $\BC$-algebra and, as a consequence, it is a finite-dimensional vector space over the field of complex numbers. The complex dimension of $H^0(Z,\OO_Z)$ is called the \textit{length} of $Z$ or the \textit{colength} of $\OI_Z$. We denote it by $d_Z$ (or $d_{\OI_Z}$)
	\[
	d_Z=\length Z=\colen \OI_Z = \dim_\BC H^0(Z, \OO_Z).
	\]

\begin{notation}
    In order to ease the notation, for any vector $\underline{d}\in \BZ^r$ we  denote by $d_i$, for $i=1,\ldots,r$, its entries. Moreover, if $\underline{d}\in \BZ^r$ is a non-decreasing sequence of positive integers, a $\underline{d}$-nesting (or simply $r$-nesting) $\underline{Z}$ in $X$ is a sequence $\underline{Z}=(Z^{(1)},\ldots,Z^{(r)})$ of closed zero-dimensional subschemes $Z^{(1)}\subset\cdots \subset Z^{(r)}\subset X$ such that $\length Z^{(i)}=d_i$, for $i=1,\ldots,r$. Finally, the support of the nesting $   \underline Z$ is the set-theoretic support of the scheme $Z^{(r)}$, i.e.~$\Supp \underline{Z}= \Spec( \OO_X /\sqrt{\OI_{Z^{(r)}}})$.
\end{notation}

	When $Z$ is a zero-dimensional closed subscheme of $X$ and $H^0(Z,\OO_Z)$ is a local $\BC$-algebra, i.e.~the support of $Z$ consists of one point, we  say that $Z$ is a \textit{fat point}. Similarly, given a non-decreasing sequence of positive integers $\underline{d}\in\BZ^r$, a \textit{fat nesting} in $X$ is a nesting $\underline{Z}=(Z^{(i)})_{i=1}^r$ of fat points in $X$.

Let $X$ be a smooth quasi-projective variety and let $\underline{d} \in \BZ^r$ be a non-decreasing   sequence of positive integers. The $\underline{d}$-nested Hilbert functor  of $X$  is the contravariant functor $\underline{\Hilb}^{\underline{d}}X:\Sch^{\opp}_{\BC} \to \Sets$ defined as follows
\begin{equation}\label{eq:defHilb}
    \left(\underline{\Hilb}^{\underline{d}}X\right)(S)=\Set{\OZ^{(1)}\subset \cdots \subset \OZ^{(r)} \subset X\times S | \OZ^{(i)}  \mbox{ closed },\ S\mbox{-flat, $S$-finite, }  \length_S\OZ^{(i)}=d_i,\mbox{ for }i=1,\ldots,r   }, 
\end{equation}
where $\length_S$ denotes the $S$-relative length. The functor  $\underline{\Hilb}^{\underline{d}}X$ is representable  by a quasi-projective scheme, see \cite[Theorem 4.5.1]{sernesi} and \cite{Kleppe}. We call\footnote{This scheme is sometimes called flag Hilbert scheme.} it the \textit{$\underline{d}$-nested Hilbert scheme}  and we denote it by $\hilbert{\underline{d}}{X}$. Recall that the closed points of $\hilbert{\underline{d}}{X}$ are in bijection with the $\underline{d}$-nestings of closed subschemes of $X$. For this reason, we denote points of the nested Hilbert scheme by $[\underline{Z}]$. Notice that, for  $r=1$, one recovers the classical Hilbert functor, whose representability was proven by Grothendieck in \cite{Grothendieck_Quot}. 

We will often denote by $\Hilb^\bullet X$ the scheme locally of finite type 
\[
\Hilb^\bullet X= \coprod_{r \geqslant 1} \coprod_{\underline{d}\in\BZ^r} \Hilb^{\underline{d}} X.
\]
It is worth mentioning that the scheme $\Hilb^\bullet X$ represents a functor analogous to the one given in \Cref{eq:defHilb}. Precisely, the functor associating to a base scheme $S$, the set of nestings of $S$-families without restriction on the number of nestings and on the  relative lengths. Notice that the connected components of $\Hilb^\bullet X$ are precisely the $\underline{d}$-nested Hilbert schemes of points on $X$, see \cite{FOGARTY,CHEACELLULAR} and therein references.

\begin{remark}\label{rem:closedinprod}
    Notice that the nested condition identifies the nested Hilbert scheme $\hilbert{\underline{d} }{X}$ with the closed subscheme of the product ${\prod}_{i=1}^r \hilbert{d_i}{X}$ cut out by the nesting conditions, see \cite{sernesi}.
\end{remark}

\begin{remark}\label{rem:etale}
     Since the questions we address in this paper are local in nature and our results concern smooth quasi-projective varieties, it is fair to put $X\cong \BA^n$ and hence to work up to \'etale covers. Moreover, whenever not specified, a fat nesting $\underline{Z}=(Z^{(i)})_{i=1}^r$ will be implicitly assumed to be supported at the origin $0\in \BA^n$, i.e.~such that the defining ideal of the scheme $Z^{(r)}$ is $\Fm$-primary.  
 \end{remark}
 
 As there is a bijection between closed subschemes $Z\subset \BA^n$ and their defining ideals $I_Z\subset R$, we will denote points of the $\underline{d}$-nested Hilbert scheme $\hilbert{\underline{d}}{\BA^n}$ by $[Z^{(1)}\subset \cdots \subset Z^{(r)}]$ or $[I_{Z^{(1)}}\supset \cdots \supset I_{Z^{(r)}}]$ referring to both as $\underline{d}$-nestings (or simply $r$-nestings).
	
	Recall that the $\underline{d}$-nested Hilbert scheme has always  a distinguished component. Precisely, the \emph{smoothable component}. It is defined as the closure of the open subscheme $U\subset \hilbert{\underline{d}}{\BA^n}$ parametrising $\underline{d}$-nestings $\underline{Z}$ with $ Z^{(r)}$ reduced. We denote it by $\Hilb^{\underline{d}}_{\mathrm{sm}}{ \BA^n} $ and we refer to points in $\Hilb^{\underline{d}}_{\mathrm{sm}}{ \BA^n} $ as \emph{smoothable points}.
	
	\begin{definition}
		An irreducible component $V\subset   \hilbert{\bullet}{\BA^n}$ is \textit{elementary} if it parametrises just fat nestings, and \textit{composite} otherwise. 
        \end{definition}
 
In  \Cref{subsec:smallcomp}, we give new examples of  elementary components on $\Hilb{\BA^4}$ of dimension smaller or equal to the dimension of the smoothable one.  This information contributes to the knowledge of the existence of irreducible components on $\hilbert{\bullet}{\BA^4}$ as every irreducible component is generically étale-locally a product of elementary components, see \cite{Iarrocomponent}. 

	We conclude this subsection by reporting on the tangent space of $ \hilbert{ }{\BA^n}$. The following result characterises the tangent space to the classical Hilbert scheme at a given point $[I]\in\hilbert{d}{\BA^n}$.
	
	\begin{theorem}[{\cite[Corollary~6.4.10]{FGAexplained}}]
		Let $d > 0$ be a positive integer and let $[I]\in \hilbert{d}{\BA^n}$ be any point. Let $\mathsf{T}_{[I]}\hilbert{d}{\BA^n}$ denote the tangent space to $\hilbert{d}{\BA^n}$ at $[I]$. Then, there is a canonical isomorphism
		\begin{equation}\label{eq:isotg}
		    \mathsf{T}_{[I]}\hilbert{d}{\BA^n} \simeq \Hom_{R}(I, R/I)  .
		\end{equation}
	\end{theorem}

    Let us now fix some non-decreasing sequence $\underline{d}\in \BZ^r_{>0 }$  of positive integers  and a point $[\underline{I}] \in \hilbert{\underline{d}}{\BA^n}$. Recall that $\hilbert{\underline{d}}{\BA^n}$ naturally sits inside the product ${\Pi}_{i=1}^r \hilbert{d_i}{\BA^n}$ as a closed subscheme, see \Cref{rem:closedinprod}. This gives a natural identification of the tangent space $\mathsf T_{[\underline{I}]} \hilbert{\underline{d}}{\BA^n}$ with the vector subspace of the direct sum $ \bigoplus_{i=1}^r\mathsf T_{[I^{(i)}]}\hilbert{d_i}{\BA^n}$   consisting of $r$-tuples $(\varphi_i)_{i=1}^r$ making all the squares of the following diagram 
\begin{equation}\label{eq:tangentvector}
    \begin{tikzcd}[row sep=huge,column sep=huge]
     I^{(1)}\arrow[d,"\varphi_1"'] & I^{(2)}\arrow[l,hook']\arrow[d,"\varphi_2"]& I^{(3)}\arrow[l,hook']\arrow[d,"\varphi_3"]& I^{(r-1)}\arrow[d,"\varphi_{r-1}"]\arrow[l,dotted,hook']&I^{(r)}\arrow[l,hook']\arrow[d,"\varphi_r"]\\
     R/I^{(1)} & R/I^{(2)}\arrow[l,two heads]& R/I^{(3)}\arrow[l,two heads]& R/I^{(r-1)}\arrow[l,dotted,two heads]& R/I^{(r)}\arrow[l,two heads],
\end{tikzcd}
\end{equation}
commute {\cite[Section 4.5]{sernesi}}.

\begin{remark}
    Recall that, by the results in  \cite{FOGARTY,CHEACELLULAR}  the connected components of $\Hilb^\bullet \BA^n$ are precisely the $\underline{d}$-nested Hilbert schemes for $\underline{d}\in\BZ^r$ non-decreasing sequence of positive integers (with possibly $r=1$). Therefore, there is a canonical isomorphism
    \begin{equation}\label{eq:tgcc}
        \mathsf{T}_{[\underline{Z}]}\hilbert{\underline{d}}{\BA^n}\cong\mathsf{T}_{[\underline{Z}]}\hilbert{\bullet}{\BA^n}.
    \end{equation}
    In what follows we will intensively adopt the identification in \eqref{eq:tgcc} to ease the notation.
\end{remark}

\subsection{The Bia{\l{}}ynicki--Birula decomposition}\label{subsec:BB}

Let $Z\subset \BA^n$ be a fat point supported at the origin  $0\in\BA^n $ defined by an $\Fm$-primary ideal $I_Z \subset R$. Put 
\[ (I_Z)_{\geqslant k} = I_Z\cap \Fm^k\qquad\mbox{ and }\qquad (R/I_Z)_{\geqslant k}= (\Fm^k + I_Z)/I_Z \subset R/I_Z. \]

\begin{definition}\label{def:negativetangents}
Let $\underline Z = (Z^{(1)}\subset \cdots\subset Z^{(r)}$) be a fat nesting supported at the origin $0\in\BA^n$.
Then, the \textit{non-negative part of the tangent space} $\mathsf T_{ [\underline{Z}]}  \hilbert{\bullet}{\BA^n}$  is the following vector subspace
\[
\mathsf T_{ [\underline{Z}]}^{\geqslant0}  \hilbert{\bullet}{\BA^n} = \left\{\varphi\in\mathsf T_{ [\underline{Z}]} \hilbert{\bullet}{\BA^n}\ \middle\vert\ \varphi\big((I_{Z^{(i)}})_{\geqslant k}\big) \subset (R/I_{Z^{(i)}})_{\geqslant k}\mbox{ for all }k \in\BN\mbox{ and }i=1,\ldots,r\right\}. 
\]
While, the \textit{negative tangent space} at $[\underline{Z}]\in \hilbert{\bullet}{\BA^n}$ is the quotient vector space
\[
\mathsf T_{ [\underline{Z}]}^{< 0}  \hilbert{\bullet}{\BA^n}=\frac{ \mathsf T_{ [\underline{Z}]}  \hilbert{\bullet}{\BA^n}}{ \mathsf T_{ [\underline{Z}]}^{\geqslant0}  \hilbert{\bullet}{\BA^n}}.
\]
\end{definition}

    Note that non-negative tangent vectors can be understood as concatenations of commutative diagrams of the form \eqref{eq:tangentvector}, where $\varphi_i\in \mathsf T_{[Z^{(i)}]}^{\ge0} \hilbert{\bullet}{\BA^n}$, for all $i=1,\ldots,r$. 
The non-negative part of the tangent space can be interpreted as the tangent space to the so-called Bia{\l{}}ynicki--Birula decomposition, whose definition we recall now. Consider the diagonal action of the torus $\mathbb G_m=\Spec \BC[s,s^{-1}]$ on $ \hilbert{\underline{d}}{\BA^n} $ given by homotheties. Then, the Bia{\l{}}ynicki--Birula decomposition is the quasi-projective scheme $\hilbert{\underline{d},+}{\BA^n}$ representing the following functor
\[ 
   \left( \underline{\Hilb}^{\underline{d},+}\BA^n\right)(B) = \left\{\varphi\colon \overline{\mathbb G}_m \times B \rightarrow \hilbert{\underline{d}}{\BA^n} \ \middle\vert\   \varphi \mbox{ is $\mathbb G_m$-equivariant}\right\} 
\]
where, by convention $\overline{\mathbb G}_m= \Spec \BC[ s^{-1}]$. 
\begin{remark}\label{rem:BBdentroHilb}
    Set-theoretically, the Bia{\l{}}ynicki--Birula decomposition is the subset of the nested Hilbert scheme parametrising fat nestings supported at the origin $0\in\BA^n$. Notice that, under this association, every point $[\underline I]\in\Hilb^{\underline{d},+}\BA^n$ has an open neighbourhood that can be interpreted as a locally closed subscheme of $\Hilb^{\bullet}\BA^n$.   
\end{remark}

According to the above notation, we put
\[
\Hilb^{\bullet,+}\BA^n=\coprod_{r\geqslant 1} \coprod_{\underline d\in\BZ^r} \Hilb^{\underline d,+} \BA^n 
\] 

The following proposition from \cite{ELEMENTARY} expresses the tangent space   $\mathsf T_{ [\underline{Z}]}  \Hilb^{+} \BA^n$  in  terms of the non-negative tangent space at $[\underline{Z}]\in \hilbert{\bullet}{\BA^n}$. We adopt the identification on tangent spaces analogous to \Cref{eq:tgcc}.
\begin{prop}[{\cite[Theorem 4.11]{ELEMENTARY}}]\label{rem:nonneg-punctual}
Let $[\underline{Z}]\in \hilbert{\bullet,+}{\BA^n}$ be a fat nesting. Then, we have
\[
    \mathsf T_{ [\underline{Z}]}  \hilbert{\bullet,+}{\BA^n}= \mathsf T_{ [\underline{Z}]}^{\geqslant0}  \hilbert{\bullet}{\BA^n} .
\]    
\end{prop}

\begin{remark}\label{rem:deftheta}
  As shown in \cite{ELEMENTARY}, when $[Z]\in \Hilb^{d,+} \BA^n$ is a fat point,  the tangent space to $\BA^n$ at its support $\{0\}=\Supp Z\subset\BA^n$ maps to the tangent space to $ \Hilb^{d} \BA^n$ at $[Z]$. Similarly, this happens for fat nestings and we give now some details. Let us identify the partial derivatives $\frac{\partial}{\partial x_j}$, for $j=1,\ldots,n$, with a basis of the tangent space $\mathsf T_0\BA^n$ and let us consider a fat nesting $[\underline{Z}]\in \hilbert{\bullet,+}{\BA^n}$. In this setting we have a natural map
  \[
    \begin{tikzcd}
         \mathsf T_{0} \BA^n \arrow[r,"\widetilde\theta"] & \mathsf T_{[\underline{Z}]}\Hilb^{\bullet} \BA^n,
    \end{tikzcd}
  \]
associating tangent vectors to $\BA^n$ at the origin to first order deformations consisting of translations. More precisely, the partial derivative $\frac{\partial}{\partial x_j}$, for $j=1,\ldots,n$, maps to an infinitesimal first order translation of all the schemes $Z^{(i)}$, for $i=1,\ldots,r$, along the $j$-th coordinate axis preserving the nesting conditions.

  We denote by  $\theta : \mathsf T_{0} \BA^n \to \mathsf T_{ [\underline{Z}]}^{< 0} \Hilb^{\bullet} \BA^n
  $ the map defined as  the composition of $\widetilde \theta$ with the canonical projection defining the negative tangent space, see \Cref{def:negativetangents}.
\end{remark}
\begin{definition}\label{def:TNT}
    Let $[\underline{Z}]\in \Hilb^{+} \BA^n$ be a fat nesting. Then, $[\underline{Z}]$ has TNT (Trivial Negative Tangents) if the   map 
\[
         \mathsf T_{0} \BA^n  \xrightarrow{\theta} \mathsf T_{ [\underline{Z}]}^{< 0}  \Hilb^\bullet \BA^n
\]
    is surjective.
\end{definition} 
\Cref{thm:tnt per nested} is a generalisation of {\cite[Theorem 4.9]{ELEMENTARY}} and it relates the existence of ideals having TNT and the existence of generically reduced elementary components. 
\begin{theorem}[{\cite[Theorem 4]{UPDATES}}] \label{thm:tnt per nested} 
Let $V\subset \Hilb^{\bullet} \BA^n$ be an irreducible component.  Suppose that $V$ is generically reduced. Then $V$ is elementary if and only if a general point of $V$ has trivial negative tangents.  
\end{theorem}

Elementary components are considered the building blocks of the Hilbert schemes of points as each irreducible component is proven, in \cite{Iarrocomponent}, to be generically \'etale locally a product of elementary components.  In \Cref{subsec:smallcomp}, we give new examples of elementary components on $\hilbert{\bullet}{\BA^4}$ and, in \Cref{sec:small} we present other examples in dimensions 5 and 6. There are few elementary components known in the literature, see  \cite{Iarrob,ELEMENTARY,Satrianostaal,GALOIS,SomeElementary,MoreElementary,ERMANVELASCO,Sha90,UPDATES}  and references therein for other examples of elementary components. We certify their existence by exhibiting  explicit points having TNT. This information contributes to the knowledge of the growth of the number of irreducible components on $\hilbert{d}{\BA^n}$ as $d$ tends to infinity whose asymptotics have been investigated in \cite{Iarrocomponent}.

\begin{remark}\label{rem:BBHOMOo}  The fixed locus of  the diagonal action of the torus $\mathbb G_m=\Spec \BC[s,s^{-1}]$ on $ \Hilb^{\bullet} \BA^n $ agrees with the locus parametrising nestings of homogeneous ideals. As a consequence, given a nesting $ \underline{I}=( I^{(1)}\supset\cdots\supset I^{(r)}) $ of homogeneous ideals, the  $\mathbb G_m$-action lifts to the tangent space $\mathsf T_{[\underline{I}]} \Hilb^{\bullet}\BA^n$ and it induces an eigenspaces decomposition  
    \[
    \mathsf T_{[\underline{I}]} \Hilb^{ \bullet}\BA^n=\bigoplus_{k\in \BZ } \mathsf T_{[\underline{I}]}^{=k} \Hilb^{\bullet}\BA^n.
    \]
    This direct sum decomposition is consistent with \Cref{def:negativetangents} meaning that
    \[
    \mathsf T_{[\underline{I}]}^{\geqslant 0}\Hilb^{\bullet}\BA^n=\bigoplus_{k\geqslant 0 } \mathsf T_{[\underline{I}]}^{=k} \Hilb^{\bullet}\BA^n\quad \mbox{ and }\quad   \mathsf T_{[\underline{I}]}^{<0} \Hilb^{\bullet}\BA^n=\bigoplus_{k<0 } \mathsf T_{[\underline{I}]}^{=k} \Hilb^{\bullet}\BA^n,
    \]
    see \cite[Section 2]{ELEMENTARY} and \Cref{rem:homobb} for more details.
\end{remark}

\subsection{Hilbert stratification}\label{subsec:HSstrat}

\begin{notation}
 For the ideal $I\subset R$ of a fat point supported at the origin, the function $\bh_{R/I}$ vanishes eventually. For this reason, we represent it as tuples of positive integers. 

    Given a nesting $\underline{I}= (I^{(i)})_{i=1}^r $, of $\Fm$-primary ideals of finite colength $d_i$, for $i=1,\ldots,r$, we denote by $\underline{\bh}_{\underline{I}},\underline{\bh}_{R/\underline{I}}$ the $r$-tuples of Hilbert functions 
    \[
    \underline{\bh}_{\underline{I}}=\left(\bh_{I^{(i)}}\right)_{i=1}^r\quad\mbox{ and }\quad    \underline{\bh}_{R/\underline{I}}= \left(\bh_{R/I^{(i)}}\right)_{i=1}^r.
    \]
\end{notation}
Moreover, we denote by $|\underline{\bh}_{R/\underline{I}}|$ the non-decreasing sequence of positive integers  
\[
\lvert\underline{\bh}_{R/\underline{I}}\rvert=\left(\lvert\bh_{R/I^{(i)}}\rvert\right)_{i=1}^r = (d_1,\ldots,d_r)\in\BZ^r.
\]
The map
\[
    \begin{tikzcd}[row sep=tiny]
         \Hilb ^{\bullet,+}\BA^n \arrow[r]& \BN^r\\
         {\left[\underline{I}\right]}\arrow[r,mapsto]& \underline{\bh}_{R/\underline{I}} ,
    \end{tikzcd}
    \]
is locally constant, see \cite[Prop.~3.1]{ELEMENTARY}. 
Since the locally closed subsets
\begin{equation}\label{eq:defHSStratum}
    H_{\underline{\bh}}^n=\left\{ \left[\underline{I}\right] \in \Hilb^{\bullet,+}  \BA^n \ \middle\vert\ \underline{\bh}_{R/\underline{I}}\equiv \underline{\bh} \right\}\subset \Hilb^\bullet \BA^n ,
\end{equation}
where $\underline{\bh}=(\bh^{(1)},\ldots,\bh^{(r)}):\BZ\to\BN^r$ is an $r$-tuple of Hilbert functions compatible with the conditions imposed by the nestings, agree with the connected components of the Bia{\l{}}ynicki--Birula decomposition, they inherit a canonical scheme structure, see \Cref{rem:BBdentroHilb} and \cite{Iarropunctual}.

\begin{definition}
Given a $r$-tuple  $\underline{\bh}=(\bh^{(1)},\ldots,\bh^{(r)}):\BZ\to\BN^r$ of   functions, the \emph{Hilbert stratum} $ H_{\underline{\bh}}^n\subset \Hilb^\bullet \BA^n $  is the (possibly empty) locally closed subset given in \eqref{eq:defHSStratum}, endowed with the schematic structure induced by the Bia{\l{}}ynicki--Birula decomposition.
\end{definition}

\begin{remark}
In order to have a non-empty Hilbert stratum $H_{\underline{\bh}}^n$, the $r$-tuple $\underline{\bh}=(\bh^{(1)},\ldots,\bh^{(r)})$ must have finite support and it must satisfy two conditions:
\begin{itemize}
    \item the vector $|\underline{\bh}|$ is non-decreasing, i.e.~$|\bh^{(i)}|\leqslant|\bh^{(j)}| $ for all $1\leqslant i< j\leqslant r$;
    \item the strata $H_{\bh^{(i)}}^n$ are non-empty, for $i=1,\ldots,r$.
 \end{itemize}

    Moreover, recall that the set of functions ${\bh}\colon\mathbb{Z} \to \mathbb{N}$ for which $H^n_{\bh}$ is non-empty is characterised by Macaulay's theorem \cite[Theorem 4.2.10]{CMRings}.
\end{remark}

\begin{remark} \label{rem:homobb}
Let $\underline d\in\BZ^r_{> 0}$ be a non-decreasing sequence of positive integers. Then, there is a surjective morphism of schemes locally of finite type
    \[
    \begin{tikzcd}[row sep=tiny]
         \Hilb ^{\bullet,+}\BA^n \arrow[r,"\pi"]& (\Hilb^{\bullet,+} \BA^n)^{\mathbb G_m},
    \end{tikzcd}
    \] 
        which set-theoretically associates to the point corresponding to a $\underline{d}$-nesting $\underline{I}=(I^{(i)})_{i=1}^r$ the point corresponding to its initial $\underline{d}$-nesting $    \In\underline{I}= (\In I^{(i)})_{i=1}^r$, 
see \Cref{rmk:gradini}.       Recall that the tangent space to the fibre of $\pi $ over ${[}\underline{I}{]}\in(\Hilb ^+\BA^n)^{\mathbb G_m}$ identifies, via \Cref{rem:nonneg-punctual}, with $\mathsf T_{ [\underline{I}]}^{>0}  \Hilb^\bullet  \BA^n$, while the tangent space to $(\Hilb ^+\BA^n)^{\mathbb G_m}$ identifies with $\mathsf{T}_{ [\underline{I}]}^{=0}  \Hilb^\bullet  \BA^n$, see \Cref{rem:BBHOMOo} and \cite{multigraded}. In formulas, we have
    \[
    \mathsf{T}_{[\underline{I}]} \pi^{-1}([\underline I]) \cong  \mathsf T_{ [\underline{I}]}^{>0}  \Hilb^\bullet  \BA^n\quad\mbox{ and }\quad\mathsf T_{ [\underline{I}]}  (\Hilb ^+\BA^n)^{\mathbb G_m} \cong \mathsf T_{ [\underline{I}]}^{=0}  \Hilb^\bullet  \BA^n.
    \]

\end{remark}

\section{A special class of ideals}\label{sec:2-step}
In this section, we introduce the notion of 2-step ideal. This class of ideals is suitable for our purpose of studying irreducibility of $\hilbert{\underline{d}}{\BA^n}$. Indeed, the loci parametrising homogeneous ideals of this kind happen to be very large with respect to the smoothable component. For instance, the compressed algebras of length 78 considered by Iarrobino in \cite{IARRO} are of the form $R/I$ for $I\subset R$ a 2-step ideal, see \Cref{exa:iarrobino}.

\subsection{Definition and general properties of 2-step ideals}\label{subsec:2-stepdef}
We start by giving the definition and the basic properties of 2-step ideals.
  
\begin{definition}  \label{def:2step}
An ideal $I \subset R$ is \emph{2-step} of order $k > 0$, if  
\[
\mathfrak{m}^{k+2} \subset I \subset \mathfrak{m}^{k}\quad\mbox{ and }\quad I\not \subset \mathfrak{m}^{k+1}.
\] 
In this context we say that the Hilbert function of $I$ (or of $R/I$) is 2-step of order $k$.
\end{definition} 

\begin{remark}
    In terms of Hilbert function, the requirements  in \Cref{def:2step} are equivalent to 
\[
{\bh}_{R/I} (t) :
     \  \begin{cases}
    = \mathsf  r_{t} &\text{for~}t<k,\\
    < \mathsf r_{k} &\text{for~}t=k,\\
      = 0&\text{for~}t\geqslant k+2,
\end{cases}
\] see \Cref{fig:2stpdis} for a pictorial description.
 \end{remark}
    
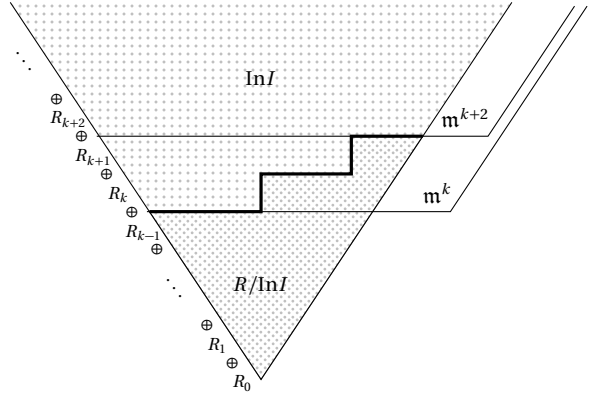
\begin{figure}[!ht]
\begin{center}
\begin{tikzpicture}[xscale = 0.33,yscale=0.5]
\draw [white,pattern=dots,opacity=0.5] (10,10)  -- (6.45,6.45) -- (3.6,6.45) -- (3.6,5.45) -- (0,5.45) -- (0,4.45) -- (-4.45,4.45)-- (-10,10);
\draw [,pattern=crosshatch dots,opacity =0.5](0,0) -- (6.45,6.45) -- (3.6,6.45) -- (3.6,5.45) -- (0,5.45) -- (0,4.45) -- (-4.45,4.45) -- cycle;

\draw[thin]  (10,10) -- (0,0) -- (-10,10);

\begin{scope}[shift={(0,-0.1)}]
\node at (0,0) [left]  {\scriptsize $R_0$};
\node at (-0.55,0.55) [left] {\scriptsize $\oplus$};
\node at (-1,1) [left] {\scriptsize $R_1$};

\node at (-1.55,1.55) [left]{\scriptsize $\oplus$};

\node at (-2.75,2.75)[left] {\scriptsize $\ddots$};

\node at (-3.55,3.55)[left] {\scriptsize $\oplus$};

\node at (-3.6,4) [left] {\scriptsize $R_{k-1}$};
\node at (-4.55,4.55) [left] {\scriptsize $\oplus$};

\node at (-4.9,5) [left]{\scriptsize $R_{k}$};
\node at (-5.55,5.55) [left] {\scriptsize $\oplus$};

\node at (-5.6,6)[left] {\scriptsize $R_{k+1}$};
\node at (-6.55,6.55)[left] {\scriptsize $\oplus$};

\node at (-6.6,7) [left]{\scriptsize $R_{k+2}$};
\node at (-7.55,7.55)[left] {\scriptsize $\oplus$};

\node at (-8.75,8.75) [left] {\scriptsize $\ddots$};

\draw [] (-4.55,4.55) -- (7.55,4.55) -- (13,10);
\node at (7.05,4.55) [above] {\small $\mathfrak{m}^k$};
\draw [] (-6.55,6.55) -- (9.05,6.55) -- (12.5,10);
\node at (8.15,6.55) [above] {\small $\mathfrak{m}^{k+2}$};

\end{scope}

\draw [very thick] (6.45,6.45) -- (3.6,6.45) -- (3.6,5.45) -- (0,5.45) -- (0,4.45) -- (-4.45,4.45);

\node at (0,2.5) [] {\footnotesize $R/\In I$};
\node at (0,8) [] {\footnotesize $\In I$};
\end{tikzpicture}
\caption{Pictorial representation of the initial ideal of a 2-step ideal.}
\label{fig:2stpdis}
\end{center}
\end{figure}

We exploit now the basic properties of 2-step Hilbert functions.  By definition of  2-step ideal of order $k$,  the Hilbert functions $\bh_I$ and $\bh_{R/I}$ of $I$  and of the corresponding quotient algebra are uniquely determined by the values  
\begin{equation}\label{eq:notation}
    \oh_{k} =   \rrr_{k} - \uh_{k}\quad\mbox{and}\quad\oh_{k+1}    = \rrr_{k+1} - \uh_{k+1},
\end{equation}
see  \Cref{notation:initial,notation:HStutti}. In fact, we have
\[
\bh_{I}(t) = \begin{cases}
    0,& 0 \leqslant t \leqslant k-1,\\
    \oh_{k} ,& t= k,\\ 
    \oh_{k+1}& t = k+1,\\
    \rrr_{t},& t \geqslant k+2,
\end{cases}
\quad\mbox{and}\quad
\bh_{R/I}(t) = \begin{cases}
   \rrr_{t},& 0 \leqslant t \leqslant k-1,\\
    \uh_{k} ,& t= k,\\ 
    \uh_{k+1},& t = k+1,\\
    0,& t \geqslant k+2,
\end{cases}
\]
and 
\[
\begin{split}
d_I = \dim_{\BC} (R/I) &{} = \sum_{q=0}^{k-1} \rrr_{q} + \uh_{k} + \uh_{k+1} = \binom{k+n-1}{n} + \uh_{k} + \uh_{k+1} = {}\\
& {} = \sum_{q=0}^{k+1} \rrr_{q} - \oh_{k} - \oh_{k+1} = \binom{k+n+1}{n} - \oh_{k} - \oh_{k+1}.    
\end{split}
\]

We stress that, given $ 0\leqslant \oh_{k} \leqslant \rrr_{k}$, the values of $\oh_{k+1}$ for which the Hilbert function is admissible are bounded from below by Macaulay's theorem, \cite[Theorem 4.2.10]{CMRings}.

In \Cref{lem: general properties Betti}, we list the basic properties of the minimal free resolutions of a homogeneous 2-step ideal.
\begin{lemma} \label{lem: general properties Betti}
Let $I\subset R$ be a homogeneous 2-step ideal of order $k>0$. Then,
\begin{enumerate}[\it (i)]
\item\label{it:bettionei} the regularity of $I$ satisfies $\reg(I)\leqslant k+2$,  
\item\label{it:Bettitwoii} the  Betti table of $I$ is
\begin{equation}\label{eq: betti table general}
\begin{array}{c|ccccc}
& 0& 1 & 2 & \ldots & n-1 \\
\hline
k& \beta_{0,k} & \beta_{1,k+1} & \beta_{2,k+2} & \cdots & \beta_{n-1,k+n-1}\\
k+1& \beta_{0,k+1} & \beta_{1,k+2} & \beta_{2,k+3}  & \cdots & \beta_{n-1,k+n}\\
k+2& \beta_{0,k+2} & \beta_{1,k+3} & \beta_{2,k+4}  & \cdots & \beta_{n-1,k+n+1}\\
\end{array}
\end{equation}
where
\begin{equation}\label{eq:bettitwostep}
\beta_{0,k} = \oh_{k}, \qquad
 \beta_{0,k+1}-\beta_{1,k+1} = \oh_{k+1}-n\oh_{k}, \qquad \beta_{0,k+2}-\beta_{1,k+2}+\beta_{2,k+2} = \rrr_{k+2} - n\oh_{k+1}+\tbinom{n}{2}\oh_{k}.
\end{equation}
\end{enumerate}
\end{lemma}

\begin{proof}
The first part of the statement is a direct consequence of the definition of 2-step ideals and of \cite[Corollary 4.4]{SYZYGIES}. We move now to the proof of the equalities in \eqref{eq:bettitwostep}. Consider the minimal graded free resolution $\begin{tikzcd}
    0&I\arrow[l]&F_\bullet \arrow[l,"d_0"],
\end{tikzcd}$ of $I$, where
\begin{equation}\label{eq: resolution I}
\begin{tikzcd}[column sep=small]
      F_\bullet:&\begin{array}{c}R(-k)^{\oplus\beta_{0,k}}\\[-6pt]\oplus\\[-6pt] R(-k-1)^{\oplus\beta_{0,k+1}} \\[-6pt]\oplus\\[-6pt] R(-k-2)^{\oplus\beta_{0,k+2}} \end{array}  & \begin{array}{c}R(-k-1)^{\oplus\beta_{1,k+1}}\\[-6pt]\oplus\\[-6pt] R(-k-2)^{\oplus\beta_{1,k+2}} \\[-6pt]\oplus\\[-6pt] R(-k-3)^{\oplus\beta_{1,k+3}} \end{array} \arrow[l,"d_1"'] & \quad \cdots\quad   \arrow[l] &  \begin{array}{c}R(-k-n+1)^{\oplus\beta_{n-1,k+n-1}}\\[-6pt]\oplus\\[-6pt] R(-k-n)^{\oplus\beta_{n-1,k+n}} \\[-6pt]\oplus\\[-6pt] R(-k-n-1)^{\oplus\beta_{n-1,k+n+1}} \end{array} \arrow[l,"d_{n-1}"'] &0 .\arrow[l]
\end{tikzcd}
\end{equation}
Then, the Hilbert series of $I$ is
\[
\boldsymbol{H}_I(T) = \sum_{t\geqslant 0} \bh_I(t) T^t = \oh_{k} T^k + \oh_{k+1} T^{k+1} + \sum_{t\geqslant k+2} \rrr_{t} T^t ,
\]
and it can be expressed as
\[
\dfrac{\sum_{i=0}^{n-1} \left( (-1)^{i}\sum_{j=k}^{k+2}\beta_{i,i+j}T^{i+j}\right)}{(1-T)^n},
\]
by \cite[Theorem 1.11]{SYZYGIES}. The equality $(1-T)^n \boldsymbol{H}_I(T) = \sum_{i=0}^{n-1} \left( (-1)^{i}\sum_{j=k}^{k+2}\beta_{i,i+j}T^{i+j}\right)$ in degree $k,k+1,k+2$ leads to \Cref{eq:bettitwostep}.
\end{proof}

\begin{remark}  
We note that the definition of 2-step ideals includes also very compressed algebras, cf. \Cref{def:compressed}.

If we add the assumption $\mathfrak{m}^{k+1} \not\subset I$ we get the equality $\reg(I)=k+2$ at point {\it(\ref{it:bettionei})} of \Cref{lem: general properties Betti}.  This can be verified by considering the minimal graded free resolution of the quotient algebra $R/I$  
\[
\begin{tikzcd} 
    0 &R/I \arrow[l]&R \arrow[l,"\pi"']& F_\bullet\arrow[l,"j\circ d_0"'] 
\end{tikzcd}
\]  
where $F_\bullet$ is the same as \eqref{eq: resolution I}, $j:I\hookrightarrow R$ is the inclusion and $\pi: R \to R/I $ is the canonical projection. Indeed, we have the equality
\[
(1-T)^n \boldsymbol{H}_{R/I}(T) = (1-T)^n\left(\sum_{t=0}^{k-1}\rrr_{t} T^t + \uh_{k} T^k + \uh_{k+1}T^{k+1}\right) = 1 - \sum_{i=0}^{n-1} \left( (-1)^{i}\sum_{j=k}^{k+2}\beta_{i,i+j}T^{i+j}\right),
\]
that in degree $k+1+n$ reads as
\[
(-1)^n\uh_{k+1} = - (-1)^{n-1}\beta_{n-1,k+n+1}.
\]
Hence, the condition $\mathfrak{m}^{k+1}\not\subset I$ implies $\beta_{n-1,k+n+1} = \uh_{k+1} \neq 0$, that is $\reg(I) = k+2$.  
\end{remark}

\begin{remark}
    The minimal graded free resolution of a homogeneous ideal $I\subset R$ encodes several information about the tangent space at the point $[I]$ to the Hilbert scheme $ \Hilb^\bullet \mathbb{A}^n$. Indeed, by  applying the functor $\Hom_R(\_,R/I)$ to the resolution \eqref{eq: resolution I} of $I$, we obtain a sequence exact in the first two terms
\begin{equation}\label{eq:Hom res}
\begin{tikzcd}[column sep = 1em]
    0 \arrow[r]&\textnormal{Hom}_R(I,R/I) \arrow[r]& \textnormal{Hom}_R\left(\begin{array}{l}R(-k)^{\oplus\beta_{0,k}}\\[-6pt]\quad\oplus\\[-6pt] R(-k-1)^{\oplus\beta_{0,k+1}} \\[-6pt]\quad\oplus\\[-6pt] R(-k-2)^{\oplus\beta_{0,k+2}} \end{array}\hspace{-4pt},R/I\right) \arrow[r,"d_1^\vee"]& \textnormal{Hom}_R\left(\begin{array}{l}R(-k-1)^{\oplus\beta_{1,k+1}}\\[-6pt]\quad\oplus\\[-6pt] R(-k-2)^{\oplus\beta_{1,k+2}} \\[-6pt]\quad\oplus\\[-6pt] R(-k-3)^{\oplus\beta_{1,k+3}} \end{array}\hspace{-4pt},R/I\right) \arrow[r]&   \cdots,
\end{tikzcd}
\end{equation}
which implies, together with the isomorphism in \eqref{eq:isotg}, the identification
\[
\mathsf{T}_{[I]} \Hilb^\bullet \mathbb{A}^n \simeq \text{Hom}_R(I,R/I) = \ker d_1^\vee.
\]  
\end{remark}
Recall that the non-negative part $\textsf{T}_{[I]}^{\geqslant 0} \Hilb^\bullet \mathbb{A}^n$ of the tangent space of $ \Hilb^\bullet \mathbb{A}^n$ at $[I]$ can be understood as the tangent space to the Bia{\l{}}ynicki--Birula decomposition, see \Cref{rem:homobb}. In particular, when $\textsf{T}_{[I]}^{\geqslant 0} \Hilb^\bullet \mathbb{A}^n$ happens to be entirely unobstructed, its dimension agrees with the dimension of the unique irreducible component of the Bia{\l{}}ynicki--Birula decomposition containing the point $[I]$, see \Cref{rem:nonneg-punctual}. This observation highlights the role of the non-negative part of the tangent space in finding loci inside the Hilbert scheme of dimension as big as possible, and it motivates the following Lemma.

\begin{lemma}\label{lem: tangent space}
Let $I \subset R$ be a homogeneous 2-step ideal of order $k>0$. Then, we have
\[
\begin{split}
\dim_{\BC} \mathsf{T}^{=t}_{[I]} \Hilb^\bullet \mathbb{A}^n&{}= 0, \qquad\forall\ t \geqslant 2,\\
\dim _{\BC}\mathsf{T}^{=1}_{[I]} \Hilb^\bullet \mathbb{A}^n &{} = \oh_{k}\uh_{k+1},\\
\dim_{\BC} \mathsf{T}^{=0}_{[I]} \Hilb^\bullet \mathbb{A}^n &{} \geqslant \max\big\{0,\oh_{k}\uh_{k} + (\oh_{k+1}-n\oh_{k})\uh_{k+1}\big\}.
\end{split}
\]
\end{lemma}
\begin{proof}
In degree $t$, the complex \eqref{eq:Hom res} reads as
\[
\begin{tikzcd}[column sep=small]
    0 \arrow[r]&\textnormal{Hom}_R(I,R/I)_t \arrow[r]& \begin{array}{l}\big((R/I)_{k+t}\big)^{\oplus\beta_{0,k}}\\[-6pt]\quad\oplus\\[-6pt] \big((R/I)_{k+1+t}\big)^{\oplus\beta_{0,k+1}} \\[-6pt]\quad\oplus\\[-6pt] \big((R/I)_{k+2+t}\big)^{\oplus\beta_{0,k+2}} \end{array} \arrow[r,"d_1^\vee"]& \begin{array}{l}\big((R/I)_{k+1+t}\big)^{\oplus\beta_{1,k+1}}\\[-6pt]\quad\oplus\\[-6pt] \big((R/I)_{k+2+t}\big)^{\oplus\beta_{1,k+2}} \\[-6pt]\quad\oplus\\[-6pt] \big((R/I)_{k+3+t}\big)^{\oplus\beta_{1,k+3}} \end{array} \arrow[r]&   \cdots,
\end{tikzcd}
\] 
and
\[
\begin{split}
\dim_{\BC} (\ker d_1^\vee)_t \geqslant{}& \sum_{j=0}^{2} \dim_{\BC} \big((R/I)_{k+t+j}\big)^{\oplus\beta_{0,k+j}} - \sum_{j=0}^{2} \dim_{\BC} \big((R/I)_{k+t+1+j}\big)^{\oplus\beta_{1,k+1+j}} = {}\\
& \beta_{0,k}\uh_{k+t} + (\beta_{0,k+1}-\beta_{1,k+1})\uh_{k+t+1} + (\beta_{0,k+2}-\beta_{1,k+2})\uh_{k+t+2} - \beta_{1,k+3}\uh_{k+t+3}.
\end{split}
\]
The statement is then a consequence of point \textit{(\ref{it:Bettitwoii})} in \cref{lem: general properties Betti}.
\end{proof}

The following theorem will be crucial in the rest of the paper, see \Cref{rem:1nonobstruct}.

\begin{theorem}\label{thm:1nonob}
Given a homogeneous 2-step ideal $I \subset R$ of order $k$, there is a canonical isomorphism  
\[
\textnormal{Hom}_R(I,R/I)_1\cong \textnormal{Hom}_{\BC}(I_k,(R/I)_{k+1}).
\] 
Moreover, all tangent vectors of degree 1 are unobstructed.   
\end{theorem}

\begin{proof}

By definition of 2-step ideal of order $k>0$, we have
\[
\textnormal{Hom}_R(I,R/I)_1=\Set{\varphi\in \textnormal{Hom}_R(I,R/I) | \varphi(I_k)\subseteq (R/I)_{k+1},\ \varphi(I_{k+1})=(0)=(R/I)_{k+2}}.
\]
Hence, we can identify via restriction the space $\textnormal{Hom}_R(I,R/I)_1$ with a complex vector subspace of $\textnormal{Hom}_{\BC}(I_k,(R/I)_{k+1})$. Notice also that, again by definition of 2-step ideal of order $k$, the vector space $(R/I)_{k+1}$ is entirely contained in the socle $(0_{R/I}:\Fm)$ of the local algebra $R/I$. As a consequence any $\BC$-linear homomorphism between $I_k$ and $(R/I)_{k+1}$ lifts to a unique $R$-linear homomorphism of degree $1$ between $I$ and $R/I$. This proves the first part.

We move now to the proof of the unobstructedness of ${\Hom}_R(I,R/I)_1$. Fix a basis $G = \Set{g_1,\ldots,g_{\oh_{k}}} \subset R$ of $I_k$ and consider some element $\varphi\in {\Hom}_R(I,R/I)_1$. Let us denote by $f_i=\varphi(g_i)\in (R/I)_{k+1}$, for $i=1,\ldots,{\oh_{k}}$, the images of the elements $g_i$ under the homomorphism $\varphi$. Note that, as a consequence of the first part of the statement, the homomorphism $\varphi$ is uniquely determined by the elements $f_i$'s. Fix also homogeneous lifts $\widetilde{f}_i\in R_{k+1}$,  for $i=1,\ldots,{\oh_{k}}$. By construction, the ideal $(g_i+\varepsilon \widetilde{f}_i \ |\  i=1,\ldots,{\oh_{k}}) + I_{\geqslant k+1}\subset R [\varepsilon]/\varepsilon^2$ defines a flat family over the spectrum of dual numbers. In order to conclude the proof we show that the ideal  
\[
I^+ = (g_i+t \widetilde{f}_i \ |\  i=1,\ldots,{\oh_{k}}) + I_{\geqslant k+1} \subset R [t]
\]
defines a flat family over the affine line $\BA^1$ with coordinate $t$. Since the scheme $\BA^1$ is reduced it is enough to show that, for any $t_0 \in \BC$ the ideal $I_{t_0}\subset R[t_0]$ obtained via base change has initial ideal $I$. This implies that the length of the fibres of the family defined by $I_t$ is constant along $\BA^1$. 

Clearly, by construction we have $\In I_{t_0}\supseteq I$. Therefore, we focus on the opposite inclusion. Let us fix some element $p\in I_{t_0}$. If $\deg\In p = k+2$, then $\In p\in \Fm^{k+2}=I_{k+2}$. On the other hand, if $\deg\In p = k$ we must have $p=\sum_{i=1}^{\oh_k}(\alpha_i + u_i )(g_i+t_0 \widetilde f_i) +q$, for some $\alpha_1,\ldots,\alpha_{\oh_k}\in\BC$, $u_1,\ldots,u_{\oh_k} \in \mathfrak{m}$ and $q\in I_{\geqslant k+1}$, which gives $\In p=\sum_{i=1}^{\oh_k}\alpha_ig_i\in I$. Finally,  if $\deg\In p = k+1$ we have   $p=\sum_{i=1}^{\oh_k}\ell_i(g_i+t_0 \widetilde f_i) + q + r$, where $\ell_1,\ldots,\ell_{\oh_k}\in R$ have order 1, $q\in I_{k+1}$ and $r\in I_{\geqslant k+2}$ which concludes the proof.  
 \end{proof}

\begin{figure}[!ht]
    \centering
    \begin{tikzpicture}[xscale = 0.33,yscale=0.5]
\draw[white,ultra thin,fill=red,opacity=0.5] (-4.5,4.5) -- (0,4.5) -- (0,5.5) -- (-5.5,5.5) -- cycle;
\draw[white,ultra thin,fill=blue,opacity=0.5] (3.5,5.5) -- (5.5,5.5) -- (6.5,6.5) -- (3.5,6.5) -- cycle;

\draw [white,pattern=dots,opacity=0.5] (9,9)  -- (6.5,6.5) -- (3.5,6.5) -- (3.5,5.5) -- (0,5.5) -- (0,4.5) -- (-4.5,4.5)-- (-9,9);
\draw [white,pattern=crosshatch dots,opacity =0.5]( 3,3)-- (6.5,6.5) -- (3.5,6.5) -- (3.5,5.5) -- (0,5.5) -- (0,4.5) -- (-4.5,4.5) -- (-3,3);

\draw[thin]  (-9,9) -- (-3,3);
\draw[thin]  (9,9) -- (3,3);

\draw [thick] (6.5,6.5) -- (3.5,6.5) -- (3.5,5.5) -- (0,5.5) -- (0,4.5) -- (-4.5,4.5);

\draw [-stealth,very thick] (-2.5,5) to[out=45,in=160]node[above]{\small $\varphi$} (4.75,6);

\begin{scope}[shift={(0.05,-0.05)}]
\node at (-5,5) [left] {\scriptsize $R_k$};
\node at (-5.5,5.5) [left] {\scriptsize $\oplus$};
\node at (-5.8,6) [left] {\scriptsize $R_{k+1}$};
\node at (-6.5,6.5) [left] {\scriptsize $\oplus$};
\node at (-6.8,7) [left] {\scriptsize $R_{k+2}$};
\node at (-4.5,4.5) [left] {\scriptsize $\ddots$};
\node at (-8,8) [left] {\scriptsize $\ddots$};
\end{scope}
\end{tikzpicture}
    \caption{Pictorial representation of a positive tangent vector in $\mathsf{T}_{[I]} \Hilb^\bullet \mathbb{A}^n$ for a homogeneous 2-step ideal. The red area corresponds to $I_k$, while the blue one corresponds to $(R/I)_{k+1}$.}
    \label{fig:positive tangent vectors}
\end{figure}
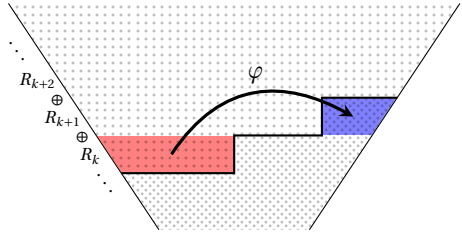

As a consequence of \Cref{thm:1nonob}, we get the following corollary.

\begin{corollary}\label{rem:1nonobstruct}
 Let $I\subset R$ be a homogeneous 2-step ideal. Then, the fibre $\pi_{\bh_I}^{-1}([I]) $ of the initial ideal morphism $H_{\bh_I}^n\xrightarrow{\pi_{\bh_I}}\OH_{\bh_I}^n$ is an affine space of dimension equal to  $\dim_{\BC}\mathsf{T}_{[I]}^{=1} \Hilb^\bullet \mathbb{A}^n$. 
\end{corollary}
\begin{proof} The statement is a direct consequence of \Cref{thm:1nonob}. Using the notation introduced in the proof of \Cref{thm:1nonob}, let $G = \{g_1,\ldots,g_{\oh_k}\}$ be a set of generators of $I$ of degree $k$, let $\{b_1,\ldots,b_{\uh_{k+1}}\}$ be a basis of $(R/I)_{k+1}$ and $\tilde{b}_j$ a homogeneous lifting of $b_j$ to $R$. The fibre $\pi_{\bh_I}^{-1}([I]) $ is described by the ideal
\[
 \left( g_i + \sum_{j=1}^{\uh_{k+1}} \alpha_{i,j} \widetilde b_j\ \middle\vert\ i=1,\ldots, \oh_{k}\right) + I_{\geqslant k+1}\subset \overline R.
\]
in the polynomial ring
\[ 
\overline{R}=\BC[x_1,\ldots,x_n]\otimes_{\BC}\BC[\alpha_{i,j}\ |\ i=1,\ldots, \oh_{k}\mbox{ and }j=1,\ldots,\uh_{k+1}]
\] 
 where coordinates $\alpha_{i,j}$, for $i=1,\ldots, \oh_{k}$ and $j=1,\ldots,\uh_{k+1}$ are the coordinates of the affine space $\mathbb{A}^{\oh_k \uh_{k+1}}$ that parametrises the fibre.
\end{proof}

\subsection{Nested configurations of 2-step ideals} \label{subsec:nest2stepbase}

Consider a pair of 2-step homogeneous ideals $I,J \subseteq R$ of respective orders $k_I$ and $k_J$,   and assume $k_I \leqslant k_J$. 

If $ k_J \geqslant k_I+2$, then $J$ is automatically contained in $I$. This gives the following isomorphism 
\[
\mathsf{T}_{[I,J]}^{\geqslant 0}\Hilb^{\bullet}\BA^n\quad \cong\quad
\mathsf{T}^{\geqslant 0}_{[I]}\Hilb^{\bullet}\BA^n\oplus\mathsf{T}^{\geqslant 0}_{[J]}\Hilb^{\bullet}\BA^n.
\] 

On the other hand, if $k_J = k_I+1$, then in general $J$ will not be contained in $I$ and the deformations of a nested pair $(I\supset J)$ are given by pairs of deformations of $I$ and $J$ preserving the inclusion. However, the inclusion is guaranteed if we consider deformations corresponding to positive tangent vectors.

\begin{theorem}\label{thm:1nonob-nested}
    Consider a nesting of 2-step homogeneous ideals $J\subset I \subset R$ of order $k_J$ and $k_I$ respectively such that $k_J-k_I>0$. Then,  there is an isomorphism  
\[
\mathsf{T}_{[I,J]}^{>0} \Hilb^\bullet \mathbb{A}^n \quad 
=\quad
\mathsf{T}_{[I,J]}^{=1} \Hilb^\bullet \mathbb{A}^n\quad \cong\quad
\Hom_{R}\big(I,R/I\big)_1 \oplus {\Hom}_{R}\big(J,R/J\big)_1.
\]
Moreover, all the tangent vectors of degree one are unobstructed. 
\end{theorem}

\begin{proof}
The case $k_J-k_I>1$ has been discussed at the beginning of the subsection. Thus, we put $k=k_I=k_J-1$. 

The first equality is a consequence of the identification of the positive tangent space $\mathsf{T}_{[I,J]}\Hilb^{(d_I,d_J)} \mathbb{A}^n$ with a linear subspace of the direct sum $\mathsf{T}_{[I]}^{>0}\Hilb^{d_I} \mathbb{A}^n\oplus \mathsf{T}_{[J]}^{>0} \Hilb^{ d_J } \mathbb{A}^n$ together with \Cref{thm:1nonob}. 

Given a basis $G=\{g_1,\ldots,g_{\oh_k^{(I)}}\}$ of $I_k$ and a basis $G' = \{g'_1,\ldots,g'_{\oh_{k+1}^{(J)}}\}$ of $J_{k+1}$, we associate to every tangent vector $\varphi \in \mathsf{T}_{[I,J]}^{=1} \Hilb^\bullet{\mathbb{A}^n}$ a pair of ideals in $R [\varepsilon]/\varepsilon^2$
\[
 (g_i+\varepsilon \widetilde{f}_i \ |\  i=1,\ldots,{\oh^{(I)}_{k}})+I_{\geqslant k+1},\qquad  (g'_i+\varepsilon \widetilde{f}'_i \ |\  i=1,\ldots,{\oh^{(J)}_{k+1}}) + J_{\geqslant k+2}
\]
where $\widetilde{f}_i \in R$ (resp.~$\widetilde{f}'_i \in R$) is a homogeneous lift of $f_i = \varphi(g_i)$ (resp.~$f'_i = \varphi(g'_i)$). By \Cref{thm:1nonob}, these two tangent vectors are unobstructed and lead to two flat families over $\mathbb{A}^1$ defined by the ideals in $R[t]$
\[
 I^+ = (g_i+t \widetilde{f}_i \ |\  i=1,\ldots,{\oh^{(I)}_{k}}) + I_{\geqslant k+1},\qquad J^+ = (g'_i+t \widetilde{f}'_i \ |\  i=1,\ldots,{\oh^{(J)}_{k+1}}) + J_{\geqslant k+2}.
\]

To prove that the tangent vector in $\varphi \in \mathsf{T}_{[I,J]}^{=1} \Hilb^\bullet{\mathbb{A}^n}$ is unobstructed, we show that the ideal $J^+$ is still contained in $I^+$, that is the nested pair $(I^+,J^+)$ describes a deformation of the nested point $(I,J)$. In fact,
 every $\tilde{f}'_j$ has degree $k+2$, so it is contained in $\mathfrak{m}^{k+2}  \subset I^+$. Moreover, by hypothesis $g'_j$ is contained in $I$ and it has degree $k+1$. Hence, $g_j' \in  I_{\geqslant k+1} \subset I^{+}$.
\end{proof}

    Although \Cref{thm:1nonob-nested} involves only nestings of two 2-step ideals, following the same logic one can show that the analogous statement holds for longer nestings so proving \Cref{thmintro:tgnested} from the introduction.

    \begin{corollary}\label{cor:1nonob-nested}
    Let  $[\underline I ]\in \Hilb^\bullet \BA^n$ be a nesting of 2-step homogeneous ideals. Denote by $k_i>0$ the order of the ideal $I^{(i)}$, for $i=1,\ldots,r$. Suppose that $k_{i+1}-k_i>0$, for all $i=1,\ldots,r-1$.  Then,  there is an isomorphism  
\[
\mathsf{T}_{[\underline I]}^{>0} \Hilb^\bullet \mathbb{A}^n =\mathsf{T}_{[\underline I]}^{=1} \Hilb^\bullet \mathbb{A}^n \cong \bigoplus_{i=1}^r \Hom_{R}\big(I^{(i)},R/I^{(i)}\big)_1.
\]
Moreover, all the tangent vectors of degree one are unobstructed. In particular, the initial ideal morphism is an affine bundle with fibres of dimension $\dim \mathsf{T}_{[\underline I]}^{>0} \Hilb^\bullet \mathbb{A}^n$.
    \end{corollary}

In the rest of the paper, we provide parametrisations of some irreducible components of Hilbert strata $H_{\bh}^n$ whose closed points correspond to 2-step ideals. By \Cref{thm:1nonob} and \Cref{thm:1nonob-nested}, we only need to give a parametrisation of some components of the homogeneous locus $\OH_{\bh}^n$. Moreover, since we are mainly interested in finding a lower bound for the dimension of $H_{\bh}^n$, it is sufficient to parametrise open subsets of $\OH_{\bh}^n$.  For this reason, we look for families of 2-step ideals with \textit{natural first anti-diagonal of the Betti table}, that is, ideals having  at most one non-zero graded Betti number in the first anti-diagonal of the Betti table. This is indeed an open condition on a family of modules having constant Hilbert function, cf.~\cite[Corollary 1.31]{AproduNagel}. Precisely, we focus on the graded Betti numbers $\beta_{0,k+1}$ and $\beta_{1,k+1}$ which are related by the equality
\[
\beta_{0,k+1} - \beta_{1,k+1} =  \oh_{k+1} - n\oh_{k}.
\] 
\begin{notation}\label{explanation-syzygies}
    For the sake of readability from now on, we denote by $\sss_{\bh}$ the quantity
    \[
   \sss_{\bh}= \oh_{k+1} - n\oh_{k}.
    \]
\end{notation}
Hence, if we consider a 2-step Hilbert function $\bh$ of order $k>0$ such that $\sss_{\bh} \geqslant 0$, then we expect the generic ideal in $\mathscr{H}^n_{\bh}$ to admit a minimal generating set consisting of $\beta_{0,k}=\oh_{k}$ elements of degree $k$ having no linear syzygies ($\beta_{1,k+1} = 0$) and $\beta_{0,k+1} = \sss_{\bh}$ generators in degree $k+1$. 
On the other hand, if $\sss_{\bh}< 0$, then the generic homogeneous ideals in $\mathscr{H}^n_{\bh}$ are expected to have $\beta_{0,k}=\oh_{k}$ generators of degree $k$ with $\beta_{1,k+1} =-\sss_{\bh}$ linear syzygies and no generators in degree $k+1$ since $\beta_{0,k+1} = 0$.

\subsection{The class of 2-step ideals without linear syzygies}\label{subsec:2-stepnosyz}
In this subsection, we study the loci in $\Hilb^{\bullet}\BA^n$ corresponding to 2-step ideals $I\subset R$ with $\sss_{\bh}\geqslant 0$, for some 2-step function $\bh\colon\BZ\to\BN$ having natural first anti-diagonal of the Betti table. In this setting, having at most one non-zero entry of the first anti-diagonal is equivalent to not having linear syzygies. We compute the dimension of these loci in the case they are not empty. We leave questions about the non-emptyness and the irreducibility of $\OH^n_{\bh}$ for further research. Let us consider  Hilbert functions 
\[
\bh = \left(1,n,\rrr_{2},\ldots,\rrr_{k-1},\uh_{k},\uh_{k+1}\right)
\]
such that $\sss_{\bh} \geqslant 0$.

\begin{theorem}\label{thm:homoloco}
Let $\bh = (1,n, \rrr_{2}, \ldots,\rrr_{k-1}, \uh_{k}, \uh_{k+1})$ be a 2-step Hilbert function of order $k>0$ such that $\sss_{\bh} \geqslant 0$.
Assume that there exists an ideal $[I]\in \mathscr{H}_{\bh}^n$ with $\beta_{1,k+1}(I)=0$. 
Then, there is a surjective morphism 
\[
\begin{tikzcd}[row sep=tiny]
 \OH_{\bh}^n\arrow[r,"\varphi"]&\Gr \left(\oh_{k},R_k\right)\\
 {[I]}\arrow[r,mapsto] &{[I_k]},
\end{tikzcd}
\]
whose generic fibre is isomorphic to $\Gr \left(\sss_{\bh},\rrr_{k+1}-n\oh_{k}\right)$.
\end{theorem}
\begin{proof}

    Consider the diagram
    \[
    \begin{tikzcd}
        \OZ_{\bh}^n\arrow[r,hook]\arrow[dr]&\OH_{\bh}^n\times \BA^n\arrow[d,"\pi"]\\&\OH_{\bh}^n.
    \end{tikzcd}
    \]
    where $\OZ_{\bh}^n\subset \OH_{\bh}^n\times \BA^n$ is the universal family of multigraded Hilbert scheme, see \cite{multigraded}.  We  now focus on the degree $k$ part of the push-forward of the universal sequence,
    \[
    \begin{tikzcd}
        0\arrow[r]&(\pi_*\OI_{\OZ_{\bh}^n})_k \arrow[r]& \OO_{\mathscr{H}_{\bh}^n}\otimes R_k\arrow[r]&(\pi_*\OO_{\OZ_{\bh}^n})_k\arrow[r] &0.
    \end{tikzcd}
    \]
This defines a family of  $\uh_{k}$-dimensional quotients of the free sheaf of rank $\rrr_{k}$. Thus, it provides a unique  morphism ${\varphi}\colon \OH_{\bh}^n\to \Gr(\oh_{k},R_{k} )$. 

Let $U\subset \mathscr{H}_{\bh}^n$ be the subset corresponding to ideals $I \subset R$ such that $\beta_{1,k+1}(I) =0$. It is open by the semicontinuity of the Betti numbers and it is not empty by the assumption. In order to show that the map $\varphi $ is surjective, it is enough to show that $\varphi(U)\subset \Gr(\oh_k,\rrr_k)$ is an open non-empty subset. Given a point $[V] \in \Gr(\oh_k,R_k)$, consider the homogeneous ideal $(V) \subset R$ generated by the vector subspace $V\subset R_k$. We have that $\dim_\BC (V)_{k+1} \leqslant \min\{\rrr_{k+1}, n (\dim_\BC (V)_{k} )\} = n\oh_k$. Let $U' \subset \Gr(\oh_k,R_k)$ be the open subset corresponding to points $[V]$  such that $\dim_\BC (V)_{k+1} = n \oh_k$. It is not empty, since $\varphi(U) \subset U'$. Now, given $[V] \in U'$, we have $\dim_\BC (V)_{k+1} = n \oh_k \leqslant \oh_{k+1}$, i.e.~$\dim_{\BC} (R/(V))_{k+1} =\rrr_{k+1}-n\oh_k$. Therefore, the fibre over each $[V]\in U'$ is 
\[
\Gr( \sss_\bh, (R/(V))_{k+1})\cong \Gr(\sss_\bh,\rrr_{k+1}-n\oh_k).\qedhere
\]
\end{proof} 

\begin{remark}\label{rem:disticomp}
Under the assumptions of \Cref{thm:homoloco}, there is an irreducible component $\mathscr{V}_{\bh}^n\subset\OH_{\bh}^n$  birational to a product of Grassmanians such that
\[
\dim \mathscr{V}_{{\bh}}^n=\oh_{k}\left({\mathsf{r}_{k}}-\oh_{k}\right) + \sss_{\bh}\big(\rrr_{k+1}-n\oh_{k} - (\oh_{k+1}-n\oh_{k})\big) = \oh_{k} \uh_{k} + \sss_{\bh}\uh_{k+1}.
\]  
This number agrees with the expected dimension of the tangent space of degree 0 to the Hilbert scheme $\Hilb^\bullet \mathbb{A}^n$, see \Cref{lem: tangent space}. Notice that the Theorem \ref{thm:homoloco} does not exclude the possible existence of  \textit{"exceeding vertical components"} of $\OH_{\bh}^n$, i.e.~components not dominating on the Grassmannian $\Gr \left(\oh_{k},R_k\right)$ via $\varphi$.
\end{remark}

\begin{corollary}\label{cor:full no syz}
Fix a 2-step Hilbert function $\bh$ of order $k$ such that $\sss_{\bh} \geqslant 0$. 
Assume there exists an ideal $[I]\in \mathscr{H}_{\bh}^n$ with $\beta_{1,k+1}(I)=0$. Then, there is a surjective morphism
\[
\begin{tikzcd}
    {H}_{\bh}^n\arrow[r]&\Gr \left(\oh_{k},R_k\right) 
\end{tikzcd}
\]
whose generic fibre is an  $\mathbb{A}^{\oh_{k} \uh_{k+1}}$-bundle over $\Gr \left(\sss_{\bh},\rrr_{k+1}-n\oh_{k}\right)$. 
\end{corollary}
\begin{proof}
    Combine \Cref{rem:1nonobstruct} and \Cref{thm:1nonob}.
\end{proof}
\subsection{The class of 2-step ideals with few linear syzygies}\label{subsec:2-stepfew}

Now, we focus on Hilbert functions 
\[
\bh = \left(1,n,\rrr_{2},\ldots,\rrr_{k-1},\uh_{k},\uh_{k+1}\right)
\]
such that $\sss_{\bh} < 0$. Suppose to have a homogeneous ideal $I\subset R$ with natural first anti-diagonal of the Betti table, i.e.~such that the ideal $I$ is generated in degree $k$ and possibly $k+2$. Fix a pair  $(\varphi,\boldsymbol{p})$,  where $\boldsymbol{p} = \{p_1,\ldots,p_{\oh_{k}}\}$ is a basis of $I_k$ and $\varphi : R(-k-1)^{\oplus-\sss_{\bh}} \to R(-k)^{\oplus\oh_{k}}$ is the syzygy matrix describing the linear syzygies among the generators $p_1, p_2,\ldots, p_{\oh_{k}}$. Let us interpret the map $\varphi$ as a $\BC$-linear homomorphism, and let us consider the restriction
\[
\begin{tikzcd}[column sep = 2cm]
R_k^{\oplus\oh_{k}}\arrow[r,"\varphi^T|_{R_k^{\oplus\oh_{k}}}"]&R_{k+1}^{\oplus-\sss_{\bh}}.
\end{tikzcd}
\]
We abuse of notation and we  denote it with the same symbol $\varphi^T=\varphi^T|_{R_k^{\oplus\oh_{k}}}\in \Hom_{\BC}(R^{\oplus\oh_{k}}_k,R^{\oplus-\sss_{\bh}}_{k+1})$. Note that  the element $\boldsymbol{p} \in R_k^{\oplus\oh_{k}}$ is contained in the kernel of $\varphi^T $ by construction. We want to emulate this argument in order to achieve a lower bound for the dimension of the locus parametrising 2-step ideals whose Hilbert function satisfies $\sss_{\bh}<0$. Recall that there is a natural inclusion $\Hom_R(R(k)^{\oplus\oh_{k}},R (k+1)^{\oplus-\sss_{\bh}})_0\subset \Hom_{\BC}(R^{\oplus\oh_{k}}_k,R^{\oplus-\sss_{\bh}}_{k+1})$ given by forgetting $R$-linearity. Then, we denote by $\mathscr{L}_{\bh} \subseteq \Hom_R(R(k)^{\oplus\oh_{k}},R (k+1)^{\oplus-\sss_{\bh}})_0$ the open subset corresponding to $\BC$-linear homomorphisms with maximal rank.  
Since we want the generic  element $\varphi \in \mathscr{L}_{\bh}$ to have non-trivial kernel, we assume
\begin{equation}\label{eq:fls}
    \oh_{k} > -\sss_{\bh} = n\oh_{k} - \oh_{k+1}\qquad\Leftrightarrow\qquad \oh_{k+1} > (n-1)\oh_{k},
\end{equation}
and we refer to 2-step ideals satisfying \eqref{eq:fls} as 2-step homogeneous ideals with \emph{few linear syzygies}. Note that, under the assumption $\oh_{k} > -\sss_{\bh}$ we have
\[
\dim R_k^{\oplus \oh_k}-\dim R_{k+1}^{\oplus-\sss_{\bh}} = ( \oh_k + \sss_{\bh})\binom{k+n-1}{n-1} - \sss_{\bh}\binom{k+n-1}{n-2}>0.
\] 
This framework is clearly more suitable for our purpose.

In order to understand the kernel $\ker \varphi$, for $\varphi\in\mathscr{L}_\bh$ generic, we act on $\varphi$ with row and column operations to obtain a sort of normal form representing the generic $\varphi\in\mathscr{L}_{\bh}$ as a matrix.

\begin{lemma}\label{lem:pseudo canonical form}
Consider a 2-step Hilbert function $\bh$ of order $k>0$, with $0 < -\sss_{\bh} < \oh_{k}$.  
    \begin{enumerate}[\it (i)]
    \item\label{it:1i} If $n(-\sss_{\bh}) \leqslant \oh_{k}$, then the generic  $\varphi \in \mathscr{L}_{\bh}$ can be reduced via row and column operations to
    \begin{center}\label{eq: normal form syzygy}
        \begin{tikzpicture}[xscale=0.4,yscale=0.4,decoration=brace]
        \node at (0,0) [] {\footnotesize $x_1$};
        \node at (1,0) [] {\footnotesize $x_2$};
        \node at (2,0) [] {\footnotesize $\cdots$};
        \node at (3,0) [] {\footnotesize $x_n$};
        \node at (4,0) [] {\footnotesize $0$};
        \node at (5.5,0) [] {\footnotesize $\cdots$};
        \node at (7,0) [] {\footnotesize $0$};

        \node at (9.5,0) [] {\footnotesize $\cdots$};

        \node at (12,0) [] {\footnotesize $0$};
        \node at (13.5,0) [] {\footnotesize $\cdots$};
        \node at (15,0) [] {\footnotesize $0$};

                \node at (16,0) [] {\footnotesize $0$};
        \node at (18,0) [] {\footnotesize $\cdots$};
        \node at (20,0) [] {\footnotesize $0$};

        \node at (0,-1) [] {\footnotesize $0$};
        \node at (1.5,-1) [] {\footnotesize $\cdots$};
        \node at (3,-1) [] {\footnotesize $0$};
        
        \node at (4,-1) [] {\footnotesize $x_1$};
        \node at (5,-1) [] {\footnotesize $x_2$};
        \node at (6,-1) [] {\footnotesize $\cdots$};
        \node at (7,-1) [] {\footnotesize $x_n$};

        \node at (9.5,-1) [] {\footnotesize $\cdots$};

        \node at (12,-1) [] {\footnotesize $0$};
        \node at (13.5,-1) [] {\footnotesize $\cdots$};
        \node at (15,-1) [] {\footnotesize $0$};

                \node at (16,-1) [] {\footnotesize $0$};
        \node at (18,-1) [] {\footnotesize $\cdots$};
        \node at (20,-1) [] {\footnotesize $0$};

        \node at (9.4,-1.8) [rotate=20] {\footnotesize $\ddots$};

        \node at (16,-1.8) [] {\footnotesize $\vdots$};
        \node at (20,-1.8) [] {\footnotesize $\vdots$};

        \node at (0,-3) [] {\footnotesize $0$};
        \node at (1.5,-3) [] {\footnotesize $\cdots$};
        \node at (3,-3) [] {\footnotesize $0$};
        
        \node at (4,-3) [] {\footnotesize $0$};
        \node at (5.5,-3) [] {\footnotesize $\cdots$};
        \node at (7,-3) [] {\footnotesize $0$};

        \node at (9.5,-3) [] {\footnotesize $\cdots$};

        \node at (12,-3) [] {\footnotesize $x_1$};
        \node at (13,-3) [] {\footnotesize $x_2$};
        \node at (14,-3) [] {\footnotesize $\cdots$};
        \node at (15,-3) [] {\footnotesize $x_n$};

        \node at (16,-3) [] {\footnotesize $0$};
        \node at (18,-3) [] {\footnotesize $\cdots$};
        \node at (20,-3) [] {\footnotesize $0$};

        \draw [thin] (15.5,0.5) -- (15.5,-3.5);

        \draw [thick,xshift=-0.2cm] (-0.3,0.5) -- (-0.5,0.5) -- (-0.5,-3.5) -- (-0.3,-3.5);

        \draw [thick,xshift=0.2cm] (20.3,0.5) -- (20.5,0.5) -- (20.5,-3.5) -- (20.3,-3.5);

        \draw[decorate] (-0.25,0.75) -- node[above]{\tiny $n(n\oh_{k}-\oh_{k+1})$ columns} (15.25,0.75);

        \draw[decorate] (15.75,0.75) -- node[above]{\tiny $\oh_{k} - n(n\oh_{k}-\oh_{k+1})$ columns} (20.25,0.75);

        \draw[decorate] (-1,-3.5) -- node[above,rotate=90]{\tiny $n\oh_{k}-\oh_{k+1}$ rows} (-1,0.5);
        
        \end{tikzpicture}.
    \end{center}
    \item Set $\oh_{k} = n q + r,\ 0 \leqslant r < n$. If $n(-\sss_{\bh}) > q$, then the generic $\varphi \in \mathscr{L}_{\bh}$ can be reduced via row and column operations to
    \begin{center}
    \begin{tikzpicture}[xscale=0.4,yscale=0.4,decoration=brace]
        \node at (0,0) [] {\footnotesize $x_1$};
        \node at (1,0) [] {\footnotesize $x_2$};
        \node at (2,0) [] {\footnotesize $\cdots$};
        \node at (3,0) [] {\footnotesize $x_n$};
        \node at (4,0) [] {\footnotesize $0$};
        \node at (5.5,0) [] {\footnotesize $\cdots$};
        \node at (7,0) [] {\footnotesize $0$};

        \node at (9.5,0) [] {\footnotesize $\cdots$};

        \node at (12,0) [] {\footnotesize $0$};
        \node at (13.5,0) [] {\footnotesize $\cdots$};
        \node at (15,0) [] {\footnotesize $0$};

                \node at (16,0) [] {\footnotesize $0$};
        \node at (18,0) [] {\footnotesize $\cdots$};
        \node at (20,0) [] {\footnotesize $0$};

        \node at (0,-1) [] {\footnotesize $0$};
        \node at (1.5,-1) [] {\footnotesize $\cdots$};
        \node at (3,-1) [] {\footnotesize $0$};
        
        \node at (4,-1) [] {\footnotesize $x_1$};
        \node at (5,-1) [] {\footnotesize $x_2$};
        \node at (6,-1) [] {\footnotesize $\cdots$};
        \node at (7,-1) [] {\footnotesize $x_n$};

        \node at (9.5,-1) [] {\footnotesize $\cdots$};

        \node at (12,-1) [] {\footnotesize $0$};
        \node at (13.5,-1) [] {\footnotesize $\cdots$};
        \node at (15,-1) [] {\footnotesize $0$};

                \node at (16,-1) [] {\footnotesize $0$};
        \node at (18,-1) [] {\footnotesize $\cdots$};
        \node at (20,-1) [] {\footnotesize $0$};

        \node at (9.4,-1.8) [rotate=20] {\footnotesize $\ddots$};

        \node at (16,-1.8) [] {\footnotesize $\vdots$};
        \node at (20,-1.8) [] {\footnotesize $\vdots$};

        \node at (0,-3) [] {\footnotesize $0$};
        \node at (1.5,-3) [] {\footnotesize $\cdots$};
        \node at (3,-3) [] {\footnotesize $0$};
        
        \node at (4,-3) [] {\footnotesize $0$};
        \node at (5.5,-3) [] {\footnotesize $\cdots$};
        \node at (7,-3) [] {\footnotesize $0$};

        \node at (9.5,-3) [] {\footnotesize $\cdots$};

        \node at (12,-3) [] {\footnotesize $x_1$};
        \node at (13,-3) [] {\footnotesize $x_2$};
        \node at (14,-3) [] {\footnotesize $\cdots$};
        \node at (15,-3) [] {\footnotesize $x_n$};

        \node at (16,-3) [] {\footnotesize $0$};
        \node at (18,-3) [] {\footnotesize $\cdots$};
        \node at (20,-3) [] {\footnotesize $0$};

        \node at (10,-5) [] {\footnotesize $\ell_{i,j}$};

        \node at (7,-5) [] {\footnotesize $\cdots$};
\node at (13,-5) [] {\footnotesize $\cdots$};

        \node at (10,-5.8) [] {\footnotesize $\vdots$};
        \node at (10,-3.8) [] {\footnotesize $\vdots$};

        \draw [thin] (15.5,0.5) -- (15.5,-6.5);
        \draw [thin] (-0.25,-3.5) -- (20.25,-3.5);

        \draw [thick,xshift=-0.2cm] (-0.3,0.5) -- (-0.5,0.5) -- (-0.5,-6.5) -- (-0.3,-6.5);

        \draw [thick,xshift=0.2cm] (20.3,0.5) -- (20.5,0.5) -- (20.5,-6.5) -- (20.3,-6.5);

        \draw[decorate] (-0.25,0.75) -- node[above]{\tiny $qn$ columns} (15.25,0.75);

        \draw[decorate] (15.75,0.75) -- node[above]{\tiny $\oh_{k}-qn$ columns} (20.25,0.75);

        \draw[decorate] (-1,-3.25) -- node[above,rotate=90]{\tiny $q$ rows} (-1,0.25);

        \draw[decorate] (-1,-6.25) -- node[above,rotate=90]{\tiny \parbox{1.5cm}{\centering $-\sss_{\bh}-q$\\ rows}} (-1,-3.75);

        \end{tikzpicture},
    \end{center}
    where, for all $i=1,\ldots,-\sss_{\bh}-q$ and $j=1,\ldots , qn$, the elements $\ell_{i,j}\in R_1$ are linear forms. 
    \end{enumerate} 
\end{lemma}
\begin{proof}
    Straightforward, by applying row and column operations to the matrix representing $\varphi$ with respect to the canonical basis of the free $R$-module $R^{\oplus\oh_k}$.
\end{proof}

    Let $\bh$ be a 2-step Hilbert function such that $0<-\sss_{\bh}<\oh_k$.   Consider then the incidence correspondence 
\[
\mathscr{K}_{\bh} = \Set{ (\varphi,\boldsymbol{p}) \in \mathscr{L}_{\bh} \times R_k^{\oplus\oh_{k}} | \boldsymbol{p} \in \ker \varphi } \subset \mathscr{L}_{\bh} \times R_k^{\oplus\oh_{k}}.
\]     
Note that  $\mathscr{K}_\bh$ is  integral. Indeed, irreducibility follows from the irreducibility of $\mathscr{L}_{\bh}$ and from the fact that the fibres of the restriction to $\mathscr{K}_{\bh}$ of the projection onto $\mathscr{L}_\bh$ are vector spaces of the same dimension by definition of $\mathscr{L}_{\bh}$. Reducedness is a consequence of the fact that $\mathscr{K}_{\bh}$ is cut out, in $\mathscr{L}_{\bh} \times R_k^{\oplus\oh_{k}}$, by  equations linear  in the coordinates of $R_k^{\oplus\oh_{k}}$.

\begin{theorem}\label{thm: homo loco syz} 
Let $\bh$ be a 2-step Hilbert function such that $0 < -\sss_{\bh} < \oh_{k}$. Assume that the generic morphism $\varphi\in\mathscr{L}_\bh$ is not injective. Then, there exists  a  rational map $\psi_{\bh}:\mathscr{K}_{\bh}\dashrightarrow \mathscr{H}_{\bh}^n$ which, on closed points, associates the generic pair $(\varphi,\boldsymbol{p})$ to the ideal   $I_{\boldsymbol{p}} = \left(\boldsymbol{p}\right) + \Fm^{k+2}$.
\end{theorem}
\begin{proof}
    In a non-empty open subset $U \subset \mathscr{K}_{\bh}$, the polynomials $\boldsymbol{p}=(p_1,\ldots,p_{\oh_{k}})$ are linearly independent and satisfy $-\sss_{\bh}$ independent linear syzygies. Thus
    \[
    \dim_\BC \left(I_{\boldsymbol{p}}\right)_k = \oh_{k} \quad \text{and}\quad \dim_\BC \left(I_{\boldsymbol{p}}\right)_{k+1} = n\oh_{k} -(- \sss_{\bh}) = \oh_{k+1}. 
    \]
     Thus, this is a flat (recall that $\mathscr{K}_{\bh}$ is reduced) family of 2-step ideals with base $U$, so providing the rational map $\psi_{\bh}$.  
\end{proof}

\begin{remark}\label{rem:veryfwok} We remark that the condition \textit{(\ref{it:1i})} in \Cref{lem:pseudo canonical form} ensures that the generic morphism $\varphi\in\mathscr{L}_\bh$ is surjective  and non-injective. In this case, we can compute the dimension of $\mathscr{K}_\bh$. We have 
\[
\dim \mathscr{K}_{\bh}\ =\ \dim_\BC \mathscr{L}_{\bh} + \dim_{\BC}\ker \varphi\ =\ n(-\sss_{\bh})\oh_{k}  + \dim_{\BC}\ker \varphi,
\]
where $\varphi$ is a generic morphism in $\mathscr{L}_{\bh}$. From the short exact sequence
\[
\begin{tikzcd}
    0 \arrow[r]& \ker \varphi\arrow[r]& R_k^{\oplus\oh_{k}} \arrow[r,"\varphi"]& R_{k+1}^{\oplus-\sss_{\bh}} \arrow[r]& 0
\end{tikzcd}
\]
one has
\[
\dim_{\BC} (\ker \varphi) = \oh_{k}\rrr_{k} + \sss_{\bh}\rrr_{k+1} ,
\]
and hence
\begin{equation}
    \label{eq:dimK}\dim\mathscr{K}_{\bh}=  n(-\sss_{\bh})\oh_{k} +\oh_{k}\rrr_{k} + \sss_{\bh}\rrr_{k+1}. 
\end{equation}

We refer to homogeneous 2-step ideals satisfying \textit{(\ref{it:1i})} as ideals with \emph{very few linear syzygies}. The name is motivated by the fact that
\[
0 < -\sss_{\bh} \leqslant \frac{1}{n}\oh_{k}\quad\Leftrightarrow\quad n\oh_k>\oh_{k+1} \geqslant \dfrac{n^2-1}{n}\oh_{k} .
\] 
Moreover, we underline that in this setting a minimal set of generators of a generic ideal in $\psi_\bh(\mathscr{K}_{\bh})\subset\mathscr{H}_{\bh}^n$ can be obtained as the union of a subset of cardinality $-\sss_{\bh}$ of the kernel of the following morphism 
    \[
        R_k^{\oplus n} \xrightarrow{\ [ x_1\ \cdots\ x_n ]\ } R_{k+1}
    \]
    with $\oh_{k} - n(-\sss_{\bh})$ further independent polynomials of degree $k$.
\end{remark} 

Now, we want to bound from below the dimension of the homogeneous locus $\mathscr{H}^n_{\bh}$ of the Hilbert stratum by   comparing $\dim\mathscr{K}_{\bh}$ and the dimension of the fibre $\psi_{\bh}^{-1}([I])$ at a generic point  $[I]\in \psi_\bh(\mathscr{K}_\bh)$. We do this in the case of very few linear syzygies, where we are able to compute the dimension of the generic fibre of $\psi_{\bh}$ as explained in \Cref{rem:veryfwok}.

\begin{corollary}\label{cor:veryfew}
    Let $\bh$ be a 2-step Hilbert function with very few linear syzygies. Then, the following inequality holds
    \[
    \dim \mathscr{H}^n_{\bh} \geqslant \oh_{k}(\rrr_{k}-\oh_{k}) + \sss_{\bh}( \rrr_{k+1} - \oh_{k+1}).
    \]
\end{corollary}
\begin{proof}    
We start by computing the dimension of the generic fibre of $\psi_{\bh}$. Given $[I]\in\psi_\bh(\mathscr{K}_\bh)\subset \mathscr{H}^n_{\bh}$,  we can act with the linear group $\GL(\oh_{k})$ to change the basis of $I_k$. Explicitly, for every $M \in \GL(\oh_{k})$, the pair $(\varphi,\boldsymbol{p})$ is in the fibre over  $[I]$ if and only if $(\varphi M^{-1},M\boldsymbol{p})$ is. In fact
\[
0 = \varphi\boldsymbol{p} = (\varphi M^{-1}) (M\boldsymbol{p}) \quad \Rightarrow \quad M\boldsymbol{p} \in \ker (\varphi M^{-1}).
\]
In this way, the general linear group acts on $\mathscr{L}_{\bh}$ via column operations of the matrix representing $\varphi$. Nevertheless, we can also act on $\mathscr{L}_{\bh}$ via row operations. Summarising, a pair $(\varphi,\boldsymbol{p})$ belongs to the fibre $\psi_{\bh}^{-1}([I])$  over $[I]$ if and only if, for every pair $(M_1,M_2) \in  \GL(-\sss_{\bh})\times \GL(\oh_{k})$, we also have $(M_1 \varphi M_2^{-1},M_2 \boldsymbol{p})\in\psi_{\bh}^{-1}([I])$, because
\[
0 = \varphi \boldsymbol{p} = M_1 \varphi \boldsymbol{p} = (M_1\varphi M_2^{-1}) (M_2 \boldsymbol{p}).
\]
By the normal form in \Cref{lem:pseudo canonical form}\emph{(\ref{eq: normal form syzygy})} the action of $\GL(-\sss_{\bh})\times \GL(\oh_{k})$ is faithful on some open, hence this gives, for a generic fibre $F\subset \mathscr{K}_{\bh}$ of $\psi_{\bh}$, the equality 
\[
\dim F = \dim  \GL(-\sss_{\bh})+\dim \GL(\oh_{k})=\sss_{\bh}^2  +  \oh_{k}^2,
\]
and we deduce from \eqref{eq:dimK} that 
\begin{equation}
\begin{split}
    \dim \mathscr{H}^n_{\bh} &{} \geqslant   \dim \psi_{\bh}(\mathscr{K}_{\bh})= \dim \mathscr{K}_{\bh} - \dim F = {}\\
    &{}=  n(-\sss_{\bh})\oh_{k} +\oh_{k}\rrr_{k} + \sss_{\bh}\rrr_{k+1} - (\sss_{\bh}^2 +  \oh_{k}^2) ={}\\ 
    &{}= \oh_{k}(\rrr_{k}-\oh_{k}) + \sss_{\bh}( \rrr_{k+1} - \oh_{k+1}). \qedhere
\end{split}    
\end{equation}
\end{proof}

\begin{remark} 
We conclude this subsection by noticing that the bound in \Cref{cor:veryfew} agrees with the lower bound for the dimension of the degree 0 part of the tangent space given in \Cref{lem: tangent space}. This in turn, provides the expected dimension of the homogeneous locus of the Hilbert stratum, see \cite{multigraded}. 

In the case of few linear syzygies, one can still look for a lower bound to the dimension of the corresponding Hilbert stratum. If the generic morphism $\varphi\in\mathscr{L}_{\bh}$ is not surjective, then the dimension of  $\ker\varphi$ cannot be deduced theoretically and it must be computed explicitly. We do this in \Cref{sec:red3fold} to find examples in dimension 3.  
\end{remark}

\subsection{Estimate of the number of \texorpdfstring{$(-1)$}{}-tangent vectors}\label{subsec:TNT2step}
From the proof of \Cref{lem: tangent space} we get, for a homogeneous 2-step ideal $I \subset R$, the inequality
\begin{equation}\label{eq:stima-1}
    \dim_{\BC} \mathsf{T}^{=-1}_{[I]}\Hilb^\bullet \mathbb{A}^n \geqslant \max\left\{0, \beta_{0,k}\rrr_{k-1} + (\beta_{0,k+1}-\beta_{1,k+1})(\rrr_{k}-\oh_{k}) + (\beta_{0,k+2}-\beta_{1,k+2})(\rrr_{k+1}-\oh_{k+1}) \right\},
\end{equation}
where
\[
\beta_{0,k} = \oh_{k}\quad\text{and}\quad \beta_{0,k+1}-\beta_{1,k+1} = \sss_{\bh}.
\]
The third summand involves the difference $\beta_{0,k+2}-\beta_{1,k+2}$ that is only a part of the coefficient in degree $k+2$ of the Hilbert series of the ideal $I$ deduced from the resolution \eqref{eq: resolution I}, see \Cref{lem: general properties Betti}. Let us put
\[
\ttt_{\bh} = \rrr_{k+2} - n \oh_{k+1} + \tbinom{n}{2} \oh_{k}.
\]
Then, we have $\beta_{0,k+2}-\beta_{1,k+2} = \ttt_{\bh} - \beta_{2,k+2}$ and we can rewrite the lower bound in \eqref{eq:stima-1} as
\[
\dim_{\BC} \mathsf{T}^{=-1}_{[I]}\Hilb^\bullet \mathbb{A}^n \geqslant \max\left\{0, \rrr_{k-1}\oh_{k} + \sss_{\bh}(\rrr_{k}-\oh_{k}) + (\ttt_{\bh}-\beta_{2,k+2})(\rrr_{k+1}-\oh_{k+1}) \right\}.
\]

We make the following observations: 
\begin{enumerate}[1.]
    \item for a given pair $(\oh_k,\oh_{k+1})$, the greater the number $\beta_{2,k+2}$ of second-order linear syzygies is, the lower the bound of the dimension of the space of $-1$-tangent vectors is; 
    \item the cases $\oh_{k+1} \geqslant \frac{n^2-1}{2}\oh_{k}$, which are covered by \Cref{thm:homoloco} and \Cref{thm: homo loco syz} (no or very few linear syzygies). In this setting the graded Betti number $\beta_{2,k+2}$ is zero for a generic ideal. In fact, in the no linear syzygies case, the vanishing $\beta_{1,k+1}=0$ implies $\beta_{2,k+2} = 0$. On the other hand, having very few linear syzygies provides injectivity of the morphism differential $ R(-k-1)^{\oplus\beta_{1,k+1}} \to R(-k)^{\oplus\beta_{0,k}}$ in \eqref{eq: resolution I}.
\end{enumerate}
  We draw in \Cref{fig:geography 2-step ideals}  the subdivision of 2-step Hilbert functions of order $k$ according to the number of linear syzygies and to the Betti numbers in degree $k$, $k+1$ and $k+2$.

\begin{figure}[!ht]
    \centering

    \begin{tikzpicture}
    
    \begin{scope}[scale=0.3]
        \fill [fill=black!20,] (0,0) -- (0,36) -- (12,36) -- cycle;
\fill [fill=black!60,] (0,0) -- (12,36) -- (27/2,36) -- cycle;
\fill [fill=black!40,] (0,0) -- (18,36) -- (27/2,36) -- cycle;
\fill [pattern={Lines[angle=45,distance=3pt,line width=0.05pt]},opacity=0.5,pattern color=white] (0,15) -- (21,36) -- (0,36) -- cycle;
\fill [pattern={Lines[angle=135,distance=3pt,line width=0.05pt]},opacity=0.5,pattern color=white] (0,15) -- (21,36) -- (28,0) -- (28,36) -- (0,0) -- cycle;

\fill [fill=white] (0,0) -- (1,3) -- (2,5) -- (3,6) -- (4,8) -- (5,9) -- (6,10) -- (7,12) -- (8,13) -- (9,14) -- (10,15) -- (11,17) -- (12,18) -- (13,19) -- (14,20) -- (15,21) -- (16,23) -- (17,24) -- (18,25) -- (19,26) -- (20,27) -- (21,28) -- (22,30) -- (23,31) -- (24,32) -- (25,33) -- (26,34) -- (27,35) -- (28,36) -- (28,0) -- cycle;

\draw [very thick,decorate,decoration={saw,amplitude=1pt,segment length=8pt},densely dotted] (0,0) -- (1,3) -- (2,5) -- (3,6) -- (4,8) -- (5,9) -- (6,10) -- (7,12) -- (8,13) -- (9,14) -- (10,15) -- (11,17) -- (12,18) -- (13,19) -- (14,20) -- (15,21) -- (16,23) -- (17,24) -- (18,25) -- (19,26) -- (20,27) -- (21,28) -- (22,30) -- (23,31) -- (24,32) -- (25,33) -- (26,34) -- (27,35) -- (28,36);
\draw [thick] (0,0) --node[left]{$\oh_{k+1}$} (0,36) --node[above,xshift=12pt]{$\oh_{k}$} (28,36);
       
\draw (-1/3,-1) -- (12+2/3,38) node[above,rotate=71,xshift=-4pt]{\small $\sss_{\bh} = 0$};

\draw (-1,14) -- (23,38) node[above,rotate=45,xshift=-12pt]{\small $\ttt_{\bh} = 0$};

\node at (0,0) [left] {\footnotesize $0$};
\node at (0,36) [left] {\footnotesize $\rrr_{k+1}$};
\node at (0,36) [above] {\footnotesize $0$};
\node at (28,36) [above] {\footnotesize $\rrr_{k}$};
\end{scope}

\begin{scope}[xscale=1.,yscale=0.5,shift={(5,7)}]
        \node at (1.5,2) [] {\parbox{3.cm}{\centering\small No linear syzygies\\ \footnotesize  $\sss_{\bh} \geqslant 0$}};
        \node at (4.5,2) [] {\parbox{3.cm}{\centering\small Very few linear\\[-4pt] syzygies\\ \footnotesize  $0 < -\sss_{\bh} \leqslant \frac{1}{n}\oh_{k} $}}; 

        \node at (7.5,2) [] {\parbox{3.cm}{\centering \footnotesize  $ -\sss_{\bh} > \frac{1}{n}\oh_{k}$ \\[4pt] \scriptsize $\beta_{0,k+2} - \beta_{1,k+2} = \ttt_{\bh}-\beta_{2,k+2}$}};

        \fill [fill=black!20,] (0,0.5) rectangle (3,-8.5);
        \fill [fill=black!60,line width=0pt] (3,0.5) rectangle (6,-8.5);
        \fill[pattern={Lines[angle=45,distance=3pt,line width=0.05pt]},opacity=0.5,pattern color=white] (0,0.5) rectangle (9,-4);
        \fill[pattern={Lines[angle=135,distance=3pt,line width=0.05pt]},opacity=0.5,pattern color=white] (0,-4) rectangle (9,-8.5);
        
        \draw (0,3.5) -- (0,-8.5);
        \draw (3,3.5) -- (3,-8.5);
        \draw (6,3.5) -- (6,-8.5);
        \draw (9,3.5) -- (9,-8.5);
        
        \draw (-2,0.5) -- (9,0.5);
        \draw (-2,-4) -- (9,-4);
        \draw (-2,-8.5) -- (9,-8.5);

        \draw (-2,0.5) -- (-2,-8.5);
        
        \draw (0,3.5) -- (9,3.5);

        \node at (-1,-1.75) [] {\footnotesize $\ttt_{\bh} - \beta_{2,k+2} < 0$};
        \node at (-1,-6.25) [] {\footnotesize $\ttt_{\bh} - \beta_{2,k+2} \geqslant 0$};

        \draw [ultra thin,xshift=-0.45cm,yshift=0.75cm] (3.2,-0.7) -- (0.7,-0.7) -- (0.7,-4.3) -- (1.6,-4.3);
        \draw [fill=white,xshift=-0.45cm,yshift=0.75cm] (0.8,-0.9) rectangle (1.5,-1.9);
        \node [xshift=-0.45cm+1.5pt,yshift=0.375cm-0.5pt] at (1.15,-1.4) [] {\scriptsize $\oh_{k}$};
        \draw [fill=white,xshift=-0.45cm,yshift=0.75cm] (0.8,-2) rectangle (1.5,-3);
        \node [xshift=-0.45cm+1.5pt,yshift=0.375cm-0.5pt] at (1.15,-2.5) [] {\scriptsize $\sss_{\bh}$};
        \draw [fill=black,xshift=-0.45cm,yshift=0.75cm] (1.125,-3.55) rectangle (1.175,-3.65);
        
        \draw [xshift=0.35cm,fill=black,yshift=0.75cm] (1.125,-1.35) rectangle (1.175,-1.45);
        \draw [xshift=0.35cm,fill=white,yshift=0.75cm] (0.8,-2) rectangle (1.5,-3);
        \node [xshift=0.35cm+1pt,yshift=0.375cm-0.5pt] at (1.15,-2.5) [] {\scriptsize ${-}\ttt_{\bh}$};
        
        \draw [xshift=1.15cm,fill=black,yshift=0.75cm] (1.125,-1.35) rectangle (1.175,-1.45);

        \draw [ultra thin,xshift=-0.45cm,yshift=-3.75cm] (3.2,-0.7) -- (0.7,-0.7) -- (0.7,-4.3) -- (1.6,-4.3);
        \draw [fill=white,xshift=-0.45cm,yshift=-3.75cm] (0.8,-0.9) rectangle (1.5,-1.9);
        \node [xshift=-0.45cm+1.5pt,yshift=-2.25cm+0.375cm-0.5pt] at (1.15,-1.4) [] {\scriptsize $\oh_{k}$};
        \draw [fill=white,xshift=-0.45cm,yshift=-3.75cm] (0.8,-2) rectangle (1.5,-3);
        \node [xshift=-0.45cm+1.5pt,yshift=-2.25cm+0.375cm-0.5pt] at (1.15,-2.5) [] {\scriptsize $\sss_{\bh}$};
        \draw [xshift=-0.45cm,yshift=-3.75cm,fill=white] (0.8,-3.1) rectangle (1.5,-4.1);
        \node[xshift=-0.45cm,yshift=-2.25cm+0.375cm-0.5pt] at (1.15,-3.6) [] {\scriptsize $\ttt_{\bh}$};
        
        \draw [xshift=0.35cm,fill=black,yshift=-3.75cm] (1.125,-1.35) rectangle (1.175,-1.45);
        \draw [xshift=0.35cm,fill=black,yshift=-3.75cm] (1.125,-2.45) rectangle (1.175,-2.55);
        
        \draw [xshift=1.15cm,fill=black,yshift=-3.75cm] (1.125,-1.35) rectangle (1.175,-1.45);

        \draw [ultra thin,xshift=-0.45cm+3cm,yshift=0.75cm] (3.2,-0.7) -- (0.7,-0.7) -- (0.7,-4.3) -- (1.6,-4.3);
        \draw [fill=white,xshift=-0.45cm+3cm,yshift=0.75cm] (0.8,-0.9) rectangle (1.5,-1.9);
        \node [xshift=-0.45cm+3cm+1.5pt,yshift=0.375cm-0.5pt] at (1.15,-1.4) [] {\scriptsize $\oh_{k}$};
        \draw [fill=black,xshift=-0.45cm+3cm,yshift=0.75cm] (1.125,-2.45) rectangle (1.175,-2.55);
        \draw [fill=black,xshift=-0.45cm+3cm,yshift=0.75cm] (1.125,-3.55) rectangle (1.175,-3.65);
        
        \draw [xshift=0.35cm+3cm,fill=white,yshift=0.75cm] (0.8,-0.9) rectangle (1.5,-1.9);
        \node [xshift=0.35cm+3cm+1pt,yshift=0.375cm-0.5pt] at (1.15,-1.4) [] {\scriptsize ${-}\sss_{\bh}$};
        \draw [xshift=0.35cm+3cm,fill=white,yshift=0.75cm] (0.8,-2) rectangle (1.5,-3);
        \node [xshift=0.35cm+3cm+1pt,yshift=0.375cm-0.5pt] at (1.15,-2.5) [] {\scriptsize ${-}\ttt_{\bh}$};
        
        \draw [xshift=1.15cm+3cm,fill=black,yshift=0.75cm] (1.125,-1.35) rectangle (1.175,-1.45);

        \draw [ultra thin,xshift=-0.45cm+3cm,yshift=-3.75cm] (3.2,-0.7) -- (0.7,-0.7) -- (0.7,-4.3) -- (1.6,-4.3);
        \draw [fill=white,xshift=-0.45cm+3cm,yshift=-3.75cm] (0.8,-0.9) rectangle (1.5,-1.9);
        
        \node [xshift=-0.45cm+3cm+1.5pt,yshift=-2.25cm+0.375cm-0.5pt] at (1.15,-1.4) [] {\scriptsize $\oh_{k}$};
        \draw [fill=black,xshift=-0.45cm+3cm,yshift=-3.75cm] (1.125,-2.45) rectangle (1.175,-2.55);
        \draw [xshift=-0.45cm+3cm,yshift=-3.75cm,fill=white] (0.8,-3.1) rectangle (1.5,-4.1);
        \node[xshift=-0.45cm+3cm,yshift=-2.25cm+0.375cm-0.5pt] at (1.15,-3.6) [] {\scriptsize $\ttt_{\bh}$};
        
        \draw [xshift=0.35cm+3cm,fill=white,yshift=-3.75cm] (0.8,-0.9) rectangle (1.5,-1.9);
        \node [xshift=0.35cm+3cm+1pt,yshift=-2.25cm+0.375cm-0.5pt] at (1.15,-1.4) [] {\scriptsize ${-}\sss_{\bh}$};
        \draw [xshift=0.35cm+3cm,fill=black,yshift=-3.75cm] (1.125,-2.45) rectangle (1.175,-2.55);
        
        \draw [xshift=1.15cm+3cm,fill=black,yshift=-3.75cm] (1.125,-1.35) rectangle (1.175,-1.45);

        \draw [ultra thin,xshift=-0.45cm+6cm,yshift=0.75cm] (3.2,-0.7) -- (0.7,-0.7) -- (0.7,-4.3) -- (1.6,-4.3);
        \draw [fill=white,xshift=-0.45cm+6cm,yshift=0.75cm] (0.8,-0.9) rectangle (1.5,-1.9);
        \node [xshift=-0.45cm+6cm+1.5pt,yshift=0.375cm-0.5pt] at (1.15,-1.4) [] {\scriptsize $\oh_{k}$};
        \draw [fill=black,xshift=-0.45cm+6cm,yshift=0.75cm] (1.125,-2.45) rectangle (1.175,-2.55);
        \draw [fill=black,xshift=-0.45cm+6cm,yshift=0.75cm] (1.125,-3.55) rectangle (1.175,-3.65);
        
        \draw [xshift=0.35cm+6cm,fill=white,yshift=0.75cm] (0.8,-0.9) rectangle (1.5,-1.9);
        \node [xshift=0.35cm+6cm+1pt,yshift=0.375cm-0.5pt] at (1.15,-1.4) [] {\scriptsize ${-}\sss_{\bh}$};
        \draw [xshift=0.35cm+6cm,fill=white,yshift=0.75cm] (0.8,-2) rectangle (1.5,-3);
        \node [xshift=0.35cm+6cm,yshift=0.375cm-0.5pt] at (1.15,-2.5) [] {\scriptsize $\beta_{1,k+2}$};
        
        \draw [xshift=1.15cm+6cm,fill=white,yshift=0.75cm] (0.8,-0.9) rectangle (1.5,-1.9);
        \node [xshift=1.15cm+6cm,yshift=0.375cm-0.5pt] at (1.15,-1.4) [] {\scriptsize $\beta_{2,k+2}$};

        \draw [ultra thin,xshift=-0.45cm+6cm,yshift=-3.75cm] (3.2,-0.7) -- (0.7,-0.7) -- (0.7,-4.3) -- (1.6,-4.3);
        \draw [fill=white,xshift=-0.45cm+6cm,yshift=-3.75cm] (0.8,-0.9) rectangle (1.5,-1.9);
        
        \node [xshift=-0.45cm+6cm+1.5pt,yshift=-2.25cm+0.375cm-0.5pt] at (1.15,-1.4) [] {\scriptsize $\oh_{k}$};
        \draw [fill=black,xshift=-0.45cm+6cm,yshift=-3.75cm] (1.125,-2.45) rectangle (1.175,-2.55);
        \draw [xshift=-0.45cm+6cm,yshift=-3.75cm,fill=white] (0.8,-3.1) rectangle (1.5,-4.1);
        \node[xshift=-0.45cm+6cm,yshift=-2.25cm+0.375cm-0.5pt] at (1.15,-3.6) [] {\scriptsize $\beta_{0,k+2}$};
        
        \draw [xshift=0.35cm+6cm,fill=white,yshift=-3.75cm] (0.8,-0.9) rectangle (1.5,-1.9);
        \node [xshift=0.35cm+6cm+1pt,yshift=-2.25cm+0.375cm-0.5pt] at (1.15,-1.4) [] {\scriptsize ${-}\sss_{\bh}$};
        \draw [xshift=0.35cm+6cm,fill=black,yshift=-3.75cm] (1.125,-2.45) rectangle (1.175,-2.55);
        
        \draw [xshift=1.15cm+6cm,fill=white,yshift=-3.75cm] (0.8,-0.9) rectangle (1.5,-1.9);
        \node [xshift=1.15cm+6cm,yshift=-2.25cm+0.375cm-0.5pt] at (1.15,-1.4) [] {\scriptsize $\beta_{2,k+2}$};
        
\end{scope}

    \end{tikzpicture}
    \caption{Initial part of Betti tables of homogeneous 2-step ideals of our interest. The symbol\  \begin{tikzpicture} \fill[fill=white] (0,-0.05) rectangle (0.2,0.1);  \fill[fill=black] (0,0) rectangle (0.1,0.1);  
    \end{tikzpicture}\   stands for 0.}
    \label{fig:geography 2-step ideals}
\end{figure}

\smallskip

The Betti number $\beta_{2,k+2}$ (like all Betti numbers) is bounded in $\mathscr{H}_{\bh}^n$ and it is useful to be able to explicitly determine its maximum value for any given pair $(\oh_k,\oh_{k+1})$. To do this, we recall the definition of the lexicographic ideal associated to a Hilbert function.

For any degree $k\in\BZ_{\ge0}$ and any integer $0\leqslant \oh_k \leqslant \rrr_k$, consider the set $L(k,\oh_k)$ of the $\oh_k$ greatest monomials of degree $k$ with respect to the lexicographic order induced by $x_1 > x_2 > \cdots > x_n$. We denote by $\oh_k^{\langle k+1\rangle}$ the dimension of the homogeneous piece of degree $k+1$ of the ideal generated by $L(k,\oh_k)$, i.e.
\[
\oh_k^{\langle k+1\rangle} =\dim_\BC \big(L(k,\oh_k)\big)_{k+1}.
\]
Among all sets of $\oh_k$ linearly independent homogeneous polynomials of degree $k$, the set $L(k,\oh_k)$ generates an ideal whose degree $k+1$ component has the smallest possible dimension.

\textit{Francis Macaulay}  proved in \cite{macaulay} that given an infinite sequence \(\bh=(\oh_i)_{i\in \mathbb{N}}\), the condition $\oh_{i+1} \geqslant \oh_{i}^{\langle i+1 \rangle}$ for all $i$ is necessary and sufficient for $\bh$ to be the Hilbert function of an ideal in $R$, see also \cite{hilbert-function-bellaterra,gin-bellaterra}. In particular, the sequence $\bh=(\oh_i)_{i\in \mathbb{N}}$, with $\oh_{i+1} \geqslant \oh_{i}^{\langle i+1 \rangle}$, for all $i\in\BN$, is the Hilbert function of the lexicographic ideal associated to $\bh$
\begin{equation}
    L_\bh = \bigoplus_{i\in\mathbb{N}} \Span_{\BC} L(i,\oh_i).
\end{equation}

The lexicographic ideal $L_\bh$ is the ideal with the highest number of generators and syzygies among all homogeneous ideals in $\OH_{\bh}^n$ as stated in the following theorem.

\begin{theorem}[\cite{Bigatti-Betti-bounds,Hulett-Betti-bounds}]\label{thm:betti bounds}
 Let $I\subset R$ be a homogeneous ideal with Hilbert function $\bh$. Then, for all $\ i,j\in\BZ$, we have
\[
\beta_{i,j} (I) \leqslant \beta_{i,j} (L_{\bh}),
\]
where $L_{\bh}\subset R$ is the lexicographic ideal with Hilbert function $\bh$.
\end{theorem}
\Cref{thm:betti bounds} can be extended to the total Betti numbers of every ideal $I$ in the Hilbert stratum $H^n_{\bh}$ in the following obvious way
\[
\beta_i(I) \leqslant \beta_i(L_{\bh}) = \sum_{j\in\BZ} \beta_{i,j}(L_{\bh}).
\]

 \medskip
 
 Now, we think of the lower bound on the dimension of $\mathsf{T}^{=-1}_{[I]}\Hilb^\bullet{\mathbb{A}^n}$ as a quadratic function of the variables $(\oh_k,\oh_{k+1})$ depending on a discrete parameter $b$ that can assume finitely many non-negative values.

\begin{definition}\label{def:TNTarea}
    Given two integers $n \geqslant 2$ and $k\geqslant 1$, we call \emph{potential TNT area} $\mathscr{T}_{k}^n\subset \BN^2$ the set of pairs $(\oh_{k},\oh_{k+1})\in\BN^2$ such that
    \begin{itemize}
        \item $\bh = (1,\ldots,\rrr_{k-1},\rrr_{k}-\oh_k,\rrr_{k+1}-\oh_{k+1})$ is a Hilbert function, i.e.~$0 \leqslant \oh_k \leqslant \rrr_k$ and  $\oh_{k}^{\langle k+1 \rangle} \leqslant \oh_{k+1} \leqslant \rrr_{k+1}$;
        \item there exists $0 \leqslant b \leqslant \beta_{2,k+2}(L_{\bh})$ such that
        \begin{equation}\label{eq:firteta}
           \oh_{k} \rrr_{k-1} + \sss_{\bh}(\rrr_{k}-\oh_{k}) + (\ttt_{\bh}-b)(\rrr_{k+1}-\oh_{k+1}) \leqslant n.
        \end{equation}
    \end{itemize}
\end{definition}

The choice of the name \emph{potential TNT area} is due to the observation that Hilbert strata corresponding to 2-step Hilbert functions with $(\oh_{k},\oh_{k+1})$ outside this region cannot produce examples of ideals with trivial negative tangents, see \Cref{def:TNT} and the bound in \Cref{eq:stima-1}.

Let us now focus on the inequality \eqref{eq:firteta}. The function
\[
\begin{split}
\Theta_{n,k,b} (\oh_k,\oh_{k+1}) &{}= \oh_{k}\rrr_{k-1} + \sss_{\bh}(\rrr_{k}-\oh_{k}) + (\ttt_{\bh}-b)(\rrr_{k+1}-\oh_{k+1}) - n = {} \\
&{} = n\oh_k^2 - \left(\tbinom{n}{2}+1\right)\oh_k\oh_{k+1} + n \oh_{k+1}^2 + \left(\rrr_{k-1}-n\rrr_{k}+\tbinom{n}{2}\rrr_{k+1}\right)\oh_k {} \\
&{} \hspace{1cm}+ (\rrr_{k}-n\rrr_{k+1}-\rrr_{k+2}+b)\oh_{k+1} + \rrr_{k+1}(\rrr_{k+2}-b) - n
\end{split}
\]
has Hessian matrix
\[
\Hess \,\Theta_{n,k,b}(\oh_k,\oh_{k+1}) =
 \left[\begin{array}{cc}2n & -\left(\binom{n}{2}+1\right) \\ -\left(\binom{n}{2}+1\right) & 2n  \end{array}\right],
\]
 whose eigenvalues are $\lambda_1 = 2n-\left(\binom{n}{2}+1\right)$ and $\lambda_2 = 2n+\left(\binom{n}{2}+1\right)$. Thus, the function $\Theta_{n,k,b}(\oh_k,\oh_{k+1})$ has a single critical point. For $n=2,3,4$, the critical point is a minimum and the level sets are ellipses. For $n\geqslant 5$, the critical point is a saddle point and the level sets are hyperbolas. 
 We will see that for $n=2,3$, there are no values of $b$ for which the minimum is non-positive and the potential TNT area turns out to be empty. For $n=4$, the minimum of $\Theta_{4,k,0}$ is negative and the potential TNT area contains at least the interior part of the ellipse $\Theta_{4,k,0} = 0$.
 For $n\geqslant 5$, the potential TNT area contains the region delimited by the two branches of a hyperbola satisfying the inequality $\Theta_{n,k,0} \leqslant 0$.

\subsection{Nesting of 2-step ideals}\label{subsec:nest2step}
Now, we adapt the two constructions introduced for homogeneous 2-step ideals in the range $\oh_{k+1}>(n-1)\oh_{k}$ to the case of nested configurations. Consider a nesting of homogeneous 2-step ideals $J \subset I \subset R$ of respective order $k+1$ and $k$. The inclusion $J \subset I$ imposes that $\bh_{J}(t) \leqslant \bh_I(t)$ for every $t\ge0$. In the case of our interest, this boils down to the unique relevant condition
\[
 \oh_{k+1}^{(J)}= \bh_{J}(k+1) \leqslant \bh_I(k+1) = \oh_{k+1}^{(I)}.
\]

 Assume that the Hilbert function $\bh_{J}$ satisfies $\sss_{\bh_J} = \oh_{k+2}^{(J)} - n \oh_{k+1}^{(J)} \geqslant 0$, i.e.~it is of type \textit{without linear syzygies}. Then, the homogeneous piece $J_{k+1}$ of the ideal $J$ can be any $\oh_{k+1}^{(J)}$-dimensional subspace of $I_{k+1}$. Moreover, the remaining $\sss_{\bh_J}$ minimal generators of $J$ of degree $k+2$ can be chosen freely in a complement of $R_1 \cdot J_{k+1}$ because $I_{k+2} = R_{k+2}$ by assumption.

\begin{theorem}\label{thm:homoloco nested}
Let $I \subset  R$ be a homogeneous 2-step ideal of order $k$. Consider a 2-step Hilbert function $\bh = (1,n, \rrr_{2}, \ldots,\rrr_{k}, \uh_{k+1}, \uh_{k+2})$   of order $k+1$ such that  $\sss_{\bh} \geqslant 0$. Denote by $\mathscr{H}_{\bh,I}^n$ the locus of homogeneous 2-step ideals with Hilbert function $\bh$ contained in $I$. Assume there exists an ideal $[J]\in \mathscr{H}_{\bh,I}^n$ with $\beta_{1,k+2}(J)=0$. 
Then, there exists a surjective morphism 
\[
\begin{tikzcd}[row sep=tiny]
 \OH_{\bh,I}^n\arrow[r,"\varphi"]&\Gr \left(\oh_{k+1},I_{k+1}\right)\\
 {[J]}\arrow[r,mapsto] &{[J_{k+1}]},
\end{tikzcd}
\]
whose generic fibre is isomorphic to $\Gr \left(\sss_{\bh},\rrr_{k+2}-n\oh_{k+1}\right)$. 
\end{theorem}
\begin{proof}
    Analogous to the proof of \Cref{thm:homoloco}.
\end{proof}

We prove now the first part of \Cref{thmintro: D} in the introduction.

\begin{corollary}\label{cor:formula dim nested}
Let $\underline{\bh} = (\bh^{(i)})_{i=0}^{r-1}$ be a $r$-tuple of 2-step Hilbert functions of respective order $k,\ldots,k+r-1$ and  such that $\sss_{\bh^{(i)}}  \geqslant 0$, for all $i=0,\ldots,r-1$. Assume that there exists a point $[\underline{I}]\in \OH_{\underline{\bh}}^n$ such that
\[
\beta_{1,k+i+1}(I^{(i)})=0,
\] 
for all $i=0,\ldots,r-1$. Then, we have
\begin{equation}\label{eq:formula dim nested}
\begin{split}
\dim H_{\underline{\bh}}^n &{}\geqslant \oh_{k}^{(0)}\left(\rrr_{k}-\oh_{k}^{(0)}\right) + \sum_{i=1}^{r-1} \oh_{k+i}^{(i)}\left(\oh_{k+i}^{(i-1)}-\oh_{k+i}^{(i)}\right) + \sum_{i=0}^{r-1}\left(\sss_{\bh^{(i)}}+\oh_{k+i}^{(i)}\right)\left(\rrr_{k+i+1}-\oh_{k+i+1}^{(i)}\right) = {}\\
&{} = \oh_{k}^{(0)}\left(\rrr_{k}-\oh_{k}^{(0)}\right) + \sum_{i=1}^{r-1} \oh_{k+i}^{(i)}\left(\oh_{k+i}^{(i-1)}-\oh_{k+i}^{(i)}\right) + \sum_{i=0}^{r-1}\left(\oh_{k+i+1}^{(i)} - (n-1)\oh_{k+i}^{(i)}\right)\left(\rrr_{k+i+1}-\oh_{k+i+1}^{(i)}\right).
\end{split}
\end{equation}  
\end{corollary}
\begin{proof}
    By \Cref{thm:homoloco} and \Cref{thm:homoloco nested},  the homogeneous locus has a distinguished component $\mathscr{V}_{\underline{\bh}}^n\subset \OH_{\underline{\bh}}^n$ of dimension
    \[
    \dim \mathscr{V}_{\underline{\bh}}^n = \oh_{k}^{(0)}\left(\rrr_{k}-\oh_{k}^{(0)}\right) + \sss_{\bh^{(0)}} \left( \rrr_{k+1} - \oh_{k+1}^{(0)}\right) + \sum_{i=1}^{r-1} \left[ \oh_{k+i}^{(i)}\left(\oh_{k+i}^{(i-1)}-\oh_{k+i}^{(i)}\right) + \sss_{\bh^{(i)}} \left( \rrr_{k+i+1} - \oh_{k+i+1}^{(i)}\right) \right],
    \] 
see \Cref{rem:disticomp}.

    On the other hand, by \Cref{thm:1nonob-nested}, the morphism  $\pi_{\bh}: H_{\bh}^n \to \mathscr{H}_\bh^n$ is an affine bundle with fibres of dimension
    \[
    \sum_{i=0}^{r-1} \oh_{k+i}^{(i)}\uh_{k+i+1}^{(i)} = \sum_{i=0}^{r-1} \oh_{k+i}^{(i)}\left(\rrr_{k+i+1}-\oh_{k+i+1}^{(i)}\right),
    \]
    and the same clearly holds for its restriction to $\mathscr{V}_{\underline{\bh}}^n$.
\end{proof}

Let us come back to a nested configuration of two homogeneous 2-step ideals $J \subset I \subset R$. Now, assume $0< -\sss_{\bh_J} < \oh_{k+1}^{(J)}$, i.e.~we consider 2-step ideals with few linear syzygies. Any minimal generating set of $\oh_{k+1}^{(J)}$ polynomials of $J$ is contained in the kernel of the restriction to $I_{k+1}^{\oplus\oh_{k+1}^{(J)}}$ of some homomorphism $\varphi: R_{k+1}^{\oplus\oh_{k+1}^{(J)}} \to R_{k+2}^{\oplus-\sss_{\bh_J}}$.  Hence, we can adapt the construction for 2-step ideals with few syzygies as follows. Consider a homogeneous 2-step ideal $I$ of order $k$ and a Hilbert function $\bh$ of 2-step ideals of order $k+1$ such that $\oh_{k+1} \leqslant \oh_{k+1}^{(I)}$ and $0 < -\sss_{\bh} < \oh_{k+1}$. We  consider the incidence correspondence
\[
\mathscr{K}_{\bh,I} = \left\{ (\varphi,\boldsymbol{p}) \in \mathscr{L}_{\bh} \times I_{k+1}^{\oplus\oh_{k+1}} \ \middle\vert\ \boldsymbol{p} \in \ker \varphi \right\} \subset \mathscr{L}_{\bh} \times I_{k+1}^{\oplus\oh_{k+1}}.
\] 

\begin{theorem}\label{thm:homoloco syz nested}
Fix a homogeneous 2-step ideal $I \subset R$ of order $k$ and consider a 2-step Hilbert function $\bh$ of order $k+1$ such that $\oh_{k+1} \leqslant \oh_{k+1}^{(I)}$ and $0< - \sss_{\bh} < \oh_{k+1}$. Denote by $\mathscr{H}_{\bh,I}^n$ the locus of homogeneous 2-step ideals contained in $I$ and having Hilbert function $\bh$ and assume that the generic morphism $\varphi\in\mathscr{L}_\bh$ is not injective. Then, there is a rational map $\mathscr{K}_{\bh,I}  \dashrightarrow \mathscr{H}_{\bh,I}^n$ which, on closed points, associates the generic pair $(\varphi,\boldsymbol{p})$ to the ideal   $I_{\boldsymbol{p}} = \left(\boldsymbol{p}\right) + \Fm^{k+3}$.
\end{theorem}
\begin{proof}
    Analogous to the proof of \Cref{thm: homo loco syz}.
\end{proof}

We give an estimate of the dimension of $\mathscr{K}_{\bh,I}$ in the case of ideals with \emph{very  few linear syzygies}, see \Cref{rem:veryfwok}. Assume $0 < -\sss_{\bh} = n\oh_{k+1}-\oh_{k+2} \leqslant \frac{1}{n}\oh_{k+1}$. Then, the kernel of a generic morphism $\varphi: R_{k+1}^{\oplus\oh_{k+1}} \to R_{k+2}^{\oplus-\sss_{\bh}}$ in $\mathscr{L}_{\bh}$ has dimension $\oh_{k+1}\rrr_{k+1} - (-\sss_{\bh})\rrr_{k+2}$. To produce a nested configuration, we need to consider $\boldsymbol{p} \in \ker \varphi  \cap I_{k+1}^{\oplus\oh_{k+1}}$ and\[
\dim_\BC\big(\ker \varphi  \cap I_{k+1}^{\oplus\oh_{k+1}}\big) = \dim_\BC \ker \varphi  + \dim_\BC I_{k+1}^{\oplus\oh_{k+1}} - \dim_\BC\big(\ker \varphi  + I_{k+1}^{\oplus\oh_{k+1}}\big).
\]
In order to ensure that the intersection $\ker\varphi\cap I_{k+1}^{\oplus\oh_{k+1}^{(J)}}$ is non-trivial, we impose the condition $\dim_\BC \ker \varphi  + \dim_\BC I_{k+1}^{\oplus\oh_{k+1}} \geqslant \dim_{\BC} R_{k+1}^{\oplus\oh_{k+1}}$, that is
\[
\oh_{k+1}\rrr_{k+1}  + \sss_{\bh}\rrr_{k+2} + \oh_{k+1}\oh_{k+1}^{(I)} \geqslant \oh_{k+1}\rrr_{k+1} \quad\Leftrightarrow\quad \rrr_{k+2}\oh_{k+2} \geqslant (n\rrr_{k+2} - \oh_{k+1}^{(I)})\oh_{k+1}.
\]
Hence, if $\oh_{k+2} \geqslant \left(\max\left\{n - \frac{1}{n},n - \frac{\oh_{k+1}^{(I)}}{\rrr_{k+2}}\right\}\right)\oh_{k+1}$, then
\[
\dim \mathscr{K}_{\bh,I} = n\oh_{k+1}(-\sss_{\bh}) + \sss_{\bh}\rrr_{k+2} + \oh_{k+1}^{(I)}\oh_{k+1},
\]
and, as a consequence of \Cref{thm:homoloco syz nested}, we get
\begin{equation}\label{eq:boundimnestvery}
    \begin{split}
\dim \mathscr{H}^n_{\bh,I} &{}\geqslant  n\oh_{k+1}(-\sss_{\bh}) + \sss_{\bh}\rrr_{k+2} + \oh_{k+1}^{(I)}\oh_{k+1} - \oh_{k+1}^2 - \sss_{\bh}^2 = {}\\
& {} = \sss_{\bh}(\rrr_{k+2}-\oh_{k+2}) + \oh_{k+1}\left(\oh_{k+1}^{(I)}-\oh_{k+1}\right).
\end{split}
\end{equation}

Note that the formula describing the dimension of a nested configuration with very few linear syzygies agrees with the formula for nested configurations without linear syzygies as expressed in the \Cref{cor:formula dim nested}. Therefore, we get the second and last part of \Cref{thmintro: D}.

\begin{corollary}\label{cor:nestveryfew}  
    Formula \eqref{eq:formula dim nested} holds for every $r$-tuple $\underline{\bh}=(\bh^{(i)})_{i=0}^{r-1}$ of 2-step Hilbert functions such that $\oh_{k+1}^{(0)} \geqslant (n-\frac{1}{n}) \oh_{k}^{(0)}$ and 
\[
\oh_{k+1+i}^{(i)} \geqslant \left( \max\left\{ n - \frac{1}{n}, n - \frac{\oh_{k+i}^{(i-1)}}{\rrr_{k+i+1}}\right\}\right) \oh_{k+i}^{(i)},\qquad\text{~for all~} i=1,\ldots,r-1.
\]
\end{corollary}
\begin{proof}
    Direct consequence of \Cref{cor:veryfew} and the inequality \eqref{eq:boundimnestvery}, which is implied by \Cref{thm:homoloco syz nested}, together with \Cref{thm:1nonob-nested}.
\end{proof}

One of the goals of the following part is to produce Hilbert strata of dimension  large enough to not be contained in the smoothable component of $\Hilb^\bullet \mathbb{A}^n$.

Given integers $n \geqslant 2$, $r\geqslant 1$ and $k\geqslant 1$, consider the subset\footnote{We omit the dependence on $n,r,k$ and we take care of not making confusion.} $\mathcal{D} \subset \mathbb{R}^{2r}$ with coordinates $\left(\oh_{k}^{(0)},\oh_{k+1}^{(0)},\ldots,\right.$ $\left.\oh_{k+r-1}^{(r-1)},\oh_{k+r}^{(r-1)}\right)$ defined by the inequalities
\[
\begin{split}
 & 0 \leqslant\oh_{k}^{(0)} \leqslant \rrr_k,\qquad\qquad \left(n-\tfrac{1}{n}\right)\oh_{k}^{(0)} \leqslant \oh_{k+1}^{(0)} \leqslant \rrr_{k+1},  \\
 & 0 \leqslant\oh_{k+i}^{(i)} \leqslant \oh_{k+i}^{(i-1)},\qquad \left(\max\left\{n-\tfrac{1}{n}, n - \tfrac{\oh_{k+i}^{(i-1)}}{\rrr_{k+i+1}}\right\}\right)\oh_{k+i}^{(i)} \leqslant \oh_{k+i+1}^{(i)} \leqslant \rrr_{k+i+1},\quad i=1,\ldots,r-1.
\end{split}
\]
The  natural points $\mathcal{D}_{\BN}=\mathcal{D}\cap{\BN}^{2r}$ correspond to 2-step Hilbert functions of nested configurations with no or very few linear syzygies considered in \Cref{cor:formula dim nested,cor:nestveryfew} respectively. We denote by $\Delta_{n,r,k}: \mathbb{R}^{2r}\to\mathbb{R}$ the function
    \[
    \begin{split}
    \Delta_{n,r,k}\left(\oh_{k}^{(0)},\oh_{k+1}^{(0)},\ldots,\oh_{k+r-1}^{(r-1)},\oh_{k+r}^{(r-1)}\right) &{} = \oh_{k}^{(0)}\left(\rrr_{k}-\oh_{k}^{(0)}\right) + \sum_{i=1}^{r-1} \oh_{k+i}^{(i)}\left(\oh_{k+i}^{(i-1)}-\oh_{k+i}^{(i)}\right)\\&\quad {} + {} \sum_{i=0}^{r-1}\left(\oh_{k+i+1}^{(i)} - (n-1)\oh_{k+i}^{(i)}\right)\left(\rrr_{k+i+1}-\oh_{k+i+1}^{(i)}\right) \\ 
    &\quad {} + n - n\left(\tbinom{k+r+n}{n} - \oh_{k+r-1}^{(r-1)} - \oh_{k+r}^{(r-1)}\right).
    \end{split}
    \]
    It gives a lower bound, in the no or very few linear syzygies case, for the difference between the dimension  of the locus parametrising fat $\underline{\bh}$-nestings not necessarily supported at the origin  and the dimension of the smoothable component of   $\Hilb^{|\underline{\bh}|}\BA^n$. 
    
From this perspective, in order to prove that a Hilbert scheme is reducible, we can look for  points in $\mathcal{D}_{\BN}$ such that $\Delta_{n,r,k}$ is non-negative. The function $\Delta_{n,r,k}$ is quadratic with tri-diagonal Hessian matrix 
\begin{equation}\label{eq:hessian}
\Hess\Delta_{n,r,k} =  
\begin{bmatrix} 
-2 & n-1 & 0 & \phantom{n-1}& \phantom{n-1}& \phantom{n-1} & \phantom{n-1} & \phantom{n-1}\\
n-1 & -2 & 1 & 0 \\
0 & 1 & -2 & n-1 & 0 \\
\phantom{n-1} & 0 & n-1 & -2 & 1 & 0 \\
& & & \ddots & \ddots & \ddots \\
 & & & 0  &  n-1 & -2 & 1 & 0 \\
 & & & & 0 & 1 & -2 & n-1 \\
 & & & & & 0 & n-1 & -2
\end{bmatrix}. 
\end{equation}
Its determinant can be computed via the continuant sequence of determinants of the matrices of increasing size starting from the top left corners:
\[
f_1 = -2,\quad f_2 = \det \left[\begin{array}{cc} -2 & n-1 \\ n-1 & -2 \ \end{array}\right] = 4 - (n-1)^2,\quad f_i = \begin{cases}
-2 f_{i-1} - f_{i-2} & \mbox{for $i$ odd,}\\
-2 f_{i-1} - (n-1)^2 f_{i-2},& \mbox{for $i$ even,}
\end{cases}
\]
see \cite{det-tridiagonal}. Then, we have $\det(\Hess\Delta_{n,r,k}) = f_{2r}$ which turns out to be always non-zero except for cases $n=3$ and $r=1$. Moreover, the Hessian matrix is also the matrix of the coefficients of the linear system we solve to determine the critical points of $\Delta_{n,r,k}$. Hence, the function $\Delta_{n,r,k}$ has a single critical point, except for the case $n=3$ and $r=1$, and according to $n\geqslant 2$ we will be able to determine its nature to obtain information about the non-negativity of $\Delta_{n,r,k}$.

\section{Reducibility of nested Hilbert schemes of points on surfaces}\label{sec:dim2}
In this section we provide new examples of reducible nested Hilbert schemes of points on a smooth surface by proving \Cref{thm: intro A}.

The Hilbert scheme $\Hilb^d \mathbb{A}^2$ is smooth and irreducible for every $d\geqslant 0$, and the only (reduced) elementary component is that of $\Hilb^1 \mathbb{A}^2$. We are interested in the nested case. But before we move on to that, we would like to highlight a feature of the TNT area. For $n=2$, the resolution of every ideal has length 2, so that $\beta_{2,k+2}$ always vanishes. The minimum of the function $\Theta_{2,k,0}$ is $\frac{1}{3}k^2 + k - 2$, so the potential TNT area is empty for every $k \geqslant 2$. As expected, this means that for $k\geqslant 2$, there is no 2-step Hilbert function $\bh$ such that $H^2_{\bh} \times\BA^2$ is a generically reduced elementary component. On the other hand, among $\Fm$-primary 2-step ideals of order $k=1$, there is only that one corresponding to a point of $\Hilb^1 \mathbb{A}^2$.

\subsection{Nested Hilbert schemes of points on surfaces}\label{subsec:dim2nest}

\paragraph{\it Known results.} We provide a brief overview of the known facts concerning the reducibility of nested Hilbert schemes of points on smooth surfaces. The basic case $r=2$ and $d_2-d_1=1$ has been treated  in \cite{CHEACELLULAR,Gottsche-motivic} where smoothness and many other properties are proven. In general, according to the results in \cite{Rasul-irr-nested},  the scheme $\Hilb^{\underline{d}}\BA^2$ is known to be irreducible when $r=2$, as well as in some other sporadic cases. Conversely, it was shown in \cite{ALESSIONESTED} that, for $r\geqslant 5$, there exist (non-trivial) elementary components of  $\prod_{\underline{d}\in\BZ^r}\Hilb^{\underline{d}}\BA^2$. Moreover, as a consequence of the results in \cite{UPDATES} it admits generically non-reduced elementary components for $r\ge6$. The geometry of the locus parametrising fat nesting has been investigated in \cite{BULOIS} and more recently in \cite{NESTSURF}, where the number of irreducible components of the punctual locus is provided for  $\underline{d}=(2,d_2)$ and is bounded for $\underline{d}=(3,d_2)$.
  
\bigskip

For $n=2$, the Hessian matrix \eqref{eq:hessian} turns out to be a Toepliz tri-diagonal matrix
\[
\Hess\Delta_{2,r,k} =
\begin{bmatrix} 
-2 & 1 & 0 & \phantom{1}& \phantom{1}& \phantom{1} \\
1 & -2 & 1 & 0 \\
0 & 1 & -2 & 1 & 0 \\
 & & \ddots & \ddots & \ddots \\
 & & 0 & 1 & -2 & 1 \\
 & & & 0 & 1 & -2
\end{bmatrix}
,
\]
with eigenvalues  
\[
\lambda_i = -2 - 2 \cos\left(\tfrac{i}{2r+1}\pi\right), 
\]
for $i=1,\ldots,2r$, see \cite[Theorem 2.2]{tridiagonal}. They are all negative, so the critical point of $\Delta_{2,r,k}$ is a maximum point. For $r=1,2,3$, the maximum values of $\Delta_{2,r,k}$ are
\[
\max \Delta_{2,1,k} = \frac{-2k^2-6k + 9}{3},\quad \max \Delta_{2,2,k} = \frac{-2k^2-15k + 5}{5}, \quad \max \Delta_{2,3,k} = \frac{-k^2-23k -15}{7}.
\]

For $r=1$, $\max \Delta_{2,1,k} \geqslant 0$ only for $k=1$. This is not surprising, as the Hilbert stratum of the function $\bh=(1)$ agrees with the smoothable component, see also the first paragraph of this section. 

For $r=2$, $\max \Delta_{2,2,k}$ is negative for all $k > 0$. Also this could be expected because the nested Hilbert scheme $\hilbert{(d_1,d_2)}{\mathbb{A}^2}$ is known to be irreducible \cite{Rasul-irr-nested}.

For $r=3$, $\max \Delta_{2,3,k}$ is negative for all $k > 0$. Thus, there are no Hilbert strata $H_{\underline{\bh}}^2$ with $\underline{\bh}$ vector of 2-step Hilbert functions with no linear syzygies or very few linear syzygies whose dimension is at least the dimension of the smoothable component. This property can be interpreted as a hint that the nested Hilbert scheme $\hilbert{(d_1,d_2,d_3)}{\mathbb{A}^2}$ might also be irreducible.

\medskip

For $r=4,5,6,7,8$, the maximum values of $\Delta_{2,r,k}$ are
\begin{equation}\label{eq:maxDelta}
    \begin{aligned}
&\max \Delta_{2,4,k} = \frac{k^2-27k - 45}{9},  &&\max \Delta_{2,5,k} = \frac{4k^2-24k -74}{11}, &&\max \Delta_{2,6,k} = \frac{8k^2-11k -86}{13}, \\
&\max \Delta_{2,7,k} = \frac{13k^2+15k -60}{15}, && \max \Delta_{2,8,k} = \frac{19k^2+57k + 30}{17}.
\end{aligned}
\end{equation}
Therefore, for  sufficiently large orders we expect many Hilbert strata of dimension larger than the dimension of the smoothable component. The following theorem and its corollary, corresponding to \Cref{thm: intro A} from the introduction, describe the first examples for different lengths of nesting. 

\begin{theorem}\label{thm:nested-srf}
    If $\underline{d}$ is one of the following increasing sequences of positive integers
    \begin{enumerate}[\rm (a)]
        \item $\underline{d} = (454,491,527,565) \in \BZ^4$,
        \item $\underline{d} = (51,64,76,87,102) \in \BZ^5$
        \item $\underline{d} = (21,30,38,45,51,61)\in \BZ^6$,
        \item $\underline{d} = (11,18,24,29,33,40,50)\in \BZ^7$,
        \item $\underline{d} = (3,8,12,18,24,29,34,43)\in \BZ^8$,
    \end{enumerate}
    then the nested Hilbert scheme $\hilbert{\underline{d}}{\BA^2}$ is reducible.
\end{theorem}
\begin{proof} 
Our strategy is the following. First, we compute the maxima  in \eqref{eq:maxDelta} and select the smallest $k$ for which the maximum is non-negative. These, in general, will not be realized by  natural numbers. Therefore, we focus on  the vertices of the hypercube of volume 1 containing the critical point considered. By doing this, we find many points with natural coordinates on which $\Delta_{2,r,k}$ assumes non-negative values. If these points are not contained in $\mathcal{D}_{\mathbb{N}}$, we explore other natural points nearby moving gradually from the critical point. If there are no points in $\mathcal{D}_{\mathbb{N}}$ with non-negative value of $\Delta_{2,r,k}$, we increase the value of $k$. Notice that some of the natural points lie on the boundary of $\mathcal{D}$. It happens that $\oh_{k+i}^{(i)} = 0$ or $\oh_{k+i+1}^{(i)} = \rrr_{k+i+1}$. In both cases, the ideal in the configuration is in fact a 1-step ideal of order $k+i$ or $k+i+1$ that we interpret as a degenerate case of 2-step ideals.

The sequences displayed in the statement correspond then to the smallest vector found with respect to the lexicographic order (from the last entry of the sequence)
\[
(d_1,\ldots,d_r) \preceq (e_1,\ldots,e_r)\quad\Leftrightarrow\quad  d_{r}=e_{r},\ \ldots,\ d_{i+1} = e_{i+1} \textnormal{~and~} d_i \leqslant e_i \text{~for some~}i.
\]

\begin{itemize}
    \item[(a)] The first degree $k$ such that $\max \Delta_{2,4,k}$ is positive is $k=29$. The maximum value of $\Delta_{2,4,29}$ is $\frac{13}{9}$ and the maximum point is
\[
\underline{\mathsf{h}}_{\max} =\tfrac{1}{9}\left(116,\:241,\:87,\:221,\:67,\:210,\:56,\:190\right).
\]
It is contained in $\mathcal{D}\subset\mathbb{R}^8$ and exploring the natural points in $\mathcal{D}_{\BN}$ starting from the vertices of the hypercube containing $\underline{\mathsf{h}}_{\max}$, we find 261 points with $\Delta_{2,4,29}\in\{0,1\}$.

The smallest sequence $\underline{d} = (d_1,d_2,d_3,d_4) \in \mathbb{Z}^4$ we find  is $(454,491,527,565)$.

\item[(b)] For 4 nestings, the first degree $k$ such that $\max \Delta_{2,5,k}$ is positive is $k=9$. The maximum value of $\Delta_{2,5,9}$ is $\frac{34}{11}$ and the maximum point is
\[
\underline{\mathsf{h}}_{\max} =\tfrac{1}{11}\left(41,\:93,\:24,\:87,\:18,\:92,\:23,\:108,\:39,\:113\right).
\]
It is again contained in $\mathcal{D}\subset\mathbb{R}^{10}$ and exploring the intersection of $\mathcal{D}_{\BN}$ with the  hypercube containing $\underline{\mathsf{h}}_{\max}$, we find 12884 integer points with non-negative value of $\Delta_{2,5,9}$. The smallest sequence $\underline{d} = (d_1,d_2,d_3,d_4,d_5) \in \mathbb{Z}^5$ we find is $(51,64,76,87,102)$.

\item[(c)] For 5 nesting, the first degree $k$ such that $\max \Delta_{2,6,k}$ is positive is $k=5$. In this case, the maximum point
\[
\underline{\mathsf{h}}_{\max} =\tfrac{1}{13}\left(21,\:55,\:-2,\:45,\:-12,\:48,\:-9,\:64,\:7,\:93,\:36,\:109\right)
\]
is not contained in $\mathcal{D}\subset\mathbb{R}^{12}$. However, there are vertices of the hypercube containing the critical point in $\mathcal{D}_{\mathbb{N}}$ on which $\Delta_{2,6,5}$ is positive. Starting from these points and moving around $\mathcal{D}_{\BN}$, the smallest sequence $\underline{d} \in \mathbb{Z}^6$ on which $\Delta_{2,6,5}$ is positive that we find is $(21,30,38,45,51,61)$.

\item[(d)] For 6 nestings, the first degree $k$ such that $\max \Delta_{2,7,k}$ is positive is $k=2$, but for $k=2$ the critical point is not contained in $\mathcal{D}\subset\mathbb{R}^{14}$ and there are no natural points in $\mathcal{D}_\mathbb{N}$ with non-negative value of $\Delta_{2,7,2}$. For $k=3$, moving around $\mathcal{D}_{\BN}$, we find the sequence $\underline{d} = (11,18,24,29,33,40,50) \in \mathbb{Z}^7$.

\item[(e)] For 7 nestings, $\max \Delta_{2,8,k}$ is always positive. For $k=1$, the critical point is quite far from $\mathcal{D}\subset\mathbb{R}^{16}$. However, there are 330 natural points in $\mathcal{D}_\mathbb{N}$ with non-negative value of $\Delta_{2,8,1}$ and lowest sequence is $(3,8,12,18,24,29,34,43)\in \BZ^8$.
\end{itemize} 

\smallskip

See \Cref{fig:nested surface} for a detailed description of the generic homogeneous ideals in the configuration of Hilbert strata $H_{\underline{\bh}}^2$ certifying the reducibility of the nested Hilbert scheme. The ancillary \textit{Macaulay2} file \href{www.paololella.it/software/reducibility-nested-Hilbert-schemes.m2}{\tt reducibility-nested-Hilbert-schemes.m2} contains the code to explicitly produce a configuration for each case.
\end{proof}

\begin{figure}[!ht]
    \centering
    
\subfigure[2-step Hilbert functions certifying the reducibility of $\hilbert{(454,491,527,565)}{\mathbb{A}^2}$]{\label{subfig: nested 2 r 4}
\begin{tikzpicture}[yscale=0.35]
    
    \draw [] (-1.5,1) -- (11.5,1);
    \draw [] (-1.5,-0.5) -- (11.5,-0.5);
    \draw [] (-1.5,1) -- (-1.5,-.5);
    \draw [] (2.5,1) -- (2.5,-.5);
    \draw [] (3.5,1) -- (3.5,-.5);
    \draw [] (4.5,1) -- (4.5,-.5);
    \draw [] (7,1) -- (7,-.5);
    \draw [] (8.5,1) -- (8.5,-.5);
    \draw [] (10,1) -- (10,-.5);
    \draw [] (11.5,1) -- (11.5,-.5);

    \node at (0.5,0.25) [] {\footnotesize $\bh^{(i)}$};
    \node at (3,0.25) [] {\footnotesize $\left\vert \bh^{(i)}\right\vert$};
    \node at (4,0.25) [] {\footnotesize $k+i$};
    \node at (5.75,0.25) [] {\footnotesize $\big(\oh_{k+i}^{(i)},\oh_{k+i+1}^{(i)}\big)$};
    \node at (7.75,0.25) [] {\footnotesize $\dim \mathsf{T}^{<0}$};
    \node at (7.75,0.25) [xshift=1.5cm] {\footnotesize $\dim \mathsf{T}^{=0}$};
    \node at (9.25,0.25) [xshift=1.5cm] {\footnotesize $\dim \mathsf{T}^{=1}$};
    
    \begin{scope}[yshift=-8pt,yscale=0.85]
    \draw [thin] (-1.5,-0.5) -- (11.5,-0.5);
    \draw [thin] (-1.5,-1.5) -- (11.5,-1.5);
    \draw [thin] (-1.5,-2.5) -- (11.5,-2.5);
    \draw [thin] (-1.5,-3.5) -- (11.5,-3.5);
    \draw [thin] (-1.5,-4.5) -- (11.5,-4.5);

    \draw [thin] (-1.5,-0.5) -- (-1.5,-4.5);
    \draw [thin] (2.5,-0.5) -- (2.5,-4.5);
    \draw [thin] (3.5,-0.5) -- (3.5,-4.5);
    \draw [thin] (4.5,-0.5) -- (4.5,-4.5);
    \draw [thin] (7,-0.5) -- (7,-4.5);
    \draw [thin] (8.5,-0.5) -- (8.5,-4.5);
    \draw [thin] (10,-0.5) -- (10,-4.5);
    \draw [thin] (11.5,-0.5) -- (11.5,-4.5);
    
    \node at (0.5,-1) [] {\parbox{3.5cm}{\scriptsize $\bh^{(0)}=(1,2,\ldots,29,16,3)$}};
    \node at (0.5,-2) [] {\parbox{3.5cm}{\scriptsize $\bh^{(1)}=(1,2,\ldots,29,30,20,6)$}};
    \node at (0.5,-3) [] {\parbox{3.5cm}{\scriptsize $\bh^{(2)}=(1,2,\ldots,29,30,31,23,8)$}};
    \node at (0.5,-4) [] {\parbox{3.5cm}{\scriptsize $\bh^{(3)}=(1,2,\ldots,29,30,31,32,25,12)$}};

    \node at (3,-1) [] {\scriptsize $454$};
    \node at (3,-2) [] {\scriptsize $491$};
    \node at (3,-3) [] {\scriptsize $527$};
    \node at (3,-4) [] {\scriptsize $565$};
    
    \node at (4,-1) [] {\scriptsize $29$};
    \node at (4,-2) [] {\scriptsize $30$};
    \node at (4,-3) [] {\scriptsize $31$};
    \node at (4,-4) [] {\scriptsize $32$};
 
    \node at (5.75,-1) [xshift=-0.6cm] {\scriptsize $(14,28)$};
    \node at (5.75,-2) [xshift=-0.2cm] {\scriptsize $(11,26)$};
    \node at (5.75,-3) [xshift=0.2cm] {\scriptsize $(9,25)$};
    \node at (5.75,-4) [xshift=0.6cm] {\scriptsize $(8,22)$};
  
    \node at (7.75,-1) [] {\scriptsize $642$}; \node at (7.75,-1) [xshift=1.5cm] {\scriptsize $224$}; \node at (9.25,-1) [xshift=1.5cm] {\scriptsize $42$};
    \node at (7.75,-2) [] {\scriptsize $672$}; \node at (7.75,-2) [xshift=1.5cm] {\scriptsize $244$}; \node at (9.25,-2) [xshift=1.5cm] {\scriptsize $66$};
    \node at (7.75,-3) [] {\scriptsize $719$}; \node at (7.75,-3) [xshift=1.5cm] {\scriptsize $263$}; \node at (9.25,-3) [xshift=1.5cm] {\scriptsize $72$};
     \node at (7.75,-4) [] {\scriptsize $762$}; \node at (7.75,-4) [xshift=1.5cm] {\scriptsize $272$}; \node at (9.25,-4) [xshift=1.5cm] {\scriptsize $96$};
    \end{scope}

    \begin{scope}[yshift=8pt]

        \node at (4.75,-5.05) [] {\footnotesize $\underline{\bh}=(\bh^{(0)},\bh^{(1)},\bh^{(2)},\bh^{(3)})$};
        \draw[thick] (2.5,-4.5) rectangle (11.5,-5.5);
        \draw (7,-4.5) -- (7,-5.5);
        \draw (8.5,-4.5) -- (8.5,-5.5);
        \draw (10,-4.5) -- (10,-5.5);
        
         \node at (7.75,-5) [] {\footnotesize $874$}; \node at (7.75,-5) [xshift=1.5cm] {\footnotesize $864$}; \node at (9.25,-5) [xshift=1.5cm] {\footnotesize $276$};

    \end{scope}
\end{tikzpicture}
}

\subfigure[2-step Hilbert functions certifying the reducibility of $\hilbert{(51,64,76,87,102)}{\mathbb{A}^2}$]{\label{subfig: nested 3 r 4}
\begin{tikzpicture}[yscale=0.35]
    \draw [] (-1.5,1) -- (11.5,1);
    \draw [] (-1.5,-0.5) -- (11.5,-0.5);
    \draw [] (-1.5,1) -- (-1.5,-.5);
    \draw [] (2.5,1) -- (2.5,-.5);
    \draw [] (3.5,1) -- (3.5,-.5);
    \draw [] (4.5,1) -- (4.5,-.5);
    \draw [] (7,1) -- (7,-.5);
    \draw [] (8.5,1) -- (8.5,-.5);
    \draw [] (10,1) -- (10,-.5);
    \draw [] (11.5,1) -- (11.5,-.5);

    \node at (0.5,0.25) [] {\footnotesize $\bh^{(i)}$};
    \node at (3,0.25) [] {\footnotesize $\left\vert \bh^{(i)}\right\vert$};
    \node at (4,0.25) [] {\footnotesize $k+i$};
    \node at (5.75,0.25) [] {\footnotesize $\big(\oh_{k+i}^{(i)},\oh_{k+i+1}^{(i)}\big)$};
    \node at (7.75,0.25) [] {\footnotesize $\dim \mathsf{T}^{<0}$};
    \node at (7.75,0.25) [xshift=1.5cm] {\footnotesize $\dim \mathsf{T}^{=0}$};
    \node at (9.25,0.25) [xshift=1.5cm] {\footnotesize $\dim \mathsf{T}^{=1}$};
    
    \begin{scope}[yshift=-8pt,yscale=0.85]
    \draw [thin] (-1.5,-0.5) -- (11.5,-0.5);
    \draw [thin] (-1.5,-1.5) -- (11.5,-1.5);
    \draw [thin] (-1.5,-2.5) -- (11.5,-2.5);
    \draw [thin] (-1.5,-3.5) -- (11.5,-3.5);
    \draw [thin] (-1.5,-4.5) -- (11.5,-4.5);
\draw [thin] (-1.5,-5.5) -- (11.5,-5.5);

    \draw [thin] (-1.5,-0.5) -- (-1.5,-5.5);
    \draw [thin] (2.5,-0.5) -- (2.5,-5.5);
    \draw [thin] (3.5,-0.5) -- (3.5,-5.5);
    \draw [thin] (4.5,-0.5) -- (4.5,-5.5);
    \draw [thin] (7,-0.5) -- (7,-5.5);
    \draw [thin] (8.5,-0.5) -- (8.5,-5.5);
    \draw [thin] (10,-0.5) -- (10,-5.5);
    \draw [thin] (11.5,-0.5) -- (11.5,-5.5);
    
    \node at (0.5,-1) [] {\parbox{3.5cm}{\scriptsize $\bh^{(0)}=(1,2,\ldots,9,5,1)$}};
    \node at (0.5,-2) [] {\parbox{3.5cm}{\scriptsize $\bh^{(1)}=(1,2,\ldots,9,10,7,2)$}};
    \node at (0.5,-3) [] {\parbox{3.5cm}{\scriptsize $\bh^{(2)}=(1,2,\ldots,9,10,11,8,2)$}};
    \node at (0.5,-4) [] {\parbox{3.5cm}{\scriptsize $\bh^{(3)}=(1,2,\ldots,9,10,11,12,8,1)$}};
    \node at (0.5,-5) [] {\parbox{3.5cm}{\scriptsize $\bh^{(4)}=(1,2,\ldots,9,10,11,12,13,8,3)$}};

    \node at (3,-1) [] {\scriptsize $51$};
    \node at (3,-2) [] {\scriptsize $64$};
    \node at (3,-3) [] {\scriptsize $76$};
    \node at (3,-4) [] {\scriptsize $87$};
    \node at (3,-5) [] {\scriptsize $102$};
    
    \node at (4,-1) [] {\scriptsize $9$};
    \node at (4,-2) [] {\scriptsize $10$};
    \node at (4,-3) [] {\scriptsize $11$};
    \node at (4,-4) [] {\scriptsize $12$};
    \node at (4,-5) [] {\scriptsize $13$};

    \node at (5.75,-1) [xshift=-0.5cm] {\scriptsize $(5,10)$};
    \node at (5.75,-2) [xshift=-0.25cm] {\scriptsize $(4,10)$};
    \node at (5.75,-3) [] {\scriptsize $(4,11)$};
    \node at (5.75,-4) [xshift=0.25cm] {\scriptsize $(5,13)$};
    \node at (5.75,-5) [xshift=0.5cm] {\scriptsize $(6,12)$};

    \node at (7.75,-1) [] {\scriptsize $72$}; \node at (7.75,-1) [xshift=1.5cm] {\scriptsize $25$}; \node at (9.25,-1) [xshift=1.5cm] {\scriptsize $5$};
    \node at (7.75,-2) [] {\scriptsize $88$}; \node at (7.75,-2) [xshift=1.5cm] {\scriptsize $32$}; \node at (9.25,-2) [xshift=1.5cm] {\scriptsize $8$};
    \node at (7.75,-3) [] {\scriptsize $106$}; \node at (7.75,-3) [xshift=1.5cm] {\scriptsize $38$}; \node at (9.25,-3) [xshift=1.5cm] {\scriptsize $8$};
     \node at (7.75,-4) [] {\scriptsize $126$}; \node at (7.75,-4) [xshift=1.5cm] {\scriptsize $43$}; \node at (9.25,-4) [xshift=1.5cm] {\scriptsize $5$};
  \node at (7.75,-5) [] {\scriptsize $138$}; \node at (7.75,-5) [xshift=1.5cm] {\scriptsize $48$}; \node at (9.25,-5) [xshift=1.5cm] {\scriptsize $18$};

    \end{scope}
    \begin{scope}[yshift=12pt]

        \node at (4.75,-6.05) [] {\footnotesize $\underline{\bh}=(\bh^{(0)},\bh^{(1)},\bh^{(2)},\bh^{(3)},\bh^{(4)})$};
        \draw[thick] (2.5,-5.5) rectangle (11.5,-6.5);
        \draw (7,-5.5) -- (7,-6.5);
        \draw (8.5,-5.5) -- (8.5,-6.5);
        \draw (10,-5.5) -- (10,-6.5);

        \node at (7.75,-6) [] {\footnotesize $150$};
         \node at (7.75,-6) [xshift=1.5cm] {\footnotesize $158$}; \node at (9.25,-6) [xshift=1.5cm] {\footnotesize $44$};
    \end{scope}
\end{tikzpicture}
}

\subfigure[2-step Hilbert functions certifying the reducibility of $\hilbert{(21,30,38,45,51,61)}{\mathbb{A}^2}$]{\label{subfig: nested 2 r 6}
\begin{tikzpicture}[yscale=0.35]
    \draw [] (-1.5,1) -- (11.5,1);
    \draw [] (-1.5,-0.5) -- (11.5,-0.5);
    \draw [] (-1.5,1) -- (-1.5,-.5);
    \draw [] (2.5,1) -- (2.5,-.5);
    \draw [] (3.5,1) -- (3.5,-.5);
    \draw [] (4.5,1) -- (4.5,-.5);
    \draw [] (7,1) -- (7,-.5);
    \draw [] (8.5,1) -- (8.5,-.5);
    \draw [] (10,1) -- (10,-.5);
    \draw [] (11.5,1) -- (11.5,-.5);

    \node at (0.5,0.25) [] {\footnotesize $\bh^{(i)}$};
    \node at (3,0.25) [] {\footnotesize $\left\vert \bh^{(i)}\right\vert$};
    \node at (4,0.25) [] {\footnotesize $k+i$};
    \node at (5.75,0.25) [] {\footnotesize $\big(\oh_{k+i}^{(i)},\oh_{k+i+1}^{(i)}\big)$};
    \node at (7.75,0.25) [] {\footnotesize $\dim \mathsf{T}^{<0}$};
    \node at (7.75,0.25) [xshift=1.5cm] {\footnotesize $\dim \mathsf{T}^{=0}$};
    \node at (9.25,0.25) [xshift=1.5cm] {\footnotesize $\dim \mathsf{T}^{=1}$};
    
    \begin{scope}[yshift=-8pt,yscale=0.85]
    \draw [thin] (-1.5,-0.5) -- (11.5,-0.5);
    \draw [thin] (-1.5,-1.5) -- (11.5,-1.5);
    \draw [thin] (-1.5,-2.5) -- (11.5,-2.5);
    \draw [thin] (-1.5,-3.5) -- (11.5,-3.5);
    \draw [thin] (-1.5,-4.5) -- (11.5,-4.5);
    \draw [thin] (-1.5,-5.5) -- (11.5,-5.5);
\draw [thin] (-1.5,-6.5) -- (11.5,-6.5);

    \draw [thin] (-1.5,-0.5) -- (-1.5,-6.5);
    \draw [thin] (2.5,-0.5) -- (2.5,-6.5);
    \draw [thin] (3.5,-0.5) -- (3.5,-6.5);
    \draw [thin] (4.5,-0.5) -- (4.5,-6.5);
    \draw [thin] (7,-0.5) -- (7,-6.5);
    \draw [thin] (8.5,-0.5) -- (8.5,-6.5);
    \draw [thin] (10,-0.5) -- (10,-6.5);
    \draw [thin] (11.5,-0.5) -- (11.5,-6.5);
    
    \node at (0.5,-1) [] {\parbox{3.5cm}{\scriptsize $\bh^{(0)}=(1,2,3,4,5,4,2)$}};
    \node at (0.5,-2) [] {\parbox{3.5cm}{\scriptsize $\bh^{(1)}=(1,2,3,4,5,6,6,3)$}};
    \node at (0.5,-3) [] {\parbox{3.5cm}{\scriptsize $\bh^{(2)}=(1,2,3,4,5,6,7,7,3)$}};
    \node at (0.5,-4) [] {\parbox{3.5cm}{\scriptsize $\bh^{(3)}=(1,2,3,4,5,6,7,8,7,2)$}};
    \node at (0.5,-5) [] {\parbox{3.5cm}{\scriptsize $\bh^{(4)}=(1,2,3,4,5,6,7,8,9,6,0)$}};
    \node at (0.5,-6) [] {\parbox{3.5cm}{\scriptsize $\bh^{(5)}=(1,2,3,4,5,6,7,8,9,10,5,1)$}};

    \node at (3,-1) [] {\scriptsize $21$};
    \node at (3,-2) [] {\scriptsize $30$};
    \node at (3,-3) [] {\scriptsize $38$};
    \node at (3,-4) [] {\scriptsize $45$};
    \node at (3,-5) [] {\scriptsize $51$};
    \node at (3,-6) [] {\scriptsize $61$};
    
    \node at (4,-1) [] {\scriptsize $5$};
    \node at (4,-2) [] {\scriptsize $6$};
    \node at (4,-3) [] {\scriptsize $7$};
    \node at (4,-4) [] {\scriptsize $8$};
    \node at (4,-5) [] {\scriptsize $9$};
    \node at (4,-6) [] {\scriptsize $10$};

    \node at (5.75,-1) [xshift=-0.5cm] {\scriptsize $(2,5)$};
    \node at (5.75,-2) [xshift=-0.3cm] {\scriptsize $(1,5)$};
    \node at (5.75,-3) [xshift=-0.1cm] {\scriptsize $(1,6)$};
    \node at (5.75,-4) [xshift=0.1cm] {\scriptsize $(2,8)$};
    \node at (5.75,-5) [xshift=0.3cm] {\scriptsize $(4,11)$};
    \node at (5.75,-6) [xshift=0.5cm] {\scriptsize $(6,11)$};
    
    \node at (7.75,-1) [] {\scriptsize $28$}; \node at (9.25,-1) [] {\scriptsize $10$}; \node at (10.75,-1) [] {\scriptsize $4$};
    \node at (7.75,-2) [] {\scriptsize $42$}; \node at (7.75,-2) [xshift=1.5cm] {\scriptsize $15$}; \node at (9.25,-2) [xshift=1.5cm] {\scriptsize $3$};
    \node at (7.75,-3) [] {\scriptsize $52$}; \node at (7.75,-3) [xshift=1.5cm] {\scriptsize $21$}; \node at (9.25,-3) [xshift=1.5cm] {\scriptsize $3$};
     \node at (7.75,-4) [] {\scriptsize $64$}; \node at (7.75,-4) [xshift=1.5cm] {\scriptsize $22$}; \node at (9.25,-4) [xshift=1.5cm] {\scriptsize $4$};
  \node at (7.75,-5) [] {\scriptsize $78$}; \node at (7.75,-5) [xshift=1.5cm] {\scriptsize $24$}; \node at (9.25,-5) [xshift=1.5cm] {\scriptsize $0$};
 \node at (7.75,-6) [] {\scriptsize $87$}; \node at (7.75,-6) [xshift=1.5cm] {\scriptsize $29$}; \node at (9.25,-6) [xshift=1.5cm] {\scriptsize $6$};

    \end{scope}
    \begin{scope}[yshift=16pt]

        \node at (4.75,-7.05) [] {\footnotesize $\underline{\bh}=(\bh^{(0)},\bh^{(1)},\bh^{(2)},\bh^{(3)},\bh^{(4)},\bh^{(5)})$};
        \draw[thick] (2.5,-6.5) rectangle (11.5,-7.5);
        \draw (7,-6.5) -- (7,-7.5);
        \draw (8.5,-6.5) -- (8.5,-7.5);
        \draw (10,-6.5) -- (10,-7.5);

        \node at (7.75,-7) [] {\footnotesize $90$};
         \node at (7.75,-7) [xshift=1.5cm] {\footnotesize $100$}; \node at (9.25,-7) [xshift=1.5cm] {\footnotesize $20$};

    \end{scope}
\end{tikzpicture}
}

\subfigure[2-step Hilbert functions certifying the reducibility of $\hilbert{(11,18,24,29,33,40,50)}{\mathbb{A}^2}$]{\label{subfig: nested 2 r 7}
\begin{tikzpicture}[yscale=0.35]
    \draw [] (-1.5,1) -- (11.5,1);
    \draw [] (-1.5,-0.5) -- (11.5,-0.5);
    \draw [] (-1.5,1) -- (-1.5,-.5);
    \draw [] (2.5,1) -- (2.5,-.5);
    \draw [] (3.5,1) -- (3.5,-.5);
    \draw [] (4.5,1) -- (4.5,-.5);
    \draw [] (7,1) -- (7,-.5);
    \draw [] (8.5,1) -- (8.5,-.5);
    \draw [] (10,1) -- (10,-.5);
    \draw [] (11.5,1) -- (11.5,-.5);

    \node at (0.5,0.25) [] {\footnotesize $\bh^{(i)}$};
    \node at (3,0.25) [] {\footnotesize $\left\vert \bh^{(i)}\right\vert$};
    \node at (4,0.25) [] {\footnotesize $k+i$};
    \node at (5.75,0.25) [] {\footnotesize $\big(\oh_{k+i}^{(i)},\oh_{k+i+1}^{(i)}\big)$};
    \node at (7.75,0.25) [] {\footnotesize $\dim \mathsf{T}^{<0}$};
    \node at (7.75,0.25) [xshift=1.5cm] {\footnotesize $\dim \mathsf{T}^{=0}$};
    \node at (9.25,0.25) [xshift=1.5cm] {\footnotesize $\dim \mathsf{T}^{=1}$};
    
    \begin{scope}[yshift=-8pt,yscale=0.85]
    \draw [thin] (-1.5,-0.5) -- (11.5,-0.5);
    \draw [thin] (-1.5,-1.5) -- (11.5,-1.5);
    \draw [thin] (-1.5,-2.5) -- (11.5,-2.5);
    \draw [thin] (-1.5,-3.5) -- (11.5,-3.5);
    \draw [thin] (-1.5,-4.5) -- (11.5,-4.5);
    \draw [thin] (-1.5,-5.5) -- (11.5,-5.5);
    \draw [thin] (-1.5,-6.5) -- (11.5,-6.5);
    \draw [thin] (-1.5,-7.5) -- (11.5,-7.5);

    \draw [thin] (-1.5,-0.5) -- (-1.5,-7.5);
    \draw [thin] (2.5,-0.5) -- (2.5,-7.5);
    \draw [thin] (3.5,-0.5) -- (3.5,-7.5);
    \draw [thin] (4.5,-0.5) -- (4.5,-7.5);
    \draw [thin] (7,-0.5) -- (7,-7.5);
    \draw [thin] (8.5,-0.5) -- (8.5,-7.5);
    \draw [thin] (10,-0.5) -- (10,-7.5);
    \draw [thin] (11.5,-0.5) -- (11.5,-7.5);
        
    \node at (0.5,-1) [] {\parbox{3.5cm}{\scriptsize $\bh^{(0)}=(1,2,3,3,2)$}};
    \node at (0.5,-2) [] {\parbox{3.5cm}{\scriptsize $\bh^{(1)}=(1,2,3,4,5,3)$}};
    \node at (0.5,-3) [] {\parbox{3.5cm}{\scriptsize $\bh^{(2)}=(1,2,3,4,5,6,3)$}};
    \node at (0.5,-4) [] {\parbox{3.5cm}{\scriptsize $\bh^{(3)}=(1,2,3,4,5,6,6,2)$}};
    \node at (0.5,-5) [] {\parbox{3.5cm}{\scriptsize $\bh^{(4)}=(1,2,3,4,5,6,7,5,0)$}};
    \node at (0.5,-6) [] {\parbox{3.5cm}{\scriptsize $\bh^{(5)}=(1,2,3,4,5,6,7,8,5,0)$}};
    \node at (0.5,-7) [] {\parbox{3.5cm}{\scriptsize $\bh^{(6)}=(1,2,3,4,5,6,7,8,9,4,1)$}};

    \node at (3,-1) [] {\scriptsize $11$};
    \node at (3,-2) [] {\scriptsize $18$};
    \node at (3,-3) [] {\scriptsize $24$};
    \node at (3,-4) [] {\scriptsize $29$};
    \node at (3,-5) [] {\scriptsize $33$};
    \node at (3,-6) [] {\scriptsize $40$};
    \node at (3,-7) [] {\scriptsize $50$};
    
    \node at (4,-1) [] {\scriptsize $3$};
    \node at (4,-2) [] {\scriptsize $4$};
    \node at (4,-3) [] {\scriptsize $5$};
    \node at (4,-4) [] {\scriptsize $6$};
    \node at (4,-5) [] {\scriptsize $7$};
    \node at (4,-6) [] {\scriptsize $8$};
    \node at (4,-7) [] {\scriptsize $9$};

    \node at (5.75,-1) [xshift=-0.6cm] {\scriptsize $(1,3)$};
    \node at (5.75,-2) [xshift=-0.4cm] {\scriptsize $(0,3)$};
    \node at (5.75,-3) [xshift=-0.2cm] {\scriptsize $(0,4)$};
    \node at (5.75,-4) [] {\scriptsize $(1,6)$};
    \node at (5.75,-5) [xshift=0.2cm] {\scriptsize $(3,9)$};
    \node at (5.75,-6) [xshift=0.4cm] {\scriptsize $(5,10)$};
    \node at (5.75,-7) [xshift=0.6cm] {\scriptsize $(6,10)$};
    
    \node at (7.75,-1) [] {\scriptsize $15$}; \node at (7.75,-1) [xshift=1.5cm] {\scriptsize $5$};  \node at (9.25,-1) [xshift=1.5cm] {\scriptsize $2$};
    \node at (7.75,-2) [] {\scriptsize $37$}; \node at (7.75,-2) [xshift=1.5cm] {\scriptsize $9$}; \node at (9.25,-2) [xshift=1.5cm] {\scriptsize $0$};
    \node at (7.75,-3) [] {\scriptsize $36$}; \node at (7.75,-3) [xshift=1.5cm] {\scriptsize $12$}; \node at (9.25,-3) [xshift=1.5cm] {\scriptsize $0$};
     \node at (7.75,-4) [] {\scriptsize $42$}; \node at (7.75,-4) [xshift=1.5cm] {\scriptsize $14$}; \node at (9.25,-4) [xshift=1.5cm] {\scriptsize $2$};
  \node at (7.75,-5) [] {\scriptsize $51$}; \node at (7.75,-5) [xshift=1.5cm] {\scriptsize $15$}; \node at (9.25,-5) [xshift=1.5cm] {\scriptsize $0$};
 \node at (7.75,-6) [] {\scriptsize $60$}; \node at (7.75,-6) [xshift=1.5cm] {\scriptsize $20$}; \node at (9.25,-6) [xshift=1.5cm] {\scriptsize $0$};
 \node at (7.75,-7) [] {\scriptsize $72$}; \node at (7.75,-7) [xshift=1.5cm] {\scriptsize $22$}; \node at (9.25,-7) [xshift=1.5cm] {\scriptsize $6$};
    \end{scope}

    \begin{scope}[yshift=20pt]

        \node at (4.75,-8.05) [] {\footnotesize $\underline{\bh}=(\bh^{(0)},\bh^{(1)},\bh^{(2)},\bh^{(3)},\bh^{(4)},\bh^{(5)},\bh^{(6)})$};
        \draw[thick] (2.5,-7.5) rectangle (11.5,-8.5);
        \draw[] (7,-7.5) -- (7,-8.5);
        \draw[] (8.5,-7.5) -- (8.5,-8.5);
        \draw[] (10,-7.5) -- (10,-8.5);
        
         \node at (7.75,-8) [] {\footnotesize $87$}; \node at (7.75,-8) [xshift=1.5cm] {\footnotesize $88$}; \node at (9.25,-8) [xshift=1.5cm] {\footnotesize $10$};

    \end{scope}
\end{tikzpicture}
}

\subfigure[2-step Hilbert functions certifying the reducibility of $\hilbert{(3,8,12,18,24,29,34,43)}{\mathbb{A}^2}$]{\label{subfig: nested 2 r 8}
\begin{tikzpicture}[yscale=0.35]
    \draw [] (-1.5,1) -- (11.5,1);
    \draw [] (-1.5,-0.5) -- (11.5,-0.5);
    \draw [] (-1.5,1) -- (-1.5,-.5);
    \draw [] (2.5,1) -- (2.5,-.5);
    \draw [] (3.5,1) -- (3.5,-.5);
    \draw [] (4.5,1) -- (4.5,-.5);
    \draw [] (7,1) -- (7,-.5);
    \draw [] (8.5,1) -- (8.5,-.5);
    \draw [] (10,1) -- (10,-.5);
    \draw [] (11.5,1) -- (11.5,-.5);

    \node at (0.5,0.25) [] {\footnotesize $\bh^{(i)}$};
    \node at (3,0.25) [] {\footnotesize $\left\vert \bh^{(i)}\right\vert$};
    \node at (4,0.25) [] {\footnotesize $k+i$};
    \node at (5.75,0.25) [] {\footnotesize $\big(\oh_{k+i}^{(i)},\oh_{k+i+1}^{(i)}\big)$};
    \node at (7.75,0.25) [] {\footnotesize $\dim \mathsf{T}^{<0}$};
    \node at (7.75,0.25) [xshift=1.5cm] {\footnotesize $\dim \mathsf{T}^{=0}$};
    \node at (9.25,0.25) [xshift=1.5cm] {\footnotesize $\dim \mathsf{T}^{=1}$};
    
    \begin{scope}[yshift=-8pt,yscale=0.85]
    \draw [thin] (-1.5,-0.5) -- (11.5,-0.5);
    \draw [thin] (-1.5,-1.5) -- (11.5,-1.5);
    \draw [thin] (-1.5,-2.5) -- (11.5,-2.5);
    \draw [thin] (-1.5,-3.5) -- (11.5,-3.5);
    \draw [thin] (-1.5,-4.5) -- (11.5,-4.5);
    \draw [thin] (-1.5,-5.5) -- (11.5,-5.5);
    \draw [thin] (-1.5,-6.5) -- (11.5,-6.5);
    \draw [thin] (-1.5,-7.5) -- (11.5,-7.5);
    \draw [thin] (-1.5,-8.5) -- (11.5,-8.5);

    \draw [thin] (-1.5,-0.5) -- (-1.5,-8.5);
    \draw [thin] (2.5,-0.5) -- (2.5,-8.5);
    \draw [thin] (3.5,-0.5) -- (3.5,-8.5);
    \draw [thin] (4.5,-0.5) -- (4.5,-8.5);
    \draw [thin] (7,-0.5) -- (7,-8.5);
    \draw [thin] (8.5,-0.5) -- (8.5,-8.5);
    \draw [thin] (10,-0.5) -- (10,-8.5);
    \draw [thin] (11.5,-0.5) -- (11.5,-8.5);
    
    \node at (0.5,-1) [] {\parbox{3.5cm}{\scriptsize $\bh^{(0)}=(1,1,1)$}};
    \node at (0.5,-2) [] {\parbox{3.5cm}{\scriptsize $\bh^{(1)}=(1,2,3,2)$}};
    \node at (0.5,-3) [] {\parbox{3.5cm}{\scriptsize $\bh^{(2)}=(1,2,3,4,2)$}};
    \node at (0.5,-4) [] {\parbox{3.5cm}{\scriptsize $\bh^{(3)}=(1,2,3,4,5,3)$}};
    \node at (0.5,-5) [] {\parbox{3.5cm}{\scriptsize $\bh^{(4)}=(1,2,3,4,5,6,3)$}};
    \node at (0.5,-6) [] {\parbox{3.5cm}{\scriptsize $\bh^{(5)}=(1,2,3,4,5,6,6,2)$}};
    \node at (0.5,-7) [] {\parbox{3.5cm}{\scriptsize $\bh^{(6)}=(1,2,3,4,5,6,7,5,1)$}};
    \node at (0.5,-8) [] {\parbox{3.5cm}{\scriptsize $\bh^{(7)}=(1,2,3,4,5,6,7,8,5,2)$}};

    \node at (3,-1) [] {\scriptsize $3$};
    \node at (3,-2) [] {\scriptsize $8$};
    \node at (3,-3) [] {\scriptsize $12$};
    \node at (3,-4) [] {\scriptsize $18$};
    \node at (3,-5) [] {\scriptsize $24$};
    \node at (3,-6) [] {\scriptsize $29$};
    \node at (3,-7) [] {\scriptsize $34$};
    \node at (3,-8) [] {\scriptsize $43$};
    
    \node at (4,-1) [] {\scriptsize $1$};
    \node at (4,-2) [] {\scriptsize $2$};
    \node at (4,-3) [] {\scriptsize $3$};
    \node at (4,-4) [] {\scriptsize $4$};
    \node at (4,-5) [] {\scriptsize $5$};
    \node at (4,-6) [] {\scriptsize $6$};
    \node at (4,-7) [] {\scriptsize $7$};
\node at (4,-8) [] {\scriptsize $8$};

    \node at (5.75,-1) [xshift=-0.7cm] {\scriptsize $(1,2)$};
    \node at (5.75,-2) [xshift=-0.5cm] {\scriptsize $(0,2)$};
    \node at (5.75,-3) [xshift=-0.3cm] {\scriptsize $(0,3)$};
    \node at (5.75,-4) [xshift=-0.1cm] {\scriptsize $(0,3)$};
    \node at (5.75,-5) [xshift=0.1cm] {\scriptsize $(0,4)$};
    \node at (5.75,-6) [xshift=0.3cm] {\scriptsize $(1,6)$};
    \node at (5.75,-7) [xshift=0.5cm] {\scriptsize $(3,8)$};
    \node at (5.75,-8) [xshift=0.7cm] {\scriptsize $(4,8)$};

    \node at (7.75,-1) [] {\scriptsize $4$}; \node at (7.75,-1) [xshift=1.5cm] {\scriptsize $1$}; \node at (9.25,-1) [xshift=1.5cm] {\scriptsize $1$};
    \node at (7.75,-2) [] {\scriptsize $12$}; \node at (7.75,-2) [xshift=1.5cm] {\scriptsize $4$}; \node at (9.25,-2) [xshift=1.5cm] {\scriptsize $0$};
    \node at (7.75,-3) [] {\scriptsize $18$};\node at (7.75,-3) [xshift=1.5cm] {\scriptsize $6$}; \node at (9.25,-3) [xshift=1.5cm] {\scriptsize $0$};
     \node at (7.75,-4) [] {\scriptsize $27$}; \node at (7.75,-4) [xshift=1.5cm] {\scriptsize $9$}; \node at (9.25,-4) [xshift=1.5cm] {\scriptsize $0$};
  \node at (7.75,-5) [] {\scriptsize $36$}; \node at (7.75,-5) [xshift=1.5cm] {\scriptsize $12$}; \node at (9.25,-5) [xshift=1.5cm] {\scriptsize $0$};
 \node at (7.75,-6) [] {\scriptsize $42$}; \node at (7.75,-6) [xshift=1.5cm] {\scriptsize $14$}; \node at (9.25,-6) [xshift=1.5cm] {\scriptsize $2$};
 \node at (7.75,-7) [] {\scriptsize $48$}; \node at (7.75,-7) [xshift=1.5cm] {\scriptsize $17$}; \node at (9.25,-7) [xshift=1.5cm] {\scriptsize $3$};
 \node at (7.75,-8) [] {\scriptsize $58$}; \node at (7.75,-8) [xshift=1.5cm] {\scriptsize $20$}; \node at (9.25,-8) [xshift=1.5cm] {\scriptsize $8$};

\end{scope}
    \begin{scope}[yshift=24pt]

        \node at (4.75,-9.05) [] {\footnotesize $\underline{\bh}=(\bh^{(0)}\hspace{-2pt},\bh^{(1)}\hspace{-2pt},\bh^{(2)}\hspace{-2pt},\bh^{(3)}\hspace{-2pt},\bh^{(4)}\hspace{-2pt},\bh^{(5)}\hspace{-2pt},\bh^{(6)}\hspace{-2pt},\bh^{(7)})$};
        \draw[thick] (2.5,-8.5) rectangle (11.5,-9.5);
        \draw[] (7,-8.5) -- (7,-9.5);
        \draw[] (8.5,-8.5) -- (8.5,-9.5);
        \draw[] (10,-8.5) -- (10,-9.5);
        
         \node at (7.75,-9) [] {\footnotesize $62$}; \node at (7.75,-9) [xshift=1.5cm] {\footnotesize $70$}; \node at (9.25,-9) [xshift=1.5cm] {\footnotesize $14$};

    \end{scope}
\end{tikzpicture}
}
    \caption{Hilbert functions  certifying the reducibility of nested Hilbert schemes on surfaces.}
    \label{fig:nested surface}
\end{figure}

\begin{corollary}\label{cor:5nestnonred}
For every $\underline{d}$ in \Cref{thm:nested-srf},  the nested Hilbert scheme $\hilbert{1,\underline{d}}{\BA^2}$ has at least one generically non-reduced component.
\end{corollary}

\begin{proof} Fix some $\underline{d}$ from \Cref{thm:nested-srf}. Let $V\subset \Hilb^{\underline{d}}$ be an irreducible component other than the smoothable one, which exists by \Cref{thm:nested-srf}. Now, $V$ is generically locally \'etale product of elementary components, say $E_1,\ldots,E_s$, where $E_i\subset \Hilb^{\underline{d}_i}\BA^2$, for some $\underline{d}_1,\ldots,\underline{d}_s\in\BZ^r$ with $\sum_{i=1}^s\underline{d}_i=\underline{d}$. Moreover, again by \Cref{thm:nested-srf} there is at least one index $i$, say $i=1$, such that $E_1$ parametrises  nestings with $(d_1)_r>1$. In this setting, we can apply \cite[Theorem 5]{UPDATES} and we get the existence of an elementary generically non-reduced component $\overline{E}_1\subset \Hilb^{1,\underline{d}_1}\BA^2$. To conclude, notice that there is an irreducible  non-reduced component  generically \'etale locally isomorphic to $\overline{E}_1\times E_2\times\cdots\times E_s$, and thus non reduced.
\end{proof}

\section{Reducibility of Hilbert schemes of points on three-folds} \label{sec:red3fold}

In this section we focus on smooth 3-folds. First, we provide many new examples of reducible Hilbert schemes of points, then we revisit Iarrobino's compressed algebras in terms of 2-step ideals. In the second part, we prove \Cref{thm: intro 3nest} about the nested case.

\paragraph{\it Known results.} The irreducibility of the Hilbert scheme of points on a smooth threefold is nowadays considered as one of the most challenging problems in the field. It is known that $\Hilb^d\BA^3$ is irreducible for $d\leqslant 11$, see \cite{Klemen,10points,JOACHIM}. The first reducible example was given by Iarrobino in \cite{Iarroredvery}, where he showed that $\Hilb^{103} \BA^3$ is reducible. Then, in \cite{IARRO} the same author considered compressed algebras and provided two more examples for $d=78,112$. The example in \cite{Iarroredvery} concerns very compressed algebras, and was refined in \cite{JOACHIM}, where it is shown that a very compressed algebra is smoothable if and only if its length is at most 95. On the other hand, 78 is nowadays the smallest length for which the Hilbert scheme of points is known to be reducible. It is worth mentioning that all these examples are achieved via a dimension-counting argument and do not provide any explicit examples of non-smoothable algebras of \textit{embedding dimension} 3, i.e.~with $\bh_A(1)=3$. As a consequence, it is not clear whether the loci considered by Iarrobino agree with irreducible components of the Hilbert scheme.   In conclusion, we note that the punctual Hilbert scheme, i.e.~the closed subset of $\Hilb^d\BA^3$ parametrising fat points, is known to be reducible for $d\geqslant 18$, \cite{JJ-Hilb-open-problems}. However, as explained above, it is unclear whether the Hilbert scheme of 18 points itself is reducible.

\begin{definition}\label{def:compressed}
The \textit{socle type} $\boldsymbol{e}_ {A}$ of a local Artinian $\BC$-algebra $(A,\Fm_A)$  is the Hilbert function of the graded $A$-module $(0_{\mathsf{gr}_{\Fm_A}(A)}:\Fm_{\mathsf{gr}_{\Fm_A}(A)})$.

    A local Artinian $R$-algebra   $A=R/I$ is compressed   if it has the maximum length among the local Artinian $R$-algebras having socle type $\be_A$. A compressed $R$-algebra $A=R/I$  is very compressed if there exists $k\geqslant 0 $ such that $\Fm^{k+1}\subset I\subset \Fm^{k}$. In this setting we say that the ideal $I$ (or that the algebra $R/I$) is compressed (resp.~very compressed) as well. In particular, very compressed implies compressed.
\end{definition}

\begin{theorem}[Iarrobino, \cite{IARRO}]\label{thm:78iarro}
For every point $[A]\in \Hilb^\bullet\BA^3$ corresponding to a compressed local Artinian algebra having socle type
\[
\be_{A}\equiv(0,0,0,0,0,0,2,5),\]
we have an equality 
\[
\bh_{A}\equiv (1,3,6,10,15,21,17,5),
\]
and hence $\length A= 78$. Moreover, the locus $V_{\mathrm{I}}$ parametrising these algebras has dimension $235=78\cdot 3+1$. As a consequence, the generic ideal of this form is non-smoothable.  
\end{theorem}

We obtain the locus in $V_{\mathrm{I}}\subset \hilbert{78}{\BA^3}$  as the locus parametrising   2-step ideals of order $k=6$ with $\oh_{6}=11$ and $\oh_{7}=31$. An example of ideal of this form is
\begin{equation}\label{eq:Iarroideal}
\begin{aligned} 
    I=  (&z^6,x^3 z^3,x y^3 z^2,xy^4z+x^3yz^2,x^3y^2z+xyz^4,x^4yz+x^2z^4,y^6+xy^2z^3+xz^5,\\ &xy^5+x^2y^3z+y^5z-y^3z^3, x^3y^3,x^4y^2-x^2y^3z+y^5z-x^4z^2-y^2z^4,y^2z^5+x^6).
\end{aligned}
\end{equation}

\begin{remark}\label{rem:remonsigns}
 We remark that, even though we present an example of compressed algebra which is 2-step, they are not all of this form. Similarly, not all 2-step ideals are compressed, see \Cref{fig:newA3k6} and \Cref{fig:newA3k7-8}. Therefore, many of our examples of non-smoothable points of embedding dimension three are new.
\end{remark}

\bigskip

We begin by observing that for $n=3$ the absolute minimum of the function $\Theta_{3,k,b}(\oh_k,\oh_{k+1})$ is
\[
\dfrac{k^{4}+8\,k^{3}+\left(21-6b\right)k^{2}+\left(20-14b\right)k-6b^{2}-120}{40}.
\]
For $k \geqslant 2$, we need $b$ to be positive in order to obtain a non-positive minimum. However, there are no pairs $(\oh_k,\oh_{k+1})$ for which the maximum of $b$ given by the Betti number $\beta_{2,k+2}(L_{\bh})$ of the lexicographic ideal $L_{\bh}$ is sufficient to make the value of $\Theta_{3,k,b}(\oh_k,\oh_{k+1})$ non-positive. Hence, the potential TNT area is empty and no generically reduced elementary component can be discovered using 2-step ideals. Note that this observation agrees with the general thought that finding a generically reduced elementary component in $\Hilb^\bullet \mathbb{A}^3$ would be very surprising.

In order to  exhibit loci parametrising non-smoothable algebras, we look for Hilbert strata with dimension greater than or equal to the dimension $3d$ of the smoothable component of $\hilbert{d}{\mathbb{A}^3}$. The Hessian matrix
\[
\Hess \Delta_{3,1,k} = \left[\begin{array}{cc} -2 & 2 \\ 2 & -2\end{array}\right]
\]
is singular with one negative eigenvalue. Thus, the graph of $\Delta_{3,1,k}$ is a non-rotational paraboloid with concavity facing downwards and symmetric with respect to a line parallel to the eigenvector $[1,1]$.

For $k\leqslant 5$, there are no pairs $(\oh_k,\oh_{k+1})$ for which the function $\Delta_{3,1,k}$ is non-negative. The first examples we find are for $k=6$, see \Cref{fig:newA3k6} and the table in it. As desired, these pairs correspond to 2-step ideals with no linear syzygies or very few linear syzygies. By \Cref{cor:full no syz} and \Cref{cor:veryfew}, the Hilbert stratum cannot be contained in the smoothable component of the corresponding Hilbert scheme so certifying its reducibility. 

\begin{figure}[!ht]
\centering
    \begin{tikzpicture}

    \begin{scope}[scale=0.2]
        \fill [fill=black!10,] (0,0) -- (0,36) -- (12,36) -- cycle;
\fill [fill=black!50,] (0,0) -- (12,36) -- (27/2,36) -- cycle;
\fill [fill=black!30,] (0,0) -- (18,36) -- (27/2,36) -- cycle;

\fill[pattern={Lines[angle=135,distance=2pt,line width=0.05pt]},opacity=0.3,pattern color=ForestGreen] (0,80/3) -- (28,36) -- (0,36) -- cycle;
\draw[ForestGreen,opacity=0.5,] (0,80/3) -- (28,36) -- (0,36) -- cycle;

\fill [pattern={Lines[angle=45,distance=3pt,line width=0.05pt]},opacity=0.3,pattern color=red] plot[rotate around={45:(0,0)},domain=10.38:17.2,smooth,variable=\t] ({-sqrt(2)*\t*\t + 40*\t - 357/sqrt(2)},{\t}) -- (0,36) -- (28,36) -- (0,0) -- cycle;

\fill [pattern={Lines[angle=45,distance=3pt,line width=0.05pt]},opacity=0.3,pattern color=blue] plot[rotate around={45:(0,0)},domain=10.38:17.2,smooth,variable=\t] ({-sqrt(2)*\t*\t + 40*\t - 357/sqrt(2)},{\t}) --cycle;

\draw[NavyBlue,opacity=0.5,thick]
  plot[rotate around={45:(0,0)},domain=10.38:17.2,smooth,variable=\t] ({-sqrt(2)*\t*\t + 40*\t - 357/sqrt(2)},{\t}); 
\draw [dotted,thin] (0,0.05) grid (27.95,35.95);
\fill [fill=white] (0,0) -- (1,3) -- (2,5) -- (3,6) -- (4,8) -- (5,9) -- (6,10) -- (7,12) -- (8,13) -- (9,14) -- (10,15) -- (11,17) -- (12,18) -- (13,19) -- (14,20) -- (15,21) -- (16,23) -- (17,24) -- (18,25) -- (19,26) -- (20,27) -- (21,28) -- (22,30) -- (23,31) -- (24,32) -- (25,33) -- (26,34) -- (27,35) -- (28,36) -- (28,0) -- cycle;
\draw [] (0,0) -- (1,3) -- (2,5) -- (3,6) -- (4,8) -- (5,9) -- (6,10) -- (7,12) -- (8,13) -- (9,14) -- (10,15) -- (11,17) -- (12,18) -- (13,19) -- (14,20) -- (15,21) -- (16,23) -- (17,24) -- (18,25) -- (19,26) -- (20,27) -- (21,28) -- (22,30) -- (23,31) -- (24,32) -- (25,33) -- (26,34) -- (27,35) -- (28,36) --node[above]{\footnotesize $\oh_6$} (0,36) --node[left]{\footnotesize $\oh_7$} cycle;

    \draw [dashed,ultra thin] (6,36) -- (18,24) node[right] {\tiny $d=78$};
    \draw [dashed,ultra thin] (1,36) -- (16,21) node[right] {\tiny $d=83$};
\draw [dashed,ultra thin] (0,32) -- (13.5,18.5) node[right] {\tiny $d=88$};
\draw [dashed,ultra thin] (0,27) -- (11,16) node[right] {\tiny $d=93$};
\draw [dashed,ultra thin] (0,22) -- (9,13) node[right] {\tiny $d=98$};
\draw [dashed,ultra thin] (0,17) -- (7,10) node[right] {\tiny $d=103$};

\node at (11,30) [draw=NavyBlue,very thick,fill=NavyBlue,diamond,inner sep=0.66pt] {};
\node at (11,31) [draw=NavyBlue,thick,fill=white,diamond,inner sep=0.8pt] {};
\node at (10,28) [draw=NavyBlue,very thick,fill=NavyBlue,diamond,inner sep=0.66pt] {};
\node at (10,29) [draw=NavyBlue,very thick,fill=NavyBlue,diamond,inner sep=0.66pt] {};
\node at (10,30) [draw=NavyBlue,very thick,fill=NavyBlue,diamond,inner sep=0.66pt] {};
\node at (10,31) [draw=NavyBlue,very thick,fill=NavyBlue,diamond,inner sep=0.66pt] {};
\node at (9,27) [draw=NavyBlue,very thick,fill=NavyBlue,diamond,inner sep=0.66pt] {};
\node at (9,28) [draw=NavyBlue,very thick,fill=NavyBlue,diamond,inner sep=0.66pt] {};
\node at (9,29) [draw=NavyBlue,very thick,fill=NavyBlue,diamond,inner sep=0.66pt] {};
\node at (9,30) [draw=NavyBlue,very thick,fill=NavyBlue,diamond,inner sep=0.66pt] {};
\node at (8,25) [draw=NavyBlue,very thick,fill=NavyBlue,diamond,inner sep=0.66pt] {};
\node at (8,26) [draw=NavyBlue,very thick,fill=NavyBlue,diamond,inner sep=0.66pt] {};
\node at (8,27) [draw=NavyBlue,very thick,fill=NavyBlue,diamond,inner sep=0.66pt] {};
\node at (8,28) [draw=NavyBlue,very thick,fill=NavyBlue,diamond,inner sep=0.66pt] {};
\node at (8,29) [draw=NavyBlue,very thick,fill=NavyBlue,diamond,inner sep=0.66pt] {};
\node at (8,30) [draw=NavyBlue,very thick,fill=NavyBlue,diamond,inner sep=0.66pt] {};
\node at (7,24) [draw=NavyBlue,very thick,fill=NavyBlue,diamond,inner sep=0.66pt] {};
\node at (7,25) [draw=NavyBlue,very thick,fill=NavyBlue,diamond,inner sep=0.66pt] {};
\node at (7,26) [draw=NavyBlue,very thick,fill=NavyBlue,diamond,inner sep=0.66pt] {};
\node at (7,27) [draw=NavyBlue,very thick,fill=NavyBlue,diamond,inner sep=0.66pt] {};
\node at (7,28) [draw=NavyBlue,very thick,fill=NavyBlue,diamond,inner sep=0.66pt] {};
\node at (7,29) [draw=NavyBlue,very thick,fill=NavyBlue,diamond,inner sep=0.66pt] {};
\node at (6,23) [draw=NavyBlue,very thick,fill=NavyBlue,diamond,inner sep=0.66pt] {};
\node at (6,24) [draw=NavyBlue,very thick,fill=NavyBlue,diamond,inner sep=0.66pt] {};
\node at (6,25) [draw=NavyBlue,very thick,fill=NavyBlue,diamond,inner sep=0.66pt] {};
\node at (6,26) [draw=NavyBlue,very thick,fill=NavyBlue,diamond,inner sep=0.66pt] {};
\node at (6,27) [draw=NavyBlue,very thick,fill=NavyBlue,diamond,inner sep=0.66pt] {};
\node at (6,28) [draw=NavyBlue,very thick,fill=NavyBlue,diamond,inner sep=0.66pt] {};
\node at (5,21) [draw=NavyBlue,very thick,fill=NavyBlue,diamond,inner sep=0.66pt] {};
\node at (5,22) [draw=NavyBlue,very thick,fill=NavyBlue,diamond,inner sep=0.66pt] {};
\node at (5,23) [draw=NavyBlue,very thick,fill=NavyBlue,diamond,inner sep=0.66pt] {};
\node at (5,24) [draw=NavyBlue,very thick,fill=NavyBlue,diamond,inner sep=0.66pt] {};
\node at (5,25) [draw=NavyBlue,very thick,fill=NavyBlue,diamond,inner sep=0.66pt] {};
\node at (5,26) [draw=NavyBlue,very thick,fill=NavyBlue,diamond,inner sep=0.66pt] {};
\node at (5,27) [draw=NavyBlue,very thick,fill=NavyBlue,diamond,inner sep=0.66pt] {};
\node at (5,28) [draw=NavyBlue,very thick,fill=NavyBlue,diamond,inner sep=0.66pt] {};
\node at (4,20) [draw=NavyBlue,very thick,fill=NavyBlue,diamond,inner sep=0.66pt] {};
\node at (4,21) [draw=NavyBlue,very thick,fill=NavyBlue,diamond,inner sep=0.66pt] {};
\node at (4,22) [draw=NavyBlue,very thick,fill=NavyBlue,diamond,inner sep=0.66pt] {};
\node at (4,23) [draw=NavyBlue,very thick,fill=NavyBlue,diamond,inner sep=0.66pt] {};
\node at (4,24) [draw=NavyBlue,very thick,fill=NavyBlue,diamond,inner sep=0.66pt] {};
\node at (4,25) [draw=NavyBlue,very thick,fill=NavyBlue,diamond,inner sep=0.66pt] {};
\node at (4,26) [draw=NavyBlue,very thick,fill=NavyBlue,diamond,inner sep=0.66pt] {};
\node at (4,27) [draw=NavyBlue,very thick,fill=NavyBlue,diamond,inner sep=0.66pt] {};
\node at (3,19) [draw=NavyBlue,very thick,fill=NavyBlue,diamond,inner sep=0.66pt] {};
\node at (3,20) [draw=NavyBlue,very thick,fill=NavyBlue,diamond,inner sep=0.66pt] {};
\node at (3,21) [draw=NavyBlue,very thick,fill=NavyBlue,diamond,inner sep=0.66pt] {};
\node at (3,22) [draw=NavyBlue,very thick,fill=NavyBlue,diamond,inner sep=0.66pt] {};
\node at (3,23) [draw=NavyBlue,very thick,fill=NavyBlue,diamond,inner sep=0.66pt] {};
\node at (3,24) [draw=NavyBlue,very thick,fill=NavyBlue,diamond,inner sep=0.66pt] {};
\node at (3,25) [draw=NavyBlue,very thick,fill=NavyBlue,diamond,inner sep=0.66pt] {};
\node at (3,26) [draw=NavyBlue,very thick,fill=NavyBlue,diamond,inner sep=0.66pt] {};
\node at (2,18) [draw=NavyBlue,very thick,fill=NavyBlue,diamond,inner sep=0.66pt] {};
\node at (2,19) [draw=NavyBlue,very thick,fill=NavyBlue,diamond,inner sep=0.66pt] {};
\node at (2,20) [draw=NavyBlue,very thick,fill=NavyBlue,diamond,inner sep=0.66pt] {};
\node at (2,21) [draw=NavyBlue,very thick,fill=NavyBlue,diamond,inner sep=0.66pt] {};
\node at (2,22) [draw=NavyBlue,very thick,fill=NavyBlue,diamond,inner sep=0.66pt] {};
\node at (2,23) [draw=NavyBlue,very thick,fill=NavyBlue,diamond,inner sep=0.66pt] {};
\node at (2,24) [draw=NavyBlue,very thick,fill=NavyBlue,diamond,inner sep=0.66pt] {};
\node at (2,25) [draw=NavyBlue,very thick,fill=NavyBlue,diamond,inner sep=0.66pt] {};
\node at (1,16) [draw=NavyBlue,very thick,fill=NavyBlue,diamond,inner sep=0.66pt] {};
\node at (1,17) [draw=NavyBlue,very thick,fill=NavyBlue,diamond,inner sep=0.66pt] {};
\node at (1,18) [draw=NavyBlue,very thick,fill=NavyBlue,diamond,inner sep=0.66pt] {};
\node at (1,19) [draw=NavyBlue,very thick,fill=NavyBlue,diamond,inner sep=0.66pt] {};
\node at (1,20) [draw=NavyBlue,very thick,fill=NavyBlue,diamond,inner sep=0.66pt] {};
\node at (1,21) [draw=NavyBlue,very thick,fill=NavyBlue,diamond,inner sep=0.66pt] {};
\node at (1,22) [draw=NavyBlue,very thick,fill=NavyBlue,diamond,inner sep=0.66pt] {};
\node at (1,23) [draw=NavyBlue,very thick,fill=NavyBlue,diamond,inner sep=0.66pt] {};
\node at (1,24) [draw=NavyBlue,very thick,fill=NavyBlue,diamond,inner sep=0.66pt] {};
\node at (1,25) [draw=NavyBlue,very thick,fill=NavyBlue,diamond,inner sep=0.66pt] {};

\end{scope}

\begin{scope}[xscale=0.55,yscale=0.35,shift={(7,11.5)}]

\draw[] (2,1.5) -- (2,-11.5);
\draw[] (4,1.5) -- (4,-11.5);
\draw[] (9,1.5) -- (9,-11.5);
\draw[thin] (10+1/3,0.4) -- (10+1/3,-11.5);
\draw[thin] (11+2/3,0.4) -- (11+2/3,-11.5);
\draw[] (13,1.5) -- (13,-11.5);
\draw[] (15,1.5) -- (15,-11.5);
\draw[] (18,1.5) -- (18,-11.5);

\draw[] (2,1.5)  -- (18,1.5);
\draw[] (2,-.5)  -- (18,-0.5);

\draw[thin] (2,-1.5)  -- (18,-1.5);
\draw[very thin] (4,-2.5)  -- (18,-2.5);
\draw[thin] (2,-3.5)  -- (18,-3.5);
\draw[thin] (2,-4.5)  -- (18,-4.5);
\draw[very thin] (4,-5.5)  -- (18,-5.5);
\draw[thin] (2,-6.5)  -- (18,-6.5);
\draw[very thin] (4,-7.5)  -- (18,-7.5);
\draw[very thin] (4,-8.5)  -- (18,-8.5);
\draw[thin] (2,-9.5)  -- (18,-9.5);
\draw[very thin] (4,-10.5)  -- (18,-10.5);
\draw[] (2,-11.5)  -- (18,-11.5);

\node at (3,0.45) [] {\small $\vert\bh\vert$};
\node at (6.5,0.5) [] {\small $\bh$};
\node at (11,0.95) [] {\small $\dim \mathsf{T}^{=\,\bullet}$};
\node at (9+2/3,-0.) [] {\footnotesize $-1$};
\node at (11,-0.) [] {\footnotesize $0$};
\node at (12+1/3,-0.) [] {\footnotesize $1$};
\node at (14,0.5) [] {\small $\Delta_{3,1,6}$};
\node at (16.5,0.5) [] {\small Type};

\node at (3,-1)[] {\footnotesize $78$};
\node at (3,-2.5)[] {\footnotesize $79$};
\node at (3,-4)[] {\footnotesize $80$};
\node at (3,-5.5)[] {\footnotesize $81$};
\node at (3,-8)[] {\footnotesize $82$};
\node at (3,-10.5)[] {\footnotesize $83$};

\node at (6.5,-1)[] {\scriptsize $(1,3,6,10,15,21,17,5)$};
\node at (9+2/3,-1) [] {\scriptsize $122$};
\node at (11,-1) [] {\scriptsize $177$};
\node at (12+1/3,-1) [] {\scriptsize $55$};
\node at (14,-1) [] {\scriptsize $1$};
\node at (16.5,-1) [] {\scriptsize very few syz};

\node at (6.5,-2)[] {\scriptsize $(1,3,6,10,15,21,17,6)$};
\node at (9+2/3,-2) [] {\scriptsize $108$};
\node at (11,-2) [] {\scriptsize $169$};
\node at (12+1/3,-2) [] {\scriptsize $66$};
\node at (14,-2) [] {\scriptsize $1$};
\node at (16.5,-2) [] {\scriptsize very few syz};

\node at (6.5,-3)[] {\scriptsize $(1,3,6,10,15,21,18,5)$};
\node at (9+2/3,-3) [] {\scriptsize $138$};
\node at (11,-3) [] {\scriptsize $185$};
\node at (12+1/3,-3) [] {\scriptsize $50$};
\node at (14,-3) [] {\scriptsize $1$};
\node at (16.5,-3) [yshift=-1pt] {\scriptsize no syz};

\node at (6.5,-4)[] {\scriptsize $(1,3,6,10,15,21,18,6)$};
\node at (9+2/3,-4) [] {\scriptsize $120$};
\node at (11,-4) [] {\scriptsize $180$};
\node at (12+1/3,-4) [] {\scriptsize $60$};
\node at (14,-4) [] {\scriptsize $3$};
\node at (16.5,-4) [yshift=-1pt] {\scriptsize no syz};

\node at (6.5,-5)[] {\scriptsize $(1,3,6,10,15,21,18,7)$};
\node at (9+2/3,-5) [] {\scriptsize $108$};
\node at (11,-5) [] {\scriptsize $173$};
\node at (12+1/3,-5) [] {\scriptsize $70$};
\node at (14,-5) [] {\scriptsize $3$};
\node at (16.5,-5) [] {\scriptsize very few syz};

\node at (6.5,-6)[] {\scriptsize $(1,3,6,10,15,21,19,6)$};
\node at (9+2/3,-6) [] {\scriptsize $138$};
\node at (11,-6) [] {\scriptsize $189$};
\node at (12+1/3,-6) [] {\scriptsize $54$};
\node at (14,-6) [] {\scriptsize $3$};
\node at (16.5,-6) [yshift=-1pt] {\scriptsize no syz};

\node at (6.5,-7)[] {\scriptsize $(1,3,6,10,15,21,18,8)$};
\node at (9+2/3,-7) [] {\scriptsize $102$};
\node at (11,-7) [] {\scriptsize $164$};
\node at (12+1/3,-7) [] {\scriptsize $80$};
\node at (14,-7) [] {\scriptsize $1$};
\node at (16.5,-7) [] {\scriptsize very few syz};

\node at (6.5,-8)[] {\scriptsize $(1,3,6,10,15,21,19,7)$};
\node at (9+2/3,-8) [] {\scriptsize $122$};
\node at (11,-8) [] {\scriptsize $185$};
\node at (12+1/3,-8) [] {\scriptsize $63$};
\node at (14,-8) [] {\scriptsize $5$};
\node at (16.5,-8) [yshift=-1pt] {\scriptsize no syz};

\node at (6.5,-9)[] {\scriptsize $(1,3,6,10,15,21,20,6)$};
\node at (9+2/3,-9) [] {\scriptsize $162$};
\node at (11,-9) [] {\scriptsize $196$};
\node at (12+1/3,-9) [] {\scriptsize $48$};
\node at (14,-9) [] {\scriptsize $1$};
\node at (16.5,-9) [yshift=-1pt] {\scriptsize no syz};

\node at (6.5,-10)[] {\scriptsize $(1,3,6,10,15,21,19,8)$};
\node at (9+2/3,-10) [] {\scriptsize $112$};
\node at (11,-10) [] {\scriptsize $179$};
\node at (12+1/3,-10) [] {\scriptsize $72$};
\node at (14,-10) [] {\scriptsize $5$};
\node at (16.5,-10) [yshift=-1pt] {\scriptsize no syz};

\node at (6.5,-11)[] {\scriptsize $(1,3,6,10,15,21,20,7)$};
\node at (9+2/3,-11) [] {\scriptsize $142$};
\node at (11,-11) [] {\scriptsize $195$};
\node at (12+1/3,-11) [] {\scriptsize $56$};
\node at (14,-11) [] {\scriptsize $5$};
\node at (16.5,-11) [yshift=-1pt] {\scriptsize no syz};

\end{scope}
    \end{tikzpicture}

    \caption[Hilbert strata of 2-step ideals of order 6 certifying the reducibility of $\Hilb^d \mathbb{A}^3$ (see \Cref{legenda} for the complete picture legend).]{ Hilbert strata of 2-step ideals of order 6 certifying the reducibility of $\Hilb^d \mathbb{A}^3$. The green area  \raisebox{-2pt}{\begin{tikzpicture} \fill[ pattern={Lines[angle=135,distance=2pt,line width=0.05pt]},opacity=0.3,pattern color=ForestGreen] (0,0) rectangle (0.35,0.35); \draw[ForestGreen,opacity=0.5,thick] (0,0) rectangle (0.35,0.35); \end{tikzpicture}} contains Hilbert functions of compressed algebras (see \Cref{legenda} for the complete picture legend).}
    \label{fig:newA3k6}
\end{figure}

For $k\geqslant 7$, there are a lot of Hilbert strata not contained in the smoothable component corresponding to 2-step ideals with few syzygies (see \Cref{fig:newA3k7-8}). In those cases, we apply \Cref{thm: homo loco syz} and compute explicitly the dimension of the generic fibre of $\psi_{\bh}$. See the ancillary \textit{Macaulay2} file \href{www.paololella.it/software/reducibility-Hilbert-schemes.m2}{\tt reducibility-Hilbert-schemes.m2} to produce and check the list of Hilbert strata not contained in the smoothable component.

\begin{figure}[!ht]
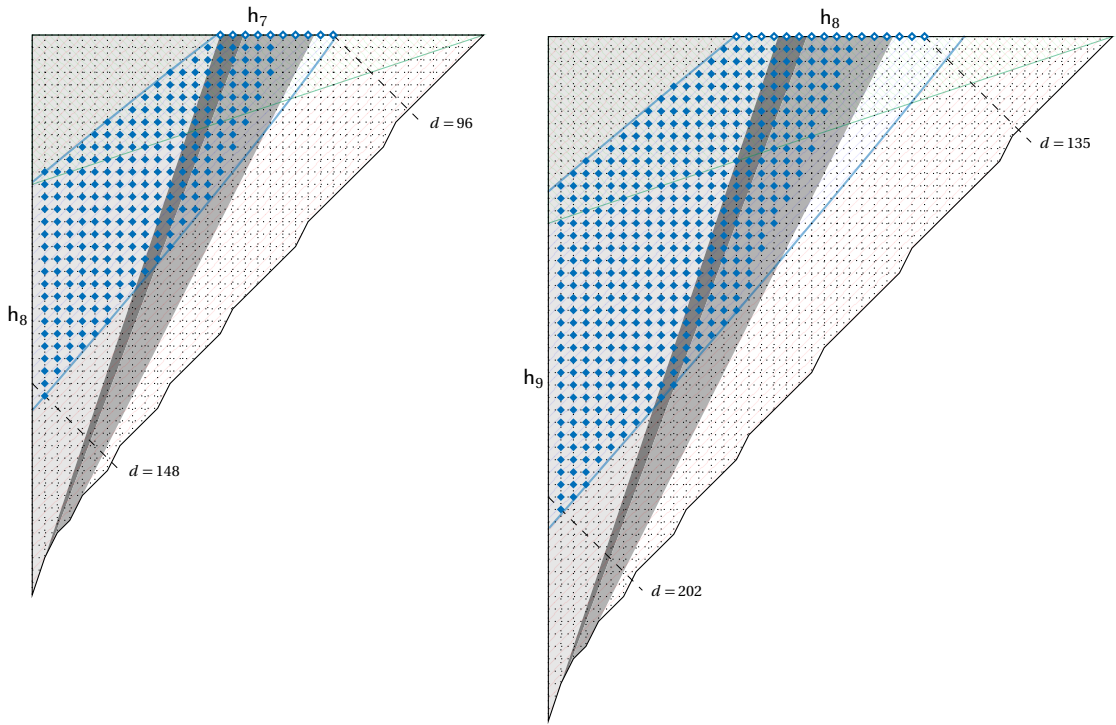

    \centering


 \caption[Hilbert strata of 2-step ideals of order 7 and 8 certifying the reducibility of $\Hilb^d \mathbb{A}^3$ (see \Cref{legenda} for the complete picture legend).]{ Hilbert strata of 2-step ideals of order 7 and 8 certifying the reducibility of $\Hilb^d \mathbb{A}^3$. The green area  \raisebox{-2pt}{%
} contains Hilbert functions of compressed algebras (see \Cref{legenda} for the complete picture legend).}
    \label{fig:newA3k7-8}
\end{figure}

The smallest example that we find has length 78 and it agrees with the smallest example of non-smoothable point yet known in the literature and given firstly in \cite{IARRO}. We give an example of such a point in \eqref{eq:Iarroideal}. It is worth mentioning that, for $k=7$ and $d=96$ we recover the {smallest very compressed non-smoothable algebras, see \cite{Iarroredvery,JOACHIM}.

\begin{example}\label{exa:iarrobino}
We describe in detail the family of 2-step ideals proving the reducibility of $\hilbert{78}{\mathbb{A}^3}$. This is a different point of view on the original example of Iarrobino \cite{IARRO}.

For $k=6$, the pair $(\oh_6,\oh_7) = (11,31)$ corresponds to a 2-step Hilbert function with very few linear syzygies, i.e.~$-\sss_\bh = 3\oh_6 - \oh_7 = 2$ and $3(-\sss_\bh) \leqslant 11$. By \Cref{lem:pseudo canonical form}, up to a change of basis, the generic morphism $\varphi : R_6^{\oplus 11} \to R_7^{\oplus 2}$ in $\mathscr{L}_\bh$ is induced by the matrix
\[
\left[\begin{array}{ccccccccccc} x & y & z & 0 & 0 & 0 & 0 & 0 & 0 & 0 & 0\\  0 & 0 & 0 & x & y & z & 0 & 0 & 0 & 0 & 0 \end{array}\right]
\]
Thus, the generic homogeneous 2-step ideal with Hilbert function $\bh = (1,3,6,10,15,21,17,5)$ has a set of generators 
\[
\{p_1,p_2,p_3\} \cup \{p_4,p_5,p_6\}\cup \{p_7,p_8,p_9,p_{10},p_{11}\}
\]
where the triples $\{p_1,p_2,p_3\}$ and $\{p_4,p_5,p_6\}$ are in the kernel of the morphism $R_6^{\oplus 3} \xrightarrow{[ x\ y\ z]} R_7$ and $\{p_7,p_8,p_9,$ $p_{10},p_{11}\}$ are linearly independent of the other six generators, see \Cref{subsec:2-stepfew}.

 We can also give a determinantal description of homogeneous 2-step ideals with very few linear syzygies. Via the Koszul complex, the triples $\{p_1,p_2,p_3\}$ and $\{p_4,p_5,p_6\}$ are in the image of the morphism $R^{\oplus 3}_5 \to R_6^{\oplus 3}$ described by the matrix
\[
\left[\begin{array}{ccc} y & z & 0 \\ -x & 0 & z \\ 0 & -x & -y\end{array}\right]
\]
that is
\[
p_{i} = yq_ {i}+ zq_{i+1},\quad p_{i+1} = - xq_i + zq_{i+2},\quad p_{i+2} = - xq_{i+1} - yq_{i+2},\quad i\in\{1,4\},\quad q_j \in R_5.
\]
The ideal generated by each triple is a determinantal ideal
\[
(p_i,p_{i+1},p_{i+2}) = (yq_i + zq_{i+1}, -xq_i + zq_{i+2}, -xq_{i+1} -yq_{i+2}) = \left(\left\{\textnormal{rk} \left( \begin{array}{ccc} x & y & z \\ q_{i+2} & -q_{i+1} & q_i\end{array}\right) \leqslant 1\right\}\right),\ i\in\{1,4\}
\]
and the same holds also for the sum of ideals so that
\begin{equation}\label{eq:IarrobinoMatrix}
(p_1,p_2,p_3,p_4,p_5,p_6) + (p_7,p_8,p_9,p_{10},p_{11}) = \left(\left\{\textnormal{rk} \left( \begin{array}{ccc} x & y & z \\ q_3 & -q_2 & q_1 \\ q_6 & -q_5 & q_4\end{array}\right) \leqslant 1\right\}\right) + (p_7,p_8,p_9,p_{10},p_{11}).
\end{equation}
Thanks to this description, we obtain a further confirmation of the dimension of the Hilbert stratum. The ideal generated by the $2 \times 2$ minors of the matrix in \Cref{eq:IarrobinoMatrix} does not change if we act with row and column operations. Hence, generically we may assume
\[q_3 = y^5 +a_1 y^3z^2 + \cdots + a_4 z^5,\quad q_6 = y^4z +a_5 y^3z^2 + \cdots + a_8 z^5,\quad q_i \in R_5 \text{~for~}i = 2,3,5,6,\]
and the dimension of this family of polynomials is $8 + 4 \rrr_5  = 92$. The last 5 generators have to be taken in a complement of $\Span(p_1,\ldots,p_6) \subset R_6$, i.e.~they correspond to a point in a Grassmannian $\Gr(5,\rrr_6 - 6)$. The dimension of this second family of polynomials is $5(\rrr_6-11) = 85$. Overall, we get 
\[
92 + 85 = 177 = \dim \mathscr{H}_{\bh}^3,
\]
which, together with \Cref{thm:1nonob}, gives $\dim {H}_{\bh}^3=235-3$ as expected.
\end{example}
 
\subsection{Nested Hilbert schemes of points on three-folds}\label{subsec:3dimnest}

For $n=3$, the Hessian matrix \eqref{eq:hessian} is
\[
\Hess\Delta_{3,r,k} =
\begin{bmatrix} 
-2 & 2 & 0 & \phantom{2}& \phantom{2}& \phantom{2} \\
2 & -2 & 2 & 0 \\
0 & 2 & -2 & 2 & 0 \\
 & & \ddots & \ddots & \ddots \\
 & & 0 & 2 & -2 & 2 \\
 & & & 0 & 2 & -2
\end{bmatrix}
\]
and for $r \geqslant 2$ it is non-singular with positive and negative eigenvalues. Thus, $\Delta_{3,r,k}$ has a unique critical point that is a saddle point and certainly assumes positive values. 

Recall that given $\underline{d}= (d_1,\ldots,d_r)$ such that $\hilbert{\underline{d}}(\mathbb{A}^n)$ is reducible, then $\hilbert{\underline{d}'}(\mathbb{A}^n)$ with $\underline{d}'=(d_1,\ldots,$ $d_i+1,\ldots,d_r+1)$ and $\underline{d}'=(d_1,\ldots,d_i,d_i+1,d_{i+1},\ldots,d_r)$ is also reducible. This translates into a partial order on the set of integer sequences of arbitrary length. We look for natural points in $\mathcal{D}_{\mathbb{N}}$ such that $\Delta_{3,r,k}$ is non-negative and the corresponding sequence $\vert\underline{\bh}\vert = (\vert\bh^{(1)}\vert,\ldots,\vert\bh^{(r)}\vert)$ is a minimal element with respect to this partial order.

As for the surface case, we list some examples of reducible nested Hilbert schemes over $\BA^3$ in the following statement corresponding to  \Cref{thm: intro 3nest} in the introduction.

\begin{theorem}\label{thm:nested-threefold}
    If $\underline{d}$ is one of the following increasing sequences of positive integers
    \begin{enumerate}[\rm (a)]
        \item $\underline{d} \in \{ (14,24), (15,24), (13,26) \} \subset \BZ^2$,
        \item $\underline{d} \in \left\{\begin{array}{c}(7,13,17), (7,12,18), (6,13,18),  (8,13,18), (6,12,20),(8,12,20), (5,13,20), \\ (5,14,20), (4,13,21),(3,14,21),(4,14,21), (6,11,22),(7,11,22),(3,13,22),\\ (4,12,23),(5,12,23),(2,14,23), (2,15,23), (3,12,24), (2,13,24), (2,12,25) \end{array}\right\} \subset \BZ^3$
    \end{enumerate}
    then the nested Hilbert scheme $\hilbert{\underline{d}}{\BA^3}$ is reducible.
\end{theorem}
\begin{proof}
We consider configurations of nested 2-step ideals with first order $k=1$ or $k=2$ and we explore all natural points in $\mathcal{D}_{\mathbb{N}}$ to find the minimal sequences $\underline{d}$. The list of sequences in the statement contains values of $\underline d$ for which the reducibility cannot be deduced from the reducibility of another Hilbert scheme.

In \Cref{fig:nested threefold}, there is a detailed description of some of the 2-step Hilbert functions leading to these sequences. The full list is available in the ancillary \textit{Macaulay2} file \href{www.paololella.it/software/reducibility-nested-Hilbert-schemes.m2}{\tt reducibility-nested-Hilbert- schemes.m2}.
As in the case $n=2$, some of the natural points lie on the boundary of $\mathcal{D}$ and some ideals in the configurations are in fact 1-step ideals (of order $k+i$ or $k+i+1$).
\end{proof}

\begin{figure}[!ht]
    \centering
\subfigure[Examples of 2-step Hilbert functions certifying the reducibility of $\hilbert{(\vert\bh^{(0)}\vert,\vert\bh^{(1)}\vert)}{\mathbb{A}^3}$]{\label{subfig: nested 3 r 5}    
\begin{tikzpicture}[yscale=0.4]
    \draw [] (0,1) -- (11.5,1);
    \draw [] (0,-0.5) -- (11.5,-0.5);
    
    \draw [] (0,1) -- (0,-0.5);
    \draw [] (2.5,1) -- (2.5,-0.5);
    \draw [] (3.5,1) -- (3.5,-0.5);
    \draw [] (4.5,1) -- (4.5,-0.5);
    \draw [] (7,1) -- (7,-0.5);
    \draw [] (8.5,1) -- (8.5,-0.5);
    \draw [] (10,1) -- (10,-0.5);
    \draw [] (11.5,1) -- (11.5,-0.5);
    
    \node at (1.25,0.25) [] {\footnotesize $\bh^{(i)}$};
    \node at (3,0.25) [] {\footnotesize $\left\vert \bh^{(i)}\right\vert$};
    \node at (4,0.25) [] {\footnotesize $k+i$};
    \node at (5.75,0.25) [] {\footnotesize $\big(\oh_{k+i}^{(i)},\oh_{k+i+1}^{(i)}\big)$};
    \node at (7.75,0.25) [] {\footnotesize $\dim \mathsf{T}^{<0}$};
    \node at (7.75,0.25) [xshift=1.5cm] {\footnotesize $\dim \mathsf{T}^{=0}$};
    \node at (9.25,0.25) [xshift=1.5cm] {\footnotesize $\dim \mathsf{T}^{=1}$};

    \begin{scope}[yshift= -8pt,yscale=0.85]
    \draw [thin] (0,-0.5) -- (11.5,-0.5);
    \draw [thin] (0,-1.5) -- (11.5,-1.5);
    \draw [thin] (0,-2.5) -- (11.5,-2.5);

    \draw [thin] (0,-0.5) -- (0,-2.5);
    \draw [thin] (2.5,-0.5) -- (2.5,-2.5);
    \draw [thin] (3.5,-0.5) -- (3.5,-2.5);
    \draw [thin] (4.5,-0.5) -- (4.5,-2.5);
    \draw [thin] (7,-0.5) -- (7,-2.5);
    \draw [thin] (8.5,-0.5) -- (8.5,-2.5);
    \draw [thin] (10,-0.5) -- (10,-2.5);
    \draw [thin] (11.5,-0.5) -- (11.5,-2.5);
    
    \node at (1.25,-1) [] {\parbox{1.75cm}{\scriptsize $\bh^{(0)}=(1,3,6,4)$}};
    \node at (1.25,-2) [] {\parbox{1.75cm}{\scriptsize $\bh^{(1)}=(1,3,6,9,5)$}};
    
    \node at (3,-1) [] {\scriptsize $14$};
    \node at (3,-2) [] {\scriptsize $24$};
    
    \node at (4,-1) [] {\scriptsize $2$};
    \node at (4,-2) [] {\scriptsize $3$};

    \node at (5.75,-1)[xshift=-4pt] {\scriptsize $(0,6)$};
    \node at (5.75,-2)[xshift=4pt] {\scriptsize $(1,10)$};

    \node at (7.75,-1) [] {\scriptsize $24$}; \node at (7.75,-1) [xshift=1.5cm] {\scriptsize $24$}; \node at (9.25,-1) [xshift=1.5cm] {\scriptsize $0$};
    \node at (7.75,-2) [] {\scriptsize $39$}; \node at (7.75,-2) [xshift=1.5cm] {\scriptsize $44$}; \node at (9.25,-2) [xshift=1.5cm] {\scriptsize $5$};
\end{scope}
\begin{scope}[yshift=0pt]
    \node at (4.75,-3.05) [] {\footnotesize $\underline{\bh}=(\bh^{(0)},\bh^{(1)})$};
        \draw[thick,] (2.5,-2.5) rectangle (11.5,-3.5);
        \draw[] (7,-2.5) -- (7,-3.5);
        \draw[] (8.5,-2.5) -- (8.5,-3.5);
        \draw[] (10,-2.5) -- (10,-3.5);
        
    \node at (7.75,-3) [] {\footnotesize $29$}; \node at (7.75,-3) [xshift=1.5cm] {\footnotesize $64$}; \node at (9.25,-3) [xshift=1.5cm] {\footnotesize $5$};
    \end{scope}

\begin{scope}[yshift=-3.4cm ,yscale=0.85]
    \draw [thin] (0,-0.5) -- (11.5,-0.5);
    \draw [thin] (0,-1.5) -- (11.5,-1.5);
    \draw [thin] (0,-2.5) -- (11.5,-2.5);

    \draw [thin] (0,-0.5) -- (0,-2.5);
    \draw [thin] (2.5,-0.5) -- (2.5,-2.5);
    \draw [thin] (3.5,-0.5) -- (3.5,-2.5);
    \draw [thin] (4.5,-0.5) -- (4.5,-2.5);
    \draw [thin] (7,-0.5) -- (7,-2.5);
    \draw [thin] (8.5,-0.5) -- (8.5,-2.5);
    \draw [thin] (10,-0.5) -- (10,-2.5);
    \draw [thin] (11.5,-0.5) -- (11.5,-2.5);
    
    \node at (1.35,-1) [] {\parbox{1.95cm}{\scriptsize $\bh^{(0)}=(1,3,5,4)$}};
    \node at (1.35,-2) [] {\parbox{1.95cm}{\scriptsize $\bh^{(1)}=(1,3,6,10,6)$}};
    
    \node at (3,-1) [] {\scriptsize $13$};
    \node at (3,-2) [] {\scriptsize $26$};
    
    \node at (4,-1) [] {\scriptsize $2$};
    \node at (4,-2) [] {\scriptsize $3$};

    \node at (5.75,-1)[xshift=-4pt] {\scriptsize $(1,6)$};
    \node at (5.75,-2)[xshift=4pt] {\scriptsize $(0,9)$};

    \node at (7.75,-1) [] {\scriptsize $18$}; \node at (7.75,-1) [xshift=1.5cm] {\scriptsize $17$}; \node at (9.25,-1) [xshift=1.5cm] {\scriptsize $4$};
    \node at (7.75,-2) [] {\scriptsize $54$}; \node at (7.75,-2) [xshift=1.5cm] {\scriptsize $54$}; \node at (9.25,-2) [xshift=1.5cm] {\scriptsize $0$};
\end{scope}
\begin{scope}[yshift=-3.1cm-0.75pt ]
    \node at (4.75,-3.05) [] {\footnotesize $\underline{\bh}=(\bh^{(0)},\bh^{(1)})$};
        \draw[thick,] (2.5,-2.5) rectangle (11.5,-3.5);
        \draw[] (7,-2.5) -- (7,-3.5);
        \draw[] (8.5,-2.5) -- (8.5,-3.5);
        \draw[] (10,-2.5) -- (10,-3.5);
        
    \node at (7.75,-3) [] {\footnotesize $36$}; \node at (7.75,-3) [xshift=1.5cm] {\footnotesize $71$}; \node at (9.25,-3) [xshift=1.5cm] {\footnotesize $4$};
    \end{scope}
    
\end{tikzpicture}
}

\medskip

\subfigure[Examples of 2-step Hilbert functions certifying the reducibility of $\hilbert{(\vert\bh^{(0)}\vert,\vert\bh^{(1)}\vert,\vert\bh^{(2)}\vert)}{\mathbb{A}^3}$]{\label{subfig: nested 3 r 3}
\begin{tikzpicture}[yscale=0.4]
    \draw [] (0,1) -- (11.5,1);
    \draw [] (0,-0.5) -- (11.5,-0.5);
    
    \draw [] (0,1) -- (0,-0.5);
    \draw [] (2.5,1) -- (2.5,-0.5);
    \draw [] (3.5,1) -- (3.5,-0.5);
    \draw [] (4.5,1) -- (4.5,-0.5);
    \draw [] (7,1) -- (7,-0.5);
    \draw [] (8.5,1) -- (8.5,-0.5);
    \draw [] (10,1) -- (10,-0.5);
    \draw [] (11.5,1) -- (11.5,-0.5);
    
    \node at (1.25,0.25) [] {\footnotesize $\bh^{(i)}$};
    \node at (3,0.25) [] {\footnotesize $\left\vert \bh^{(i)}\right\vert$};
    \node at (4,0.25) [] {\footnotesize $k+i$};
    \node at (5.75,0.25) [] {\footnotesize $\big(\oh_{k+i}^{(i)},\oh_{k+i+1}^{(i)}\big)$};
    \node at (7.75,0.25) [] {\footnotesize $\dim \mathsf{T}^{<0}$};
    \node at (7.75,0.25) [xshift=1.5cm] {\footnotesize $\dim \mathsf{T}^{=0}$};
    \node at (9.25,0.25) [xshift=1.5cm] {\footnotesize $\dim \mathsf{T}^{=1}$};

    \begin{scope}[yshift=-10pt,yscale=0.85]
    \draw [thin] (0,-0.5) -- (11.5,-0.5);
    \draw [thin] (0,-1.5) -- (11.5,-1.5);
    \draw [thin] (0,-2.5) -- (11.5,-2.5);
    \draw [thin] (0,-3.5) -- (11.5,-3.5);

    \draw [thin] (0,-0.5) -- (0,-3.5);
    \draw [thin] (2.5,-0.5) -- (2.5,-3.5);
    \draw [thin] (3.5,-0.5) -- (3.5,-3.5);
    \draw [thin] (4.5,-0.5) -- (4.5,-3.5);
    \draw [thin] (7,-0.5) -- (7,-3.5);
    \draw [thin] (8.5,-0.5) -- (8.5,-3.5);
    \draw [thin] (10,-0.5) -- (10,-3.5);
    \draw [thin] (11.5,-0.5) -- (11.5,-3.5);
    
    \node at (1.25,-1) [] {\parbox{1.75cm}{\scriptsize $\bh^{(0)}=(1,3,3)$}};
    \node at (1.25,-2) [] {\parbox{1.75cm}{\scriptsize $\bh^{(1)}=(1,3,6,3)$}};
    \node at (1.25,-3) [] {\parbox{1.75cm}{\scriptsize $\bh^{(2)}=(1,3,6,6,1)$}};
    
    \node at (3,-1) [] {\scriptsize $7$};
    \node at (3,-2) [] {\scriptsize $13$};
    \node at (3,-3) [] {\scriptsize $17$};
    
    \node at (4,-1) [] {\scriptsize $1$};
    \node at (4,-2) [] {\scriptsize $2$};
    \node at (4,-3) [] {\scriptsize $3$};

    \node at (5.75,-1)[xshift=-8pt] {\scriptsize $(0,3)$};
    \node at (5.75,-2)[xshift=0pt] {\scriptsize $(0,7)$};
    \node at (5.75,-3)[xshift=8pt] {\scriptsize $(4,14)$};

    \node at (7.75,-1) [] {\scriptsize $12$}; \node at (7.75,-1) [xshift=1.5cm] {\scriptsize $9$}; \node at (9.25,-1) [xshift=1.5cm] {\scriptsize $0$};
    \node at (7.75,-2) [] {\scriptsize $24$}; \node at (7.75,-2) [xshift=1.5cm] {\scriptsize $21$}; \node at (9.25,-2) [xshift=1.5cm] {\scriptsize $0$};
    \node at (7.75,-3) [] {\scriptsize $27$}; \node at (7.75,-3) [xshift=1.5cm] {\scriptsize $26$}; \node at (9.25,-3) [xshift=1.5cm] {\scriptsize $4$};
\end{scope}
\begin{scope}[yshift=2pt]
    \node at (4.75,-4.05) [] {\footnotesize $\underline{\bh}=(\bh^{(0)},\bh^{(1)},\bh^{(2)})$};
        \draw[thick,] (2.5,-3.5) rectangle (11.5,-4.5);
        \draw[] (7,-3.5) -- (7,-4.5);
        \draw[] (8.5,-3.5) -- (8.5,-4.5);
        \draw[] (10,-3.5) -- (10,-4.5);
        
    \node at (7.75,-4) [] {\footnotesize $12$}; \node at (7.75,-4) [xshift=1.5cm] {\footnotesize $44$}; \node at (9.25,-4) [xshift=1.5cm] {\footnotesize $4$};
    \end{scope}

    \begin{scope}[yshift=-4.15cm-4pt,yscale=0.85]
    \draw [thin] (0,-0.5) -- (11.5,-0.5);
    \draw [thin] (0,-1.5) -- (11.5,-1.5);
    \draw [thin] (0,-2.5) -- (11.5,-2.5);
    \draw [thin] (0,-3.5) -- (11.5,-3.5);

    \draw [thin] (0,-0.5) -- (0,-3.5);
    \draw [thin] (2.5,-0.5) -- (2.5,-3.5);
    \draw [thin] (3.5,-0.5) -- (3.5,-3.5);
    \draw [thin] (4.5,-0.5) -- (4.5,-3.5);
    \draw [thin] (7,-0.5) -- (7,-3.5);
    \draw [thin] (8.5,-0.5) -- (8.5,-3.5);
    \draw [thin] (10,-0.5) -- (10,-3.5);
    \draw [thin] (11.5,-0.5) -- (11.5,-3.5);
    
    \node at (1.25,-1) [] {\parbox{1.75cm}{\scriptsize $\bh^{(0)}=(1,3,3)$}};
    \node at (1.25,-2) [] {\parbox{1.75cm}{\scriptsize $\bh^{(1)}=(1,3,6,2)$}};
    \node at (1.25,-3) [] {\parbox{1.75cm}{\scriptsize $\bh^{(2)}=(1,3,6,6,2)$}};
    
    \node at (3,-1) [] {\scriptsize $7$};
    \node at (3,-2) [] {\scriptsize $12$};
    \node at (3,-3) [] {\scriptsize $18$};
    
    \node at (4,-1) [] {\scriptsize $1$};
    \node at (4,-2) [] {\scriptsize $2$};
    \node at (4,-3) [] {\scriptsize $3$};

    \node at (5.75,-1)[xshift=-8pt] {\scriptsize $(0,3)$};
    \node at (5.75,-2)[xshift=0pt] {\scriptsize $(0,8)$};
    \node at (5.75,-3)[xshift=8pt] {\scriptsize $(4,13)$};

    \node at (7.75,-1) [] {\scriptsize $12$}; \node at (7.75,-1) [xshift=1.5cm] {\scriptsize $9$}; \node at (9.25,-1) [xshift=1.5cm] {\scriptsize $0$};
    \node at (7.75,-2) [] {\scriptsize $30$}; \node at (7.75,-2) [xshift=1.5cm] {\scriptsize $16$}; \node at (9.25,-2) [xshift=1.5cm] {\scriptsize $0$};
    \node at (7.75,-3) [] {\scriptsize $20$}; \node at (7.75,-3) [xshift=1.5cm] {\scriptsize $26$}; \node at (9.25,-3) [xshift=1.5cm] {\scriptsize $8$};
    \end{scope}

\begin{scope}[yshift=-3.725cm-4pt]
    \node at (4.75,-4.05) [] {\footnotesize $\underline{\bh}=(\bh^{(0)},\bh^{(1)},\bh^{(2)})$};
        \draw[thick,] (2.5,-3.5) rectangle (11.5,-4.5);
        \draw[] (7,-3.5) -- (7,-4.5);
        \draw[] (8.5,-3.5) -- (8.5,-4.5);
        \draw[] (10,-3.5) -- (10,-4.5);
        
    \node at (7.75,-4) [] {\footnotesize $11$}; \node at (7.75,-4) [xshift=1.5cm] {\footnotesize $43$}; \node at (9.25,-4) [xshift=1.5cm] {\footnotesize $8$};
    \end{scope}

    \begin{scope}[yshift=-4.15cm - 3.875cm-6pt,yscale=0.85]
      \draw [thin] (0,-0.5) -- (11.5,-0.5);
    \draw [thin] (0,-1.5) -- (11.5,-1.5);
    \draw [thin] (0,-2.5) -- (11.5,-2.5);
    \draw [thin] (0,-3.5) -- (11.5,-3.5);

    \draw [thin] (0,-0.5) -- (0,-3.5);
    \draw [thin] (2.5,-0.5) -- (2.5,-3.5);
    \draw [thin] (3.5,-0.5) -- (3.5,-3.5);
    \draw [thin] (4.5,-0.5) -- (4.5,-3.5);
    \draw [thin] (7,-0.5) -- (7,-3.5);
    \draw [thin] (8.5,-0.5) -- (8.5,-3.5);
    \draw [thin] (10,-0.5) -- (10,-3.5);
    \draw [thin] (11.5,-0.5) -- (11.5,-3.5);
    
    \node at (1.25,-1) [] {\parbox{1.75cm}{\scriptsize $\bh^{(0)}=(1,3,2)$}};
    \node at (1.25,-2) [] {\parbox{1.75cm}{\scriptsize $\bh^{(1)}=(1,3,6,2)$}};
    \node at (1.25,-3) [] {\parbox{1.75cm}{\scriptsize $\bh^{(2)}=(1,3,6,7,3)$}};
    
    \node at (3,-1) [] {\scriptsize $6$};
    \node at (3,-2) [] {\scriptsize $12$};
    \node at (3,-3) [] {\scriptsize $20$};
    
    \node at (4,-1) [] {\scriptsize $1$};
    \node at (4,-2) [] {\scriptsize $2$};
    \node at (4,-3) [] {\scriptsize $3$};

    \node at (5.75,-1)[xshift=-8pt] {\scriptsize $(0,4)$};
    \node at (5.75,-2)[xshift=0pt] {\scriptsize $(0,8)$};
    \node at (5.75,-3)[xshift=8pt] {\scriptsize $(3,12)$};

    \node at (7.75,-1) [] {\scriptsize $10$}; \node at (7.75,-1) [xshift=1.5cm] {\scriptsize $8$}; \node at (9.25,-1) [xshift=1.5cm] {\scriptsize $0$};
    \node at (7.75,-2) [] {\scriptsize $30$}; \node at (7.75,-2) [xshift=1.5cm] {\scriptsize $16$}; \node at (9.25,-2) [xshift=1.5cm] {\scriptsize $0$};
    \node at (7.75,-3) [] {\scriptsize $21$}; \node at (7.75,-3) [xshift=1.5cm] {\scriptsize $30$}; \node at (9.25,-3) [xshift=1.5cm] {\scriptsize $9$};
    \end{scope}

\begin{scope}[yshift=-3.725cm - 3.875cm-6pt]
    \node at (4.75,-4.05) [] {\footnotesize $\underline{\bh}=(\bh^{(0)},\bh^{(1)},\bh^{(2)})$};
        \draw[thick,] (2.5,-3.5) rectangle (11.5,-4.5);
        \draw[] (7,-3.5) -- (7,-4.5);
        \draw[] (8.5,-3.5) -- (8.5,-4.5);
        \draw[] (10,-3.5) -- (10,-4.5);
        
    \node at (7.75,-4) [] {\footnotesize $20$}; \node at (7.75,-4) [xshift=1.5cm] {\footnotesize $48$}; \node at (9.25,-4) [xshift=1.5cm] {\footnotesize $9$};
    \end{scope}

    \begin{scope}[yshift=-4.15cm - 2*3.875cm-8pt,yscale=0.85]
      \draw [thin] (0,-0.5) -- (11.5,-0.5);
    \draw [thin] (0,-1.5) -- (11.5,-1.5);
    \draw [thin] (0,-2.5) -- (11.5,-2.5);
    \draw [thin] (0,-3.5) -- (11.5,-3.5);

    \draw [thin] (0,-0.5) -- (0,-3.5);
    \draw [thin] (2.5,-0.5) -- (2.5,-3.5);
    \draw [thin] (3.5,-0.5) -- (3.5,-3.5);
    \draw [thin] (4.5,-0.5) -- (4.5,-3.5);
    \draw [thin] (7,-0.5) -- (7,-3.5);
    \draw [thin] (8.5,-0.5) -- (8.5,-3.5);
    \draw [thin] (10,-0.5) -- (10,-3.5);
    \draw [thin] (11.5,-0.5) -- (11.5,-3.5);
    
    \node at (1.25,-1) [] {\parbox{1.75cm}{\scriptsize $\bh^{(0)}=(1,2,1)$}};
    \node at (1.25,-2) [] {\parbox{1.75cm}{\scriptsize $\bh^{(1)}=(1,3,6,3)$}};
    \node at (1.25,-3) [] {\parbox{1.75cm}{\scriptsize $\bh^{(2)}=(1,3,6,8,3)$}};
    
    \node at (3,-1) [] {\scriptsize $4$};
    \node at (3,-2) [] {\scriptsize $13$};
    \node at (3,-3) [] {\scriptsize $21$};
    
    \node at (4,-1) [] {\scriptsize $1$};
    \node at (4,-2) [] {\scriptsize $2$};
    \node at (4,-3) [] {\scriptsize $3$};

    \node at (5.75,-1)[xshift=-8pt] {\scriptsize $(1,5)$};
    \node at (5.75,-2)[xshift=0pt] {\scriptsize $(0,7)$};
    \node at (5.75,-3)[xshift=8pt] {\scriptsize $(2,12)$};

    \node at (7.75,-1) [] {\scriptsize $7$}; \node at (7.75,-1) [xshift=1.5cm] {\scriptsize $4$}; \node at (9.25,-1) [xshift=1.5cm] {\scriptsize $1$};
    \node at (7.75,-2) [] {\scriptsize $24$}; \node at (7.75,-2) [xshift=1.5cm] {\scriptsize $21$}; \node at (9.25,-2) [xshift=1.5cm] {\scriptsize $0$};
    \node at (7.75,-3) [] {\scriptsize $33$}; \node at (7.75,-3) [xshift=1.5cm] {\scriptsize $34$}; \node at (9.25,-3) [xshift=1.5cm] {\scriptsize $6$};
    \end{scope}

\begin{scope}[yshift=-3.725cm - 2*3.875cm-8pt]
    \node at (4.75,-4.05) [] {\footnotesize $\underline{\bh}=(\bh^{(0)},\bh^{(1)},\bh^{(2)})$};
        \draw[thick,] (2.5,-3.5) rectangle (11.5,-4.5);
        \draw[] (7,-3.5) -- (7,-4.5);
        \draw[] (8.5,-3.5) -- (8.5,-4.5);
        \draw[] (10,-3.5) -- (10,-4.5);
        
    \node at (7.75,-4) [] {\footnotesize $25$}; \node at (7.75,-4) [xshift=1.5cm] {\footnotesize $53$}; \node at (9.25,-4) [xshift=1.5cm] {\footnotesize $7$};
    \end{scope}

    \begin{scope}[yshift=-4.15cm - 3*3.875cm-10pt,yscale=0.85]
     \draw [thin] (0,-0.5) -- (11.5,-0.5);
    \draw [thin] (0,-1.5) -- (11.5,-1.5);
    \draw [thin] (0,-2.5) -- (11.5,-2.5);
    \draw [thin] (0,-3.5) -- (11.5,-3.5);

    \draw [thin] (0,-0.5) -- (0,-3.5);
    \draw [thin] (2.5,-0.5) -- (2.5,-3.5);
    \draw [thin] (3.5,-0.5) -- (3.5,-3.5);
    \draw [thin] (4.5,-0.5) -- (4.5,-3.5);
    \draw [thin] (7,-0.5) -- (7,-3.5);
    \draw [thin] (8.5,-0.5) -- (8.5,-3.5);
    \draw [thin] (10,-0.5) -- (10,-3.5);
    \draw [thin] (11.5,-0.5) -- (11.5,-3.5);
    
    \node at (1.25,-1) [] {\parbox{1.75cm}{\scriptsize $\bh^{(0)}=(1,3,2)$}};
    \node at (1.25,-2) [] {\parbox{1.75cm}{\scriptsize $\bh^{(1)}=(1,3,5,2)$}};
    \node at (1.25,-3) [] {\parbox{1.75cm}{\scriptsize $\bh^{(2)}=(1,3,6,8,4)$}};
    
    \node at (3,-1) [] {\scriptsize $6$};
    \node at (3,-2) [] {\scriptsize $11$};
    \node at (3,-3) [] {\scriptsize $22$};
    
    \node at (4,-1) [] {\scriptsize $1$};
    \node at (4,-2) [] {\scriptsize $2$};
    \node at (4,-3) [] {\scriptsize $3$};

    \node at (5.75,-1)[xshift=-8pt] {\scriptsize $(0,4)$};
    \node at (5.75,-2)[xshift=0pt] {\scriptsize $(1,8)$};
    \node at (5.75,-3)[xshift=8pt] {\scriptsize $(2,11)$};

    \node at (7.75,-1) [] {\scriptsize $10$}; \node at (7.75,-1) [xshift=1.5cm] {\scriptsize $8$}; \node at (9.25,-1) [xshift=1.5cm] {\scriptsize $0$};
    \node at (7.75,-2) [] {\scriptsize $16$}; \node at (7.75,-2) [xshift=1.5cm] {\scriptsize $15$}; \node at (9.25,-2) [xshift=1.5cm] {\scriptsize $2$};
    \node at (7.75,-3) [] {\scriptsize $28$}; \node at (7.75,-3) [xshift=1.5cm] {\scriptsize $36$}; \node at (9.25,-3) [xshift=1.5cm] {\scriptsize $8$};
    \end{scope}

\begin{scope}[yshift=-3.725cm - 3*3.875cm-10pt]
    \node at (4.75,-4.05) [] {\footnotesize $\underline{\bh}=(\bh^{(0)},\bh^{(1)},\bh^{(2)})$};
        \draw[thick,] (2.5,-3.5) rectangle (11.5,-4.5);
        \draw[] (7,-3.5) -- (7,-4.5);
        \draw[] (8.5,-3.5) -- (8.5,-4.5);
        \draw[] (10,-3.5) -- (10,-4.5);
        
    \node at (7.75,-4) [] {\footnotesize $20$}; \node at (7.75,-4) [xshift=1.5cm] {\footnotesize $55$}; \node at (9.25,-4) [xshift=1.5cm] {\footnotesize $10$};
    \end{scope}

    \begin{scope}[yshift=-4.15cm - 4*3.875cm-12pt,yscale=0.85]
     \draw [thin] (0,-0.5) -- (11.5,-0.5);
    \draw [thin] (0,-1.5) -- (11.5,-1.5);
    \draw [thin] (0,-2.5) -- (11.5,-2.5);
    \draw [thin] (0,-3.5) -- (11.5,-3.5);

    \draw [thin] (0,-0.5) -- (0,-3.5);
    \draw [thin] (2.5,-0.5) -- (2.5,-3.5);
    \draw [thin] (3.5,-0.5) -- (3.5,-3.5);
    \draw [thin] (4.5,-0.5) -- (4.5,-3.5);
    \draw [thin] (7,-0.5) -- (7,-3.5);
    \draw [thin] (8.5,-0.5) -- (8.5,-3.5);
    \draw [thin] (10,-0.5) -- (10,-3.5);
    \draw [thin] (11.5,-0.5) -- (11.5,-3.5);
    
    \node at (1.25,-1) [] {\parbox{1.75cm}{\scriptsize $\bh^{(0)}=(1,2,1)$}};
    \node at (1.25,-2) [] {\parbox{1.75cm}{\scriptsize $\bh^{(1)}=(1,3,5,3)$}};
    \node at (1.25,-3) [] {\parbox{1.75cm}{\scriptsize $\bh^{(2)}=(1,3,6,9,4)$}};
    
    \node at (3,-1) [] {\scriptsize $4$};
    \node at (3,-2) [] {\scriptsize $12$};
    \node at (3,-3) [] {\scriptsize $23$};
    
    \node at (4,-1) [] {\scriptsize $1$};
    \node at (4,-2) [] {\scriptsize $2$};
    \node at (4,-3) [] {\scriptsize $3$};

    \node at (5.75,-1)[xshift=-8pt] {\scriptsize $(1,5)$};
    \node at (5.75,-2)[xshift=0pt] {\scriptsize $(1,7)$};
    \node at (5.75,-3)[xshift=8pt] {\scriptsize $(1,11)$};

    \node at (7.75,-1) [] {\scriptsize $7$}; \node at (7.75,-1) [xshift=1.5cm] {\scriptsize $4$}; \node at (9.25,-1) [xshift=1.5cm] {\scriptsize $1$};
    \node at (7.75,-2) [] {\scriptsize $16$}; \node at (7.75,-2) [xshift=1.5cm] {\scriptsize $17$}; \node at (9.25,-2) [xshift=1.5cm] {\scriptsize $3$};
    \node at (7.75,-3) [] {\scriptsize $42$}; \node at (7.75,-3) [xshift=1.5cm] {\scriptsize $41$}; \node at (9.25,-3) [xshift=1.5cm] {\scriptsize $4$};
    \end{scope}

\begin{scope}[yshift=-3.725cm - 4*3.875cm-12pt]
    \node at (4.75,-4.05) [] {\footnotesize $\underline{\bh}=(\bh^{(0)},\bh^{(1)},\bh^{(2)})$};
        \draw[thick,] (2.5,-3.5) rectangle (11.5,-4.5);
        \draw[] (7,-3.5) -- (7,-4.5);
        \draw[] (8.5,-3.5) -- (8.5,-4.5);
        \draw[] (10,-3.5) -- (10,-4.5);
        
    \node at (7.75,-4) [] {\footnotesize $27$}; \node at (7.75,-4) [xshift=1.5cm] {\footnotesize $58$}; \node at (9.25,-4) [xshift=1.5cm] {\footnotesize $8$};
    \end{scope}

    \begin{scope}[yshift=-4.15cm - 5*3.875cm-14pt,yscale=0.85]
     \draw [thin] (0,-0.5) -- (11.5,-0.5);
    \draw [thin] (0,-1.5) -- (11.5,-1.5);
    \draw [thin] (0,-2.5) -- (11.5,-2.5);
    \draw [thin] (0,-3.5) -- (11.5,-3.5);

    \draw [thin] (0,-0.5) -- (0,-3.5);
    \draw [thin] (2.5,-0.5) -- (2.5,-3.5);
    \draw [thin] (3.5,-0.5) -- (3.5,-3.5);
    \draw [thin] (4.5,-0.5) -- (4.5,-3.5);
    \draw [thin] (7,-0.5) -- (7,-3.5);
    \draw [thin] (8.5,-0.5) -- (8.5,-3.5);
    \draw [thin] (10,-0.5) -- (10,-3.5);
    \draw [thin] (11.5,-0.5) -- (11.5,-3.5);
    
    \node at (1.25,-1) [] {\parbox{1.75cm}{\scriptsize $\bh^{(0)}=(1,1,1)$}};
    \node at (1.25,-2) [] {\parbox{1.75cm}{\scriptsize $\bh^{(1)}=(1,3,5,3)$}};
    \node at (1.25,-3) [] {\parbox{1.75cm}{\scriptsize $\bh^{(2)}=(1,3,6,9,5)$}};
    
    \node at (3,-1) [] {\scriptsize $3$};
    \node at (3,-2) [] {\scriptsize $12$};
    \node at (3,-3) [] {\scriptsize $24$};
    
    \node at (4,-1) [] {\scriptsize $1$};
    \node at (4,-2) [] {\scriptsize $2$};
    \node at (4,-3) [] {\scriptsize $3$};

    \node at (5.75,-1)[xshift=-8pt] {\scriptsize $(2,5)$};
    \node at (5.75,-2)[xshift=0pt] {\scriptsize $(1,7)$};
    \node at (5.75,-3)[xshift=8pt] {\scriptsize $(1,10)$};

    \node at (7.75,-1) [] {\scriptsize $5$}; \node at (7.75,-1) [xshift=1.5cm] {\scriptsize $2$}; \node at (9.25,-1) [xshift=1.5cm] {\scriptsize $2$};
    \node at (7.75,-2) [] {\scriptsize $16$}; \node at (7.75,-2) [xshift=1.5cm] {\scriptsize $17$}; \node at (9.25,-2) [xshift=1.5cm] {\scriptsize $3$};
    \node at (7.75,-3) [] {\scriptsize $39$}; \node at (7.75,-3) [xshift=1.5cm] {\scriptsize $44$}; \node at (9.25,-3) [xshift=1.5cm] {\scriptsize $5$};
    \end{scope}

\begin{scope}[yshift=-3.725cm - 5*3.875cm-14pt]
    \node at (4.75,-4.05) [] {\footnotesize $\underline{\bh}=(\bh^{(0)},\bh^{(1)},\bh^{(2)})$};
        \draw[thick,] (2.5,-3.5) rectangle (11.5,-4.5);
        \draw[] (7,-3.5) -- (7,-4.5);
        \draw[] (8.5,-3.5) -- (8.5,-4.5);
        \draw[] (10,-3.5) -- (10,-4.5);
        
    \node at (7.75,-4) [] {\footnotesize $26$}; \node at (7.75,-4) [xshift=1.5cm] {\footnotesize $59$}; \node at (9.25,-4) [xshift=1.5cm] {\footnotesize $10$};
    \end{scope}

     \begin{scope}[yshift=-4.15cm - 6*3.875cm-16pt,yscale=0.85]
    \draw [thin] (0,-0.5) -- (11.5,-0.5);
    \draw [thin] (0,-1.5) -- (11.5,-1.5);
    \draw [thin] (0,-2.5) -- (11.5,-2.5);
    \draw [thin] (0,-3.5) -- (11.5,-3.5);

    \draw [thin] (0,-0.5) -- (0,-3.5);
    \draw [thin] (2.5,-0.5) -- (2.5,-3.5);
    \draw [thin] (3.5,-0.5) -- (3.5,-3.5);
    \draw [thin] (4.5,-0.5) -- (4.5,-3.5);
    \draw [thin] (7,-0.5) -- (7,-3.5);
    \draw [thin] (8.5,-0.5) -- (8.5,-3.5);
    \draw [thin] (10,-0.5) -- (10,-3.5);
    \draw [thin] (11.5,-0.5) -- (11.5,-3.5);
    
    \node at (1.35,-1) [] {\parbox{1.95cm}{\scriptsize $\bh^{(0)}=(1,1)$}};
    \node at (1.35,-2) [] {\parbox{1.95cm}{\scriptsize $\bh^{(1)}=(1,3,5,3)$}};
    \node at (1.35,-3) [] {\parbox{1.95cm}{\scriptsize $\bh^{(2)}=(1,3,6,10,5)$}};
    
    \node at (3,-1) [] {\scriptsize $2$};
    \node at (3,-2) [] {\scriptsize $12$};
    \node at (3,-3) [] {\scriptsize $25$};
    
    \node at (4,-1) [] {\scriptsize $1$};
    \node at (4,-2) [] {\scriptsize $2$};
    \node at (4,-3) [] {\scriptsize $3$};

    \node at (5.75,-1)[xshift=-8pt] {\scriptsize $(2,6)$};
    \node at (5.75,-2)[xshift=0pt] {\scriptsize $(1,7)$};
    \node at (5.75,-3)[xshift=8pt] {\scriptsize $(0,10)$};

    \node at (7.75,-1) [] {\scriptsize $4$}; \node at (7.75,-1) [xshift=1.5cm] {\scriptsize $2$}; \node at (9.25,-1) [xshift=1.5cm] {\scriptsize $0$};
    \node at (7.75,-2) [] {\scriptsize $16$}; \node at (7.75,-2) [xshift=1.5cm] {\scriptsize $17$}; \node at (9.25,-2) [xshift=1.5cm] {\scriptsize $3$};
    \node at (7.75,-3) [] {\scriptsize $55$}; \node at (7.75,-3) [xshift=1.5cm] {\scriptsize $50$}; \node at (9.25,-3) [xshift=1.5cm] {\scriptsize $0$};
    \end{scope}

\begin{scope}[yshift=-3.725cm - 6*3.875cm-16pt]
    \node at (4.75,-4.05) [] {\footnotesize $\underline{\bh}=(\bh^{(0)},\bh^{(1)},\bh^{(2)})$};
        \draw[thick,] (2.5,-3.5) rectangle (11.5,-4.5);
        \draw[] (7,-3.5) -- (7,-4.5);
        \draw[] (8.5,-3.5) -- (8.5,-4.5);
        \draw[] (10,-3.5) -- (10,-4.5);
        
    \node at (7.75,-4) [] {\footnotesize $42$}; \node at (7.75,-4) [xshift=1.5cm] {\footnotesize $69$}; \node at (9.25,-4) [xshift=1.5cm] {\footnotesize $3$};
    \end{scope}
\end{tikzpicture}
}
    \caption{Hilbert functions certifying the reducibility of nested Hilbert schemes on three-folds.}
       \label{fig:nested threefold}
\end{figure}

\begin{remark}
We remark that the ideals of colength 18 and 24 in \Cref{fig:nested threefold} already appeared in the literature in other contexts. The generic compressed algebra with Hilbert function $(1,6,6,2)$ does not lie in the curvilinear component of the punctual Hilbert scheme, i.e.~the closure of the locus parametrising curvilinear ideals, see \cite{JJ-Hilb-open-problems,GGGL}. On the other hand, the counterexample to the constancy of the Behrend function given in \cite{BEHNONCONST} has length 24 and Hilbert function $(1,3,6,9,5)$, see also \cite{UPDATES} for a proof of its smoothability. 
\end{remark}

\section{Reducibility of Hilbert schemes of points on four-folds}\label{subsec:smallcomp}

As an application of the theory of 2-step ideals, we provide a complete list of the generically reduced elementary components corresponding to 2-step ideals with no or very few linear syzygies of order $k=2,3$ plus some other sporadic example. We stress that this theory produces elementary components for each order $k\geqslant 2 $. Then, in the final part we prove \Cref{thm:intro B}.

\paragraph{\it Known results.} Let $n\geqslant 4$ be a positive integer. The values of $d\geqslant0$ for which the Hilbert scheme $\Hilb^d\BA^n$ is irreducible have been classified in \cite{MAZZOLA,Iarrob}, see also \cite{8POINTS}. Then, many examples of elementary components were found, see \cite{ELEMENTARY,SomeElementary,MoreElementary,Sha90,UPDATES}. In contrast to the three-dimensional setting, where the available techniques only allow the detection of irreducible components of dimension greater than the smoothable one, the situation is much different for $n\geqslant 4$. For instance, \cite{Satrianostaal,GALOIS} presents many small elementary components, i.e.~components of dimension smaller than that of the curvilinear locus. It is also worth noting that \cite{ERMANVELASCO} provides a characterisation of elementary components that parametrise non-smoothable algebras of a given embedding dimension and minimum possible length.

\smallskip

Among the elementary components described in the literature, few of them correspond to 2-step ideals.
\begin{itemize}
\item Five elementary components correspond in fact to families of 1-step ideals, one generically reduced component in $\hilbert{8}{\mathbb{A}^4}$ \cite{Iarrob} and four generically non-reduced components in $\hilbert{d}{\mathbb{A}^4}$ for $d=21,22,23,24$ \cite{jelisiejewnonred}.
\item Two generically reduced elementary components are Hilbert strata of 2-step ideals, one corresponding to 2-step ideals of order 2 \cite{Satrianostaal} and one corresponding to 2-step ideals of order 3 \cite{ELEMENTARY}.
\end{itemize}
These three generically reduced elementary components are small components, in the sense explained above.

\bigskip

For $n=4$ the absolute minimum of the function $\Theta_{4,k,b}(\oh_k,\oh_{k+1})$
\[
\dfrac{-k^{6}-15\,k^{5}-79\,k^{4}+\left(-24\,b-165\right)k^{3}+\left(-180\,b-64\right)k^{2}+\left(-408\,b+180\right)k-144\,b^{2}-2160}{540}
\]
is negative for every $k \geqslant 1$ and every $b \geqslant 0$. The potential TNT area contains at least the interior part of the ellipse corresponding to the condition $\Theta_{4,0,k} \leqslant 0$ (see \Cref{fig:newA4k2} and \Cref{fig:newA4k3}).

The Hessian matrix of $\Delta_{4,1,k}$ 
\[
\Hess \Delta_{4,1,k} = \left[\begin{array}{cc} -2 & 3 \\ 3 & -2\end{array}\right]
\]
is non-singular with a positive and a negative eigenvalue. Hence, $\Delta_{4,1,k}$ has a saddle point and admits positive and negative values. In this situation, we look for 
\begin{itemize}
    \item 2-step Hilbert functions in the potential TNT area that are realized by a homogeneous ideals with the TNT property;
    \item 2-step Hilbert functions with $\Delta_{4,1,k} \geqslant 0$.
\end{itemize}
In both cases, we certify that the Hilbert scheme $\hilbert{\vert\bh\vert}{\mathbb{A}^4}$ is reducible, but in the first case, we detect generically reduced elementary components. 

We systematically examine all 2-step ideals with no or few linear syzygies of order $k=2$ and $k=3$. The results are summarized in \Cref{fig:newA4k2} and \Cref{fig:newA4k3}, see \Cref{legenda} for the picture legend.

Among the families of 2-step ideals of order 2 with no linear syzygies, we find two new generically reduced elementary components (see \Cref{fig:newA4k2}) in the Hilbert schemes $\hilbert{18}{\mathbb{A}^4}$ and $\hilbert{20}{\mathbb{A}^4}$. The dimension of these components is smaller than the dimension of the smoothable component, but the components are not small because the dimension of the Hilbert stratum is greater than the dimension of the curvilinear locus. In the following, we refer to these elementary components as \textit{$\Delta$-negative components}. Among 2-step ideals of order 3, we find two $\Delta$-negative generically reduced elementary components and 27 generically reduced elementary components whose dimension is greater than the dimension of the smoothable component, see \Cref{fig:newA4k3}.

With the help of \textit{Macaulay2}, it is possible to determine plenty of other generically reduced elementary components. For instance, among 2-step ideals of order 4, there are 95 Hilbert strata whose generic ideal has trivial negative tangents (see the ancillary \textit{Macaulay2} file \href{www.paololella.it/software/reducibility-Hilbert-schemes.m2}{\tt reducibility-Hilbert- schemes.m2}). We notice that none of them is $\Delta$-negative. In fact, the intersection between the areas $\Delta_{4,1,4} < 0$ and $\Theta_{4,4,0} \leqslant 0$ is empty.
 
\begin{figure}
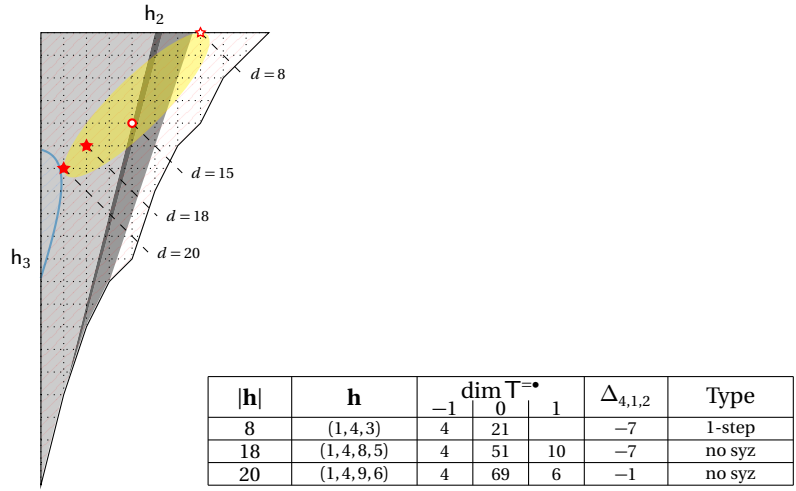

\centering

\caption{Components in $\Hilb^\bullet \mathbb{A}^4$ coming from families of 2-step ideals of order 2. In the table, we describe those covered by our construction.\label{fig:newA4k2}}
\end{figure}

\begin{figure}
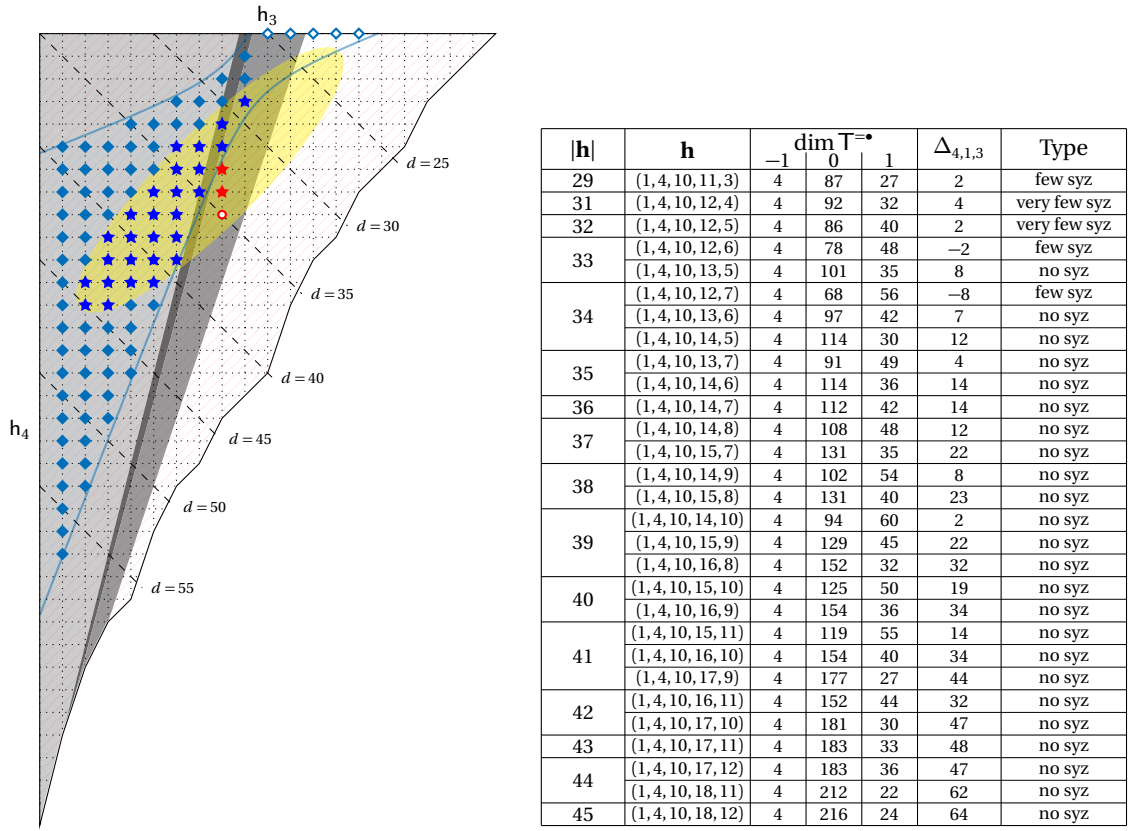

\centering
%

    \caption{Components in $\Hilb^\bullet \mathbb{A}^4$ coming from families of 2-step ideals of order 3. In the table, we describe those covered by our construction.\label{fig:newA4k3}}
     
   
\end{figure}  

\begin{remark} 
The generically reduced elementary component of $\hilbert{35}{\mathbb{A}^4}$ detected by Jelisiejew in \cite{ELEMENTARY} corresponds to the 2-step Hilbert function $\bh=(1,4,10,12,8)$. The pair $(\oh_3,\oh_4) = (8,27)$ lies in the few syzygies area, since $\frac{1}{4}\oh_3 < - \sss_\bh = 5 <  \oh_3$, but the generic homomorphism in $\mathscr{L}_{\bh}$ is injective and \Cref{thm: homo loco syz} does not apply. In fact, the resolution of the generic homogeneous element in the component does not have a natural first anti-diagonal as shown below.
\[ \begin{array}{c|cccc}
       & 0 & 1 & 2 & 3\\
       \hline
      3 & 8 & 6 & \centerdot & \centerdot\\[-4pt]
      4 & 1 & 4 & 2 & \centerdot\\[-4pt]
      5 & \centerdot & 12 & 20 & 8\end{array}
 \]
       
The generically reduced elementary component of $\hilbert{15}{\mathbb{A}^4}$ detected by Satriano and Staal in \cite{Satrianostaal} corresponds to the 2-step Hilbert function $\bh=(1,4,6,4)$. The pair $(\oh_2,\oh_3) = (4,16)$ lies in the no syzygies area, since $ \sss_\bh = 0$. Applying, \Cref{thm:homoloco}, we describe homogeneous 2-step ideals with Betti table
        \[
      \begin{array}{c|cccc}
       & 0 & 1 & 2 & 3\\
       \hline
      2 & 4 & \centerdot & \centerdot & \centerdot\\[-4pt]
      3 & \centerdot & 6 & \centerdot & \centerdot \\[-4pt]
      4 & 1 & 4 & 10 & 4\end{array}
        \]
while the generic homogeneous element in the correspondent elementary component has the following Betti table. 
        \[
      \begin{array}{c|cccc}
       & 0 & 1 & 2 & 3\\
       \hline
      2 & 4 & 4 & 1 & \centerdot \\[-4pt]
      3 & 4 & 6 & 2 & \centerdot \\[-4pt]
      4 & \centerdot  & 6 & 10 & 4\end{array}
        \]
We notice that the homogeneous locus $\mathscr{H}^4_{\bh}$ has at least 2 irreducible components. In fact, by the semicontinuity of the Betti numbers, we cannot obtain a 2-step homogeneous ideal with the first Betti table as (homogeneous) specialization of 2-step homogeneous ideals lying on the elementary component with the second Betti table (and viceversa). Moreover, this implies, together with the existence of the initial ideal morphism $\pi_{\bh}:H^4_{\bh}\to\OH_{\bh}^4$, that the Hilbert stratum $H_{\bh}^4$ is not irreducible.
\end{remark}

\subsection{Nested Hilbert schemes of points on four-folds}\label{subsec:dim4nest}

 In this subsection we prove \Cref{thm:intro B} which shows that, paying the price of considering nestings, the $(3,7)$-nested Hilbert scheme on $\BA^4$ has an elementary component. From this, we also deduce that the $(1,3,7)$-nested Hilbert scheme on $\BA^4$ has a generically non-reduced elementary component.

 \begin{theorem}\label{thm:redim4nest} 
        The nested Hilbert scheme $\hilbert{(3,7)}{\BA^4}$ has a generically reduced elementary  component $V$ whose closed points correspond to nestings having Hilbert function vector $((1,2),(1,4,2))$. Moreover, we have an isomorphism
    \[
    \begin{tikzcd}
        (V)_{\red}\cong   \Gr(2,4) \times \Gr(2,10)\times \BA^4\cong   H_{(1,2)}^4 \times H_{(1,4,2)}^4 \times \BA^4.
    \end{tikzcd}
    \]
     As a consequence, the nested Hilbert scheme $\hilbert{(1,3,7)}{\BA^4}$ has a generically non-reduced elementary component $V_1$ such that $(V_1)_{\red}=(V)_{\red}$.
\end{theorem}

\begin{proof}
The last part of the statement is a consequence of \cite[Theorem 5]{UPDATES}. Moreover, the second isomorphism is a consequence of the well known description of the Hilbert stratum associated to very compressed algebras. We focus on the first isomorphism. We start by the observation that 
\[
H_{(\bh^{(1)},\bh^{(2)})}^4\cong H_{\bh^{(1)}}^4 \times H_{\bh^{(2)}}^4,
\]
where  $\bh^{(1)}=(1,2), \mbox{ and }\bh^{(2)}=(1,4,2)$. This is true because $H_{(\bh^{(1)},\bh^{(2)})}^4$ is the closed subset of the product $ H_{\bh^{(1)}}^4 \times H_{\bh^{(2)}}^4$ cut out by the nested conditions, but in this setting the nesting is guaranteed by construction and no condition arises.

 In order to conclude the proof we exhibit a nesting $\underline{I}$ in $H^4_{(\bh^{(1)},\bh^{(2)})}$ having TNT.  
The nesting we consider is $\underline{I}=(I^{(1)}\supset I^{(2)})$, where
\[
I^{(2)}=\left(zw,xw,z^2-w^2,yz,xz+yw \right) +\left( x,y\right )^2,\qquad 
\text{~and~}\qquad
I^{(1)}=I^{(2)}+(z,w)
\]
in the polynomial ring $\BC[x,y,z,w]$. 
\end{proof}

\section{Further new  elementary components}\label{sec:small} 

We list now all the generically reduced elementary components of $\Hilb^\bullet \BA^n$, for $n=5,6$, corresponding to 2-step ideals with no or very few linear syzygies of order $k=2$ plus some other sporadic example. Then, in the final part we prove \Cref{thm: intro tantecomp}.

\paragraph{\it Known results.}  
Among the elementary components described in the literature for $n \geqslant 5$, most of them correspond to families of 2-step ideals.
\begin{itemize}
\item The elementary components arising from families of 1-step ideals with Hilbert function $\bh=(1,n,s)$ are treated in \cite{Sha90} where it is stated that for
\[
3 \leqslant s \leqslant \dfrac{(n-1)(n-2)}{6} + 2
\]
the Hilbert stratum $H_{\bh}^n$ gives a generically reduced elementary component.
\item Three generically reduced elementary components in $\Hilb^\bullet{\mathbb{A}^5}$ are Hilbert strata of 2-step ideals of order 2 with Hilbert function $\bh=(1,5,3,4)$ (see \cite{SomeElementary}), $\bh=(1,5,9,7)$ (see \cite{MoreElementary}) and $\bh=(1,5,6,1)$ (see \cite{KleimanKleppe2025}).
\item Four generically reduced elementary components in $\Hilb^\bullet{\mathbb{A}^5}$ are Hilbert strata of 2-step ideals of order 3 with Hilbert function $\bh=(1,5,15,7,9)$, $\bh=(1,5,15,9,12)$, $\bh=(1,5,15,10,12)$ and $\bh=(1,5,15,10,15)$  (see \cite{MoreElementary}).
\item Six generically reduced elementary components in $\Hilb^\bullet{\mathbb{A}^6}$ are Hilbert strata of 2-step ideals of order 2 are obtained via apolarity studying the socle type of the associated algebra (see \Cref{def:compressed}).
There is the inspiring example with Hilbert function $\bh=(1,6,6,1)$ by Iarrobino \cite{Iarrob} generalized in \cite{KleimanKleppe2025} to Hilbert functions $\bh=(1,6,12,2)$ and $\bh=(1,6,6+s,1)$ with $s=2,3,4,5$. Notice that for $s=1$ the 2-step ideals of the Hilbert stratum are smoothable.
\item There are three other generically reduced elementary components in $\Hilb^\bullet{\mathbb{A}^6}$ with Hilbert functions $\bh=(1,6,6,10)$ (see \cite{SomeElementary}), $\bh=(1,6,5,7)$ (see \cite{MoreElementary}) and $\bh=(1,6,12,7)$ (see \cite{UPDATES}).
\item There is one generically reduced elementary components in $\Hilb^\bullet{\mathbb{A}^6}$ with 2-step Hilbert function $\bh=(1,6,21,10,15)$ of order 3 \cite{SomeElementary}.
\item In $\Hilb^\bullet{\mathbb{A}^n}$ with $n \geqslant 7$, there are generically reduced elementary components arising from Hilbert strata with Hilbert function $\bh=(1,7,7,1)$ (see \cite{BertoneCioffiRoggero}), $\bh=(1,7,10,16)$ (see \cite{MoreElementary}), $\bh=(1,7,7+s,1),\ s=1,\ldots,8$, $\bh=(1,8,8+s,1),\ s=1,\ldots,10$ and $\bh=(1,n,2n,2),(1,n,2n+1,2)$ with $n=7,8$ (see \cite{KleimanKleppe2025}).
\end{itemize}

    \begin{figure}
    \centering



    \caption{Components in $\Hilb^\bullet \mathbb{A}^5$ coming from families of 2-step ideals of order 2. In the table, we describe those covered by our construction.\label{fig:newA5}}
\end{figure}

\bigskip

For $n\geqslant 5$ the level curves of the function $\Theta_{n,k,b}(\oh_k,\oh_{k+1})$ are hyperbolas and the potential TNT area contains the connected area in $\mathcal{D}$ between the two branches of the hyperbola of equations $\Theta_{n,k,0}(\oh_k,\oh_{k+1}) = 0$ (see Figures \ref{fig:newA5}-\ref{fig:newA6-3}).

The determinant of the Hessian matrix of $\Delta_{n,1,k}$ is negative for $n \geqslant 4$
\[
\det \Hess \Delta_{n,1,k} = \det \left[\begin{array}{cc} -2 & n-1 \\ n-1 & -2\end{array}\right] = 4 - (n-1)^2
\]
so that $\Delta_{n,1,k}$ has a saddle point and it admits positive and negative values. We notice that according to the parity of $n$, the area $\Delta_{n,1,k} \geqslant 0$ can be the connected area between the two branches of a hyperbola ($n$ even) or its complement ($n$ odd).

\medskip

We systematically examine all 2-step ideals with no or few linear syzygies of order $k=2$ for $n=5$ and $n=6$. In particular, we
look for 2-step Hilbert functions in the potential TNT area that are realized by homogeneous ideals with the TNT property that identify a generically reduced elementary component of the Hilbert scheme. The results are summarized in Figures \ref{fig:newA5}-\ref{fig:newA6-3} (see \Cref{legenda} for the picture legend). In particular, we find 
\begin{itemize}
    \item 43 new elementary components in $\Hilb^\bullet{\mathbb{A}^5}$, two of which are $\Delta$-negative;
    \item 140 new elementary components in $\Hilb^\bullet{\mathbb{A}^6}$.
\end{itemize}
Increasing the order of the 2-step ideals allows to find thousands of new generically reduced elementary components. For instance, the potential TNT area of 2-step Hilbert functions of order 3 contains 304 natural points for $n=5$ and 973 natural points for $n=6$, while with order 4, there are 1351 natural points for $n=5$ and $4104$ natural points for $n=6$ (see the ancillary \textit{Macaulay2} file \href{www.paololella.it/software/reducibility-Hilbert-schemes.m2}{\tt reducibility-Hilbert-schemes.m2}).

Our examples suggest that the number of elementary components in a given Hilbert scheme $\hilbert{d}{\mathbb{A}^n}$ might be arbitrarily large. To give an idea, this proves \Cref{thm: intro tantecomp} from the introduction.

\begin{theorem}\label{thm:many components}
The Hilbert scheme $\hilbert{34}{\mathbb{A}^6}$ has at least 12 generically reduced elementary components.
\end{theorem}
\begin{proof}
In $\hilbert{34}{\mathbb{A}^6}$, we find 12 Hilbert strata whose generic 2-step ideal has trivial negative tangents:
\begin{itemize}
    \item three generically reduced elementary components describing algebras of embedding dimension 4 with Hilbert functions
    \[
    \bh=(1,4,10,12,7),\qquad \bh=(1,4,10,13,6),\qquad \bh=(1,4,10,14,5);
    \]
     \item two generically reduced elementary components describing algebras of embedding dimension 5 with Hilbert functions
    \[
    \bh=(1,5,13,15),\qquad \bh=(1,5,14,14);
    \] 
    \item seven generically reduced elementary components describing algebras of embedding dimension 6 with Hilbert functions
    \begin{align*}
    &\bh=(1,6,14,13), &&\bh=(1,6,15,12), &&\bh=(1,6,16,11), &&\bh=(1,6,17,10),\\[-4pt]
    & \bh=(1,6,18,9), &&\bh=(1,6,19,8), &&\bh=(1,6,20,7). && \hfill\qedhere
    \end{align*}
\end{itemize}
\end{proof}

\begin{remark}
Among the elementary components given by 1-step ideals in $\Hilb^\bullet{\mathbb{A}^5}$, Shafarevich's formula \cite{Sha90} provides for two cases: $\bh=(1,5,3)$ and $\bh=(1,5,4)$. The first Hilbert function gives in fact an elementary component, see also \cite{ERMANVELASCO}, while the  Hilbert stratum corresponding to the function $\bh=(1,5,4)$ is contained in a composite component. Thus, Shafarevich's formula is incorrect, but we point out that all the preliminary lemmas in \cite{Sha90} require the embedding dimension to be different from 5 and treat the 5-dimensional case separately.

A new example of elementary component is given by the Hilbert function $\bh=(1,5,5)$. In \Cref{fig:newA5}, the corresponding natural point is marked with a blue star because $\Delta_{5,1,2} = 0$. In fact, the second branch (not drawn in the picture) of the hyperbola $\Delta_{5,1,2} = 0$ is tangent to line $\oh_3 = 35$ at the point $(10,35)$.

For $n=6$, Shafarevich's formula provides three cases. We find other 5 families of 1-step ideals giving a generically reduced elementary component in $\Hilb^\bullet{\mathbb{A}^6}$.
\end{remark}

\begin{figure}[!ht]
    \centering



    \caption{Components in $\Hilb^\bullet \mathbb{A}^6$ coming from families 2-step ideals of order 2 (first part). In the table, we describe those covered by our construction.\label{fig:newA6-1}}
    
\end{figure}

\begin{figure}[!ht]
    \centering


    \caption{Components in $\Hilb^\bullet \mathbb{A}^6$ coming from families of 2-step ideals of order 2 (second part). In the table, we describe those covered by our construction.\label{fig:newA6-2}}
    \label{fig:newA62}
\end{figure}

\begin{figure}[!ht]
    \centering


    \caption{Components in $\Hilb^\bullet \mathbb{A}^6$ coming from families of 2-step ideals of order 2 (last part). In the table, we describe those covered by our construction.\label{fig:newA6-3}}
\end{figure}

\subsection{Further developments}
We conclude with three questions that naturally emerge from this paper and may represent future research directions.
\begin{itemize}
    \item[(Q1)] What about 2-step Hilbert functions with {\em lots} of syzygies, i.e.~in the range $\oh_{k+1}\leqslant (n-1)\oh_k$? Is it possible to find some structure theorem also in this range?
    \item[(Q2)] Is there a generically reduced elementary component for every 2-step Hilbert function in the potential TNT area?
    \item[(Q3)] What about 3-step Hilbert functions and ideals?
\end{itemize}

Related to question (Q3), we recall that the potential TNT area for 2-step ideals in 3 variables is empty. Thus, the understanding of more complicated ideals seems to be inevitable to tackle the problem of the irreducibility of $\hilbert{d}{\mathbb{A}^3}$.

Related to question (Q1) and (Q2), we point out that some results in the unexplored area can be obtained via slight modifications of known results (as shown in the next example) but it is hard to expect to fill the potential TNT area in this way.

\begin{example}\label{ex:final}
In a previous paper \cite{UPDATES}, we proved that the point defined by the ideal
\[
(x_1x_2x_3-x_4x_5x_6)+(x_1,x_6)^2 +(x_2,x_5)^2+(x_3,x_4)^2 \subset\BC[x_1,\ldots,x_6]
\]
lies on a generically reduced elementary component of $\Hilb^{26}\BA^6$. Its  Hilbert function is $(1,6,12,7)$ and the associated natural point $(9,49) \in \mathcal{D}_{\mathbb{N}}$ lies in the few syzygies area but it is not covered by the main results described in \Cref{subsec:2-stepnosyz,subsec:2-stepfew} (the syzygy matrix is not generic among the matrices of the same shape). 

Adding a sufficiently general cubic as in the following ideal
\[
(x_1x_2x_3-x_4x_5x_6,x_4x_2x_3+x_1x_5x_3+x_1x_2x_6)+(x_1,x_6)^2 +(x_2,x_5)^2+(x_3,x_4)^2\subset\BC[x_1,\ldots,x_6],
\]
we obtain a 2-step with Hilbert function $(1,6,12,6)$, few syzygies and trivial negative tangents. Hence, it identifies a generically reduced elementary component of $\Hilb^{25}\BA^6$. This case is denoted by a red pentagon in \Cref{fig:newA6-1}.
\end{example}

\appendix

\section{Figure legend}\label{legenda} 

In Figures \ref{fig:newA3k6}, \ref{fig:newA3k7-8}, \ref{fig:newA4k2}, \ref{fig:newA4k3},\ref{fig:newA5}, \ref{fig:newA6-1}, \ref{fig:newA6-2} and \ref{fig:newA6-3} we draw the subset in $\mathbb{R}^2$ containing the pairs $(\oh_k,\oh_{k+1})$ defining 2-step Hilbert functions of order $k$ with the usual convention (increasing values of $\oh_k$ moving to the right and increasing values of $\oh_{k+1}$ moving upwards).

 The grey areas correspond to 2-step Hilbert functions considered in \Cref{thm:homoloco} and \Cref{thm: homo loco syz}:
 \begin{itemize}
     \item[\raisebox{-3pt}{\tikz{\draw [black!10,fill=black!10] (0,0) rectangle (0.45,0.45);}}] 2-step Hilbert functions with \emph{no} linear syzygies 
     \[\sss_\bh \geqslant 0\quad \Leftrightarrow\quad \oh_{k+1} \geqslant n\oh_k;\]
     \item[\raisebox{-3pt}{\tikz{\draw [black!50,fill=black!50] (0,0) rectangle (0.45,0.45);}}] 2-step Hilbert functions with \emph{very few} linear syzygies 
     \[0 < -\sss_\bh \leqslant \tfrac{1}{n}\oh_k\quad \Leftrightarrow\quad \left(n-\tfrac{1}{n}\right)\oh_k \leqslant \oh_{k+1} < n\oh_k;\]
     \item[\raisebox{-3pt}{\tikz{\draw [black!30,fill=black!30] (0,0) rectangle (0.45,0.45);}}] 2-step Hilbert functions with \emph{few} linear syzygies 
     \[ \tfrac{1}{n}\oh_k < -\sss_\bh < \oh_k\quad \Leftrightarrow\quad (n-1)\oh_k < \oh_{k+1} < \left(n-\tfrac{1}{n}\right)\oh_k.\]
 \end{itemize}
 
The coloured areas describe the sign of functions $\Delta_{n,1,k}$ and $\Theta_{n,k,0}$:
\begin{center}
    \raisebox{-4pt}{\tikz{\fill [pattern={Lines[angle=45,distance=2pt,line width=0.075pt]},opacity=0.3,pattern color=blue] (0,0) rectangle (0.45,0.45);\draw[NavyBlue,opacity=0.5,thick] (0,0) rectangle (0.45,0.45);}}\quad $\Delta_{n,1,k} \geqslant 0$\qquad\qquad
     \raisebox{-4pt}{\tikz{\fill [pattern={Lines[angle=45,distance=2pt,line width=0.075pt]},opacity=0.3,pattern color=red] (0,0) rectangle (0.45,0.45);}}\quad $\Delta_{n,1,k} < 0$\qquad\qquad
     \raisebox{-4pt}{\tikz{\fill [fill=yellow,opacity=0.4] (0,0) rectangle (0.45,0.45);}}\quad $\Theta_{n,k,0} \leqslant 0$.
 \end{center}
Notice that the function $\Delta_{n,1,k}$ has been defined on the smaller subset $\mathcal{D}\subset \mathbb{R}^2$ but we show its sign on the whole drawn domain. In fact, by direct computation, we find lots of natural points $(\oh_k,\oh_{k+1})$ outside $\mathcal{D}_\BN$ where the dimension of the Hilbert stratum $H_{\bh}^n$ agrees with the expected dimension in $\Delta_{n,1,k}$. The yellow area (if not empty) contains natural points that certainly belong to the potential TNT area $\mathscr{T}^n_k$. There might also be other natural points in the potential TNT area, but we do not display them as they require $\beta_{2,k+2} > 0$ and are not covered by our main results.

The meaning of the symbols denoting natural points is the following.
\begin{itemize}
    \item[\raisebox{-1pt}{\tikz{\node at (0,0)[draw=NavyBlue,very thick,fill=NavyBlue,diamond,inner sep=2pt] {};}}] The navy blue diamond denotes a pair $(\oh_k,\oh_{k+1})$ corresponding to a Hilbert stratum $H_\bh^n$ with $\Delta_{n,1,k}\geqslant 0$, thus certifying the reducibility of $\hilbert{\vert\bh\vert}{\mathbb{A}^n}$. However, the generic element in $\mathscr{H}_\bh^n$ has not trivial negative tangents so $H_\bh^n$ might not describe a full irreducible component of $\hilbert{\vert\bh\vert}{\mathbb{A}^n}$.
    \item[\raisebox{-1pt}{\tikz{\node at (0,0) [draw=blue,very thick,fill=blue,star,star point ratio=2.2,inner sep=1.2pt] {};}}] The blue star denotes a pair $(\oh_k,\oh_{k+1})$ corresponding to a Hilbert stratum $H_\bh^n$ with $\Delta_{n,1,k}\geqslant 0$ such that the generic element in $\mathscr{H}_\bh^n$ has trivial negative tangents, thus identifying a generically reduced elementary component of $\hilbert{\vert\bh\vert}{\mathbb{A}^n}$.
    \item[\raisebox{-1pt}{\tikz{\node at (0,0) [draw=red,very thick,fill=red,star,star point ratio=2.2,inner sep=1.2pt] {};}}] The red star denotes a pair $(\oh_k,\oh_{k+1})$ corresponding to a Hilbert stratum $H_\bh^n$ with $\Delta_{n,1,k}< 0$  such that the generic element in $\mathscr{H}_\bh^n$ has trivial negative tangents, thus identifying a generically reduced elementary $\Delta$-negative component of $\hilbert{\vert\bh\vert}{\mathbb{A}^n}$.
    \item[\raisebox{-1pt}{\tikz{\node at (0,0)[draw=NavyBlue,very thick,fill=white,diamond,inner sep=2pt] {}; \node at (0.4,0) [draw=blue,line width=1,fill=white,star,star point ratio=2.2,inner sep=1.2pt] {}; \node at (0.8,0) [draw=red,line width=1,fill=white,star,star point ratio=2.2,inner sep=1.2pt] {};}}] Empty symbols denote pairs $(\oh_k,\oh_{k+1})$ corresponding to a Hilbert strata covered by the main results of the paper but already known in literature.
    \item[\raisebox{-1pt}{\tikz{\node at (0,0) [draw=red,very thick,fill=white,circle,inner sep=2pt] {}; \node at (0.35,0)[draw=blue,very thick,fill=white,circle,inner sep=2pt] {};}}] Empty circles denote pairs $(\oh_k,\oh_{k+1})$ corresponding to a Hilbert stratum $H_\bh^n$ identifying a generically reduced elementary component known in literature but not covered by the main results of the paper. The color has the same meaning as above.
\end{itemize}


\ifx\undefined\bysame
\newcommand{\bysame}{\leavevmode\hbox to3em{\hrulefill}\,}
\fi

\end{document}